\documentclass[12pt]{amsart}
\usepackage{amsmath}
\usepackage{amssymb}
\usepackage{latexsym}

\newcommand{\Zint}{\mathbb {Z}}    
     
\newcommand{\Rea}{\mathbb {R}}      
\newcommand{\Cplx}{\mathbb {C}}     
\newcommand{\halmos}{\rule{5pt}{5pt}}

\numberwithin{equation}{section}
\newtheorem{df}{\bf Definition}
\newtheorem{prop}{\bf Proposition}[section]
\newtheorem{thm}[prop]{\bf Theorem}

\newtheorem{cor}[prop]{\bf Corollary}

\newtheorem{exa}{\bf Example}

\begin{document}

\title[Integral transformation of Heun's equation]
{Integral transformation of Heun's equation and some applications}
\author{Kouichi Takemura}
\address{Department of Mathematics, Faculty of Science and Technology, Chuo University, 1-13-27 Kasuga, Bunkyo-ku Tokyo 112-8551, Japan.}
\email{takemura@math.chuo-u.ac.jp}
\begin{abstract}
It is known that the Fuchsian differential equation which produces the sixth Painlev\'e equation corresponds to the Fuchsian differential equation with different parameters via Euler's integral transformation, and Heun's equation also corresponds to Heun's equation with different parameters, again via Euler's integral transformation.
In this paper we study the correspondences in detail.
After investigating correspondences with respect to monodromy, it is demonstrated that the existence of polynomial-type solutions corresponds to 
apparency of a singularity.
For the elliptical representation of Heun's equation, correspondence with respect to monodromy implies isospectral symmetry.
We apply the symmetry to finite-gap potentials and express the monodromy of Heun's equation with parameters which have not yet been studied.
\end{abstract}

\maketitle

\section{Introduction}

The Gauss hypergeometric differential equation
\begin{equation}
z(1-z) \frac{d^2y}{dz^2} + \left( \gamma - (\alpha + \beta +1)z \right) \frac{dy}{dz} -\alpha \beta  y=0.
\label{Gauss}
\end{equation}
is very famous both in physics and especially so in mathematics; it is a canonical form of the second-order Fuchsian differential equation with three singularities on the Riemann sphere $\Cplx \cup \{ \infty \}$.
There are several important generalizations of Gauss hypergeometric differential equation.
We now treat two examples, Heun's equation and the sixth Painlev\'e equation.

Heun's differential equation (or Heun's equation) is a canonical form of a second-order Fuchsian equation with four singularities, which is given by
\begin{equation}
\frac{d^2y}{dz^2} + \left( \frac{\epsilon _0}{z}+\frac{\epsilon _1}{z-1}+\frac{\epsilon _t}{z-t}\right) \frac{dy}{dz} +\frac{\alpha \beta z -q}{z(z-1)(z-t)} y=0,
\label{Heun}
\end{equation}
with the condition 
$\epsilon _0 +\epsilon _1+\epsilon _t =\alpha +\beta +1$ (see \cite{Ron}).
The exponents at $z=0$ (resp. $z=1$, $z=t$, $z=\infty$) are $0$ and $1-\epsilon _0 $ (resp. $0$ and $1-\epsilon _1$, $0$ and $1-\epsilon _t $, $\alpha $ and $\beta$).

The sixth Painlev\'e system is a system of non-linear ordinary differential equations defined by
\begin{equation}
\frac{d\lambda }{dt} =\frac{\partial H}{\partial \mu}, \quad \quad
\frac{d\mu }{dt} =-\frac{\partial H}{\partial \lambda} ,
\label{eq:Psys}
\end{equation}
with the Hamiltonian 
\begin{align}
H = & \frac{1}{t(t-1)} \left\{ \lambda (\lambda -1) (\lambda -t) \mu^2 -\left\{ \theta _0 (\lambda -1) (\lambda -t)+\theta _1 \lambda (\lambda -t) \right. \right. \label{eq:P6Ham} \\
& \left. \left. \quad \quad \quad \quad \quad \quad \quad \quad  +(\theta _t -1) \lambda (\lambda -1) \right\} \mu +\kappa _1(\kappa _2 +1) (\lambda -t)\right\} .\nonumber
\end{align}
By eliminating $\mu $ in Eq.(\ref{eq:Psys}), we obtain the sixth Painlev\'e equation for $\lambda $ which is a non-linear ordinary differential equation of order two in the independent variable $t$.
It is known that the sixth Painlev\'e system is related to monodromy preserving deformations of certain Fuchsian differential equations.
Let $\lambda \not \in \{0,1,t,\infty \}$ and $D_{y_1}(\theta _0, \theta _1, \theta _t, \theta _{\infty}; \lambda ,\mu )$ be the second-order linear differential equation given by
\begin{align}
& \frac{d^2y_1(z)}{dz^2} + \left( \frac{1-\theta _0}{z}+\frac{1-\theta _1}{z-1}+\frac{1-\theta _t}{z-t}-\frac{1}{z-\lambda}  \right)  \frac{dy_1(z)}{dz} \label{eq:linP6} \\
&  \quad \quad + \left( \frac{\kappa _1(\kappa _2 +1)}{z(z-1)}+\frac{\lambda (\lambda -1)\mu}{z(z-1)(z-\lambda)}-\frac{t (t -1)H}{z(z-1)(z-t)}  \right) y_1(z)=0, \nonumber \\
& \kappa _1= (\theta _{\infty } -\theta _0 -\theta _1 -\theta _t)/2, \quad \kappa _2= -(\theta _{\infty } +\theta _0 +\theta _1 +\theta _t)/2, \nonumber
\end{align}
where $H$ is given as in Eq.(\ref{eq:P6Ham}). Then Eq.(\ref{eq:linP6}) is a Fuchsian differential equation with five singularities $\{ 0,1,t, \lambda ,\infty \}$ on the Riemann sphere.
The exponents at $z=p$ $(p \in \{0,1,t \})$ (resp. $z=\lambda $, $z=\infty$) are $0$ and $\theta _p$ (resp. $0$ and $2$, $\kappa_1 $ and $\kappa _2 +1$),
and it follows from the condition given by Eq.(\ref{eq:P6Ham}) that the singularity $z =\lambda $ is apparent.
The sixth Painlev\'e system (Eq.(\ref{eq:Psys})) is derived from a monodromy preserving deformation of Eq.(\ref{eq:linP6}) (for details, see \cite{IKSY}), and the function $y_1(z)$ is obtained from a first order $2\times 2$ Fuchsian differential system with four singularities $\{0,1,t,\infty \}$, denoted by $D_Y(\theta _0, \theta _1, \theta _t, \theta _{\infty}; \lambda ,\mu ;k )$ in \cite{TakI,TakM}.

Heun's equation and the Fuchsian differential equation $D_{y_1}(\theta _0, \theta _1, \theta _t, \theta _{\infty}; \lambda ,\mu )$ admit integral transformations.
We fix a base point $o $ for the integrals in the complex plane $\Cplx$ appropriately.
Let $f(z)$ be a holomorphic function locally defined around $z=o$ and $f^{\gamma }(z)$ be the function analytically continued along a cycle $\gamma $ whose base point is $o$.
Define
\begin{equation}
\langle \gamma , f\rangle _{\kappa } = \int _{\gamma } f(w) (z-w)^{\kappa } dw .
\label{eq:Etrans}
\end{equation}
This is called Euler's integral transformation (or an Euler transformation).
Let $p $ be an singularity of the function $f(w)$ in the Riemann sphere $\Cplx \cup \{\infty \}$, $\gamma _p$ be a cycle on the Riemann sphere with variable $w$ which starts from $w=o$, goes around $w=p$ in a counter-clockwise direction and ends at $w=o$, and $[\gamma _z ,\gamma _p] = \gamma _z \gamma _p \gamma _z ^{-1} \gamma _p ^{-1}$ be the Pochhammer contour.
The following proposition was obtained by Novikov \cite{Nov}, independently by Kazakov and Slavyanov \cite{KS1}, and also derived by considering an explicit form of a middle convolution of a $2 \times 2$ Fuchsian differential system \cite{TakM}: 
\begin{prop} $($\cite{Nov,KS1,TakM}$)$ \label{prop:Nov}
If $y_1(z)$ is a solution of $D_{y_1}(\theta _0, \theta _1, \theta _t,\theta _{\infty}; \lambda ,\mu )$,
then the function 
\begin{align}
\tilde{y} (z)  = \langle [\gamma _{z} ,\gamma _p] , y_1 \rangle _{\kappa _2 -1} = \int _{[\gamma _{z} ,\gamma _p]} y_1(w) (z-w)^{\kappa _2 -1} dw,
\label{eq:inttransDy1}
\end{align}
satisfies $D_{y_1}(\tilde{\theta }_0, \tilde{\theta }_1, \tilde{\theta }_t, \tilde{\theta }_{\infty} ; \tilde{\lambda } ,\tilde{\mu } )$ for $p \in \{ 0,1,t,\infty \}$, where
\begin{align}
& \kappa _2= -(\theta _{\infty } +\theta _0 +\theta _1 +\theta _t)/2, \quad \tilde{\theta }_p = \kappa_2 +\theta _p \quad (p=0,1,t,\infty ), \label{eq:tlamtmu} \\
& \tilde{\lambda} =\lambda -\frac{\kappa _2}{\mu -\frac{\theta _0}{\lambda }-\frac{\theta _1}{\lambda -1}-\frac{\theta _t}{\lambda -t}} ,\quad  \tilde{\mu} =\frac{\kappa _2 +\theta _0}{\tilde{\lambda }} +\frac{\kappa _2 +\theta _1}{\tilde{\lambda }-1} +\frac{\kappa _2 +\theta _t}{\tilde{\lambda }-t} +\frac{\kappa _2 }{\lambda - \tilde{\lambda }} .  \nonumber 
\end{align}
\end{prop}
Kazakov and Slavyanov established an integral transformation for solutions of Heun's equation in \cite{KS}, and it was also obtained by taking suitable limits in Proposition \ref{prop:Nov}, which was discussed in \cite{TakM} by considering the relationship with the space of initial conditions of the sixth Painlev\'e equation.
\begin{prop} $($\cite{KS,TakM}$)$ \label{prop:Heunint}
Let $\alpha $, $\beta $, $\epsilon _0 $, $\epsilon _1 $, $\epsilon _t $ be the parameters in Heun's equation (\ref{Heun02}).
Set $\eta = \alpha \mbox{ or } \beta $ and 
\begin{align}
& \epsilon _0 '=\epsilon _0 -\eta+1 , \; \epsilon ' _1  =\epsilon _1-\eta +1, \; \epsilon _t '=\epsilon _t -\eta +1 , \label{eq:mualbe} \\
& \{ \alpha ' , \beta ' \} = \{ 2-\eta , \alpha +\beta -2\eta +1 \} , \; q'=q+(1-\eta )(\epsilon _t +\epsilon _1t+(\epsilon _0 -\eta ) (t+1)). \nonumber
\end{align}
Let $v(w)$ be a solution of 
\begin{equation}
\frac{d^2v}{dw^2} + \left( \frac{\epsilon _0 '}{w}+\frac{\epsilon _1'}{w-1}+\frac{\epsilon _t '}{w-t}\right) \frac{dv}{dw} +\frac{\alpha ' \beta ' w -q'}{w(w-1)(w-t)} v=0.
\label{Heun01}
\end{equation}
Then the function 
\begin{align}
& y(z)=\langle [\gamma _{z} ,\gamma _p] , v \rangle _{-\eta} = \int _{[\gamma _z ,\gamma _p]} v(w) (z-w)^{-\eta } dw 
\label{eq:inttransHeun}
\end{align}
is a solution of 
\begin{equation}
\frac{d^2y}{dz^2} + \left( \frac{\epsilon _0}{z}+\frac{\epsilon _1}{z-1}+\frac{\epsilon _t}{z-t}\right) \frac{dy}{dz} +\frac{\alpha \beta z -q}{z(z-1)(z-t)} y=0,
\label{Heun02}
\end{equation}
for $p \in \{ 0,1,t,\infty \}$.
\end{prop}
Note that the inverse correspondence of the parameters is given by setting $\eta = \alpha \mbox{ or } \beta $ and 
\begin{align}
& \eta =2-\eta ', \;  \epsilon _0 =\epsilon _0 ' -\eta '+1, \; \epsilon _1=\epsilon _1'-\eta '+1, \; \epsilon _t =\epsilon _t '-\eta '+1 , \label{eq:mualbe18} \\
& \{ \alpha , \beta \} = \{ 2-\eta ' , \alpha '+\beta '-2\eta '+1 \} , \; q=q'+(1-\eta ')(\epsilon _t '+\epsilon _1't+(\epsilon _0 '-\eta ') (t+1)). \nonumber
\end{align}

In this paper, Euler's integral transformations given by Eqs.(\ref{eq:inttransDy1}), (\ref{eq:inttransHeun}) are considered.
If we have a solution of the differential equation $D_{y_1}(\theta _0, \theta _1, \theta _t,\theta _{\infty}; \lambda ,\mu )$ (resp. Heun's differential equation (\ref{Heun01})), then we may study the solution of $D_{y_1}(\tilde{\theta }_0, \tilde{\theta }_1, \tilde{\theta }_t, \tilde{\theta }_{\infty} ; \tilde{\lambda } ,\tilde{\mu } )$ (resp. Eq.(\ref{Heun02})) by means of Euler's integral transformations in Eq.(\ref{eq:inttransDy1}) (resp. Eq.(\ref{eq:inttransHeun})).
We apply this strategy for the case where the differential equation $D_{y_1}(\theta _0, \theta _1, \theta _t,\theta _{\infty}; \lambda ,\mu )$ (resp. Eq.(\ref{Heun01})) has a polynomial-type solution.
Then it is shown that one of the singularities $\{ 0,1,t,\infty \}$ of the differential equation $D_{y_1}(\tilde{\theta }_0, \tilde{\theta }_1, \tilde{\theta }_t, \tilde{\theta }_{\infty} ; \tilde{\lambda } ,\tilde{\mu } )$ (resp. Eq.(\ref{Heun02})) turns out to be apparent, and the inverse statement also holds (see Theorems \ref{thm:nonlog-polyn} and \ref{thm:nonlog-polynHeun}).
As a by-product, we have integral representations of solutions of Heun's equation for which one of the singularities $\{ 0,1,t,\infty \}$ is apparent (see Theorems \ref{thm:exprnonlogsol} and \ref{thm:exprnonlogsol2}). 
We also investigate properties of monodromy by means of integral transformations, which are used for the study of solutions.

It is known that Heun's equation has an elliptical representation.
Let $\wp (x)$ be the Weierstrass doubly periodic function with periods $(2\omega _1, 2\omega _3)$, $\omega _0 (=0)$, $\omega _1$, $\omega _2(=-\omega _1 -\omega _3)$, $\omega _3$ be the half-periods and $e_i=\wp (\omega _i)$ $(i=1,2,3)$.
Heun's equation (\ref{Heun}) is transformed to
\begin{equation}
\left(-\frac{d^2}{dx^2} + \sum_{i=0}^3 l_i(l_i+1)\wp (x+\omega _i)-E\right) f(x)=0,
\label{InoEF0}
\end{equation}
by setting $z=(\wp (x) -e_1)/(e_2-e_1)$, $t=(e_3-e_1)/(e_2-e_1)$.
For details see section \ref{sec:ell}.   
Then the integral transformation of Eq.(\ref{eq:inttransHeun}) provides a correspondence of Eq.(\ref{InoEF0}) with a different parameter described in Proposition \ref{prop:ellipinttras}.
For the elliptical representation, the invariance of monodromy by the integral transformation with respect to the shift of a period is remarkable.  
For details see Theorem \ref{thm:pp0}.
We also obtain correspondences of solutions expressed by quasi-solvability (existence of a polynomial-type solution) and apparency of one of the singularities $\{0,\omega _1, \omega _2, \omega _3 \}$.
We apply the integral transformation for the case where Heun's equation has the finite-gap property, i.e. the case where $l_0,l_1,l_2,l_3 \in \Zint$.
For the case $l_0,l_1,l_2,l_3 \in \Zint$ we can calculate the monodomy in principle for all $E$ by means of hyperelliptic integrals \cite{Tak3} and by the Hermite-Krichever Ansatz \cite{Tak4}. 
By applying monodromy invariance, we can calculate the monodromy of Heun's equation for the case $l_0, l_1, l_2, l_3 \in \Zint +1/2$ and $l_0+l_1+l_2+l_3 \in 2\Zint +1$, which have not been studied previously.

This paper is organized as follows:
In section \ref{sec:MonInttr}, we investigate the transformation of the monodromy induced by Euler's integral  transformation. 
In section \ref{sec:solmon}, we obtain some properties of solutions and monodromy of the Fuchsian differential equations $D_{y_1}(\theta _0, \theta _1, \theta _t, \theta _{\infty} ; \lambda ,\mu )$, $D_{y_1}(\tilde{\theta }_0, \tilde{\theta }_1, \tilde{\theta }_t, \tilde{\theta }_{\infty} ; \tilde{\lambda } ,\tilde{\mu } )$ and Heun's equations (Eqs.(\ref{Heun01}), (\ref{Heun02})).
In section \ref{sec:polnonlog}, we have correspondences of polynomial-type solutions and solutions such that one of the singularities is apparent.
In section \ref{sec:QSnonlogHeun}, we obtain integral representations of solutions of Heun's equation which have apparent singularities by using polynomial-type solutions. 
In section \ref{sec:ell}, we translate the results to the elliptical representation of Heun's equation.
In section \ref{sec:HeunFGIT}, we review results on finite-gap potentials and calculate the monodromy of the elliptical representation of Heun's equation for the case $l_0, l_1, l_2, l_3 \in \Zint +1/2$ and $l_0+l_1+l_2+l_3 \in 2\Zint +1$.
In the appendix we provide the technical details.

\section{Monodromy and integral transformation} \label{sec:MonInttr}
In this section we investigate the transformation of the monodromy induced by Euler's integral transformation given by Eq.(\ref{eq:Etrans}). 

We review some facts about cycles in order to discuss the monodromy.
Let $a, b \in \Cplx \cup \{ \infty \}$ $(a\neq b)$
and $p _{ab}$ be a path linking $a$ and $b$.
We put the base point $o$ of integrals in Eq.(\ref{eq:Etrans}) on the left side of the path $p _{ab}$.
Let $z$ be a point on $p _{ab}$.
We consider deformations of the cycles $\gamma _a$, $\gamma _b$ and $\gamma _z$ in the $w$-plane as the point $z$ turns around the singularity $a$ or $b$ anti-clockwise.
As the point $z$ turns around the singularity $w=a$ anti-clockwise, the cycle $\gamma _a$ is deformed to $\gamma _a \gamma _z \gamma _a  \gamma _z ^{-1} \gamma_ a ^{-1}$, the cycle $\gamma _z$ is deformed to $\gamma _a \gamma _z \gamma _a ^{-1}$ and the cycle $\gamma _b$ is not deformed (see Figure 1).
As the point $z$ turns around the singularity $b$ anti-clockwise, the cycle $\gamma _b$ is deformed to $\gamma _z \gamma _b \gamma _z  ^{-1}$, the cycle $\gamma _z$ is deformed to $\gamma _z \gamma _b \gamma _z \gamma _b ^{-1}\gamma _z ^{-1}$ and the cycle $\gamma _a$ is not deformed (see also Figure 1).
\begin{center}
\begin{picture}(400,210)(0,0)
\qbezier(10,140)(50,170)(90,140)
\put(10,140){\circle*{3}}
\put(90,140){\circle*{3}}
\put(40,154){\circle*{3}}
\put(45,200){\circle*{3}}
\put(40,147){$z$}
\put(10,133){$a$}
\put(92,133){$b$}
\put(48,200){$o$}
\put(60,156){$p_{ab}$}
\put(105,110){\line(0,1){100}}
\put(105,110){\line(-1,0){110}}

\put(165,200){\circle*{3}}
\put(168,200){$o$}
\put(150,150){\circle*{3}}
\put(150,153){$a$}
\put(168,150){\circle*{3}}
\put(170,146){$z$}
\qbezier(165,200)(140,170)(140,150)
\put(140,150){\vector(0,-1){1}}
\put(150,150){\oval(20,20)[b]}
\qbezier(160,150)(160,170)(165,200)
\put(160,150){\vector(0,1){1}}
\put(150,128){$\gamma _a $}

\put(180,150){$\Rightarrow$}
\put(265,200){\circle*{3}}
\put(268,200){$o$}
\put(250,150){\circle*{3}}
\put(250,153){$a$}
\put(232,150){\circle*{3}}
\put(232,143){$z$}
\qbezier(265,200)(210,170)(210,150)
\put(210,150){\vector(0,-1){1}}
\qbezier(210,150)(210,115)(240,115)
\qbezier(240,115)(260,115)(260,150)
\put(250,150){\oval(20,20)[t]}
\put(232,150){\oval(16,42)[b]}
\put(224,150){\vector(0,1){1}}
\qbezier(224,150)(224,170)(265,200)
\qbezier[20](268,150)(268,168)(250,168)
\qbezier[20](232,150)(232,168)(250,168)
\put(250,168){\vector(-1,0){1}}

\put(285,150){$\Rightarrow$}
\put(375,200){\circle*{3}}
\put(378,200){$o$}
\put(360,150){\circle*{3}}
\put(360,153){$a$}
\put(378,150){\circle*{3}}
\put(380,146){$z$}
\qbezier(375,200)(320,170)(320,150)
\put(320,150){\vector(0,-1){1}}
\qbezier(320,150)(320,115)(360,115)
\qbezier(360,115)(395,115)(395,150)
\qbezier(395,150)(395,170)(375,195)
\qbezier(375,195)(350,160)(350,150)
\put(360,150){\oval(20,20)[b]}
\qbezier(386,150)(386,170)(375,190)
\qbezier(375,190)(370,160)(370,150)
\put(360,150){\oval(52,52)[b]}
\put(334,150){\vector(0,1){1}}
\qbezier(334,150)(334,170)(375,200)
\qbezier[20](378,150)(378,168)(360,168)
\qbezier[20](378,150)(378,132)(360,132)
\qbezier[20](342,150)(342,168)(360,168)
\qbezier[20](342,150)(342,132)(360,132)
\put(373,136){\vector(1,1){1}}
\put(325,103){$\gamma _a \gamma _z \gamma _a \gamma _z ^{-1} \gamma_ a ^{-1}$}

\put(95,90){\circle*{3}}
\put(98,90){$o$}
\put(80,50){\circle*{3}}
\put(80,53){$a$}
\put(98,50){\circle*{3}}
\put(100,46){$z$}
\qbezier(95,90)(54,70)(54,50)
\put(54,50){\vector(0,-1){1}}
\qbezier(95,90)(70,60)(70,50)
\put(80,50){\oval(20,20)[b]}
\qbezier(106,50)(106,70)(95,80)
\qbezier(95,80)(90,60)(90,50)
\put(80,50){\oval(52,52)[b]}
\put(70,50){\vector(0,1){1}}
\qbezier[20](98,50)(98,68)(80,68)
\qbezier[20](98,50)(98,32)(80,32)
\qbezier[20](62,50)(62,68)(80,68)
\qbezier[20](62,50)(62,32)(80,32)
\put(93,36){\vector(1,1){1}}
\put(55,3){$\gamma _a \gamma _z  \gamma_ a ^{-1}$}

\put(195,90){\circle*{3}}
\put(198,90){$o$}
\put(202,50){\circle*{3}}
\put(195,53){$z$}
\put(218,50){\circle*{3}}
\put(218,40){$b$}
\qbezier(195,90)(174,70)(174,50)
\put(174,50){\vector(0,-1){1}}
\qbezier(195,90)(190,60)(190,50)
\put(200,50){\oval(20,20)[b]}
\qbezier(226,50)(226,70)(210,80)
\qbezier(210,80)(210,60)(210,50)
\put(200,50){\oval(52,52)[b]}
\put(190,50){\vector(0,1){1}}
\qbezier(174,50)(174,70)(195,90)
\qbezier[28](238,50)(238,88)(220,88)
\qbezier[20](238,50)(238,32)(220,32)
\qbezier[28](202,50)(202,88)(220,88)
\qbezier[20](202,50)(202,32)(220,32)
\put(233,36){\vector(1,1){1}}
\put(185,3){$\gamma _z \gamma _b \gamma _z ^{-1}$}

\put(345,90){\circle*{3}}
\put(348,90){$o$}
\put(332,50){\circle*{3}}
\put(332,53){$z$}
\put(348,50){\circle*{3}}
\put(350,46){$b$}
\qbezier(345,90)(300,70)(300,50)
\put(300,50){\vector(0,-1){1}}
\qbezier(300,50)(300,15)(330,15)
\qbezier(330,15)(375,15)(375,50)
\qbezier(375,50)(375,75)(345,85)
\qbezier(345,85)(320,60)(320,50)
\put(330,50){\oval(20,20)[b]}
\qbezier(356,50)(356,70)(345,75)
\qbezier(345,75)(340,60)(340,50)
\put(333,50){\oval(46,52)[b]}
\put(310,50){\vector(0,1){1}}
\qbezier(310,50)(310,70)(345,90)
\qbezier[20](368,50)(368,78)(350,78)
\qbezier[20](368,50)(368,32)(350,32)
\qbezier[20](332,50)(332,78)(350,78)
\qbezier[20](332,50)(332,32)(350,32)
\put(363,36){\vector(1,1){1}}
\put(305,3){$\gamma _z \gamma _b \gamma _z \gamma _b ^{-1} \gamma_ z ^{-1}$}
\end{picture}
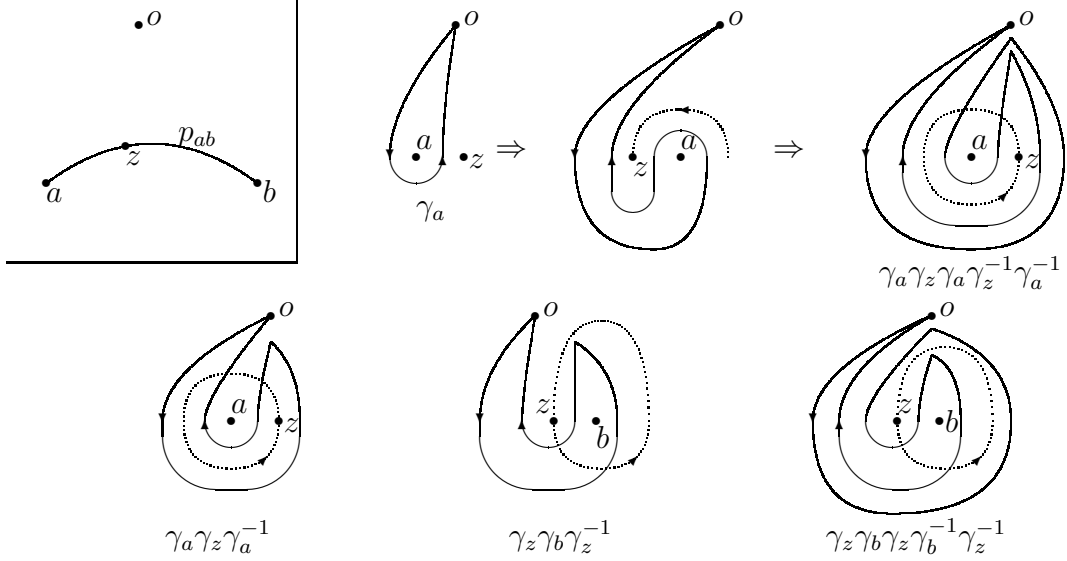
Figure 1. Deformation of the cycles.
\end{center}

The Euler's integral transformation by the Pochhammer contour admits the following expression,
\begin{align}
& \langle [\gamma _{z} ,\gamma _a] , f \rangle _{\kappa } = \int _{\gamma _{z} \gamma _a \gamma _{z} ^{-1} \gamma _a ^{-1}} f(w) (z-w)^{\kappa } dw = (e^{2\pi \sqrt{-1}\kappa }-1) \langle \gamma _a ,f \rangle _{\kappa }+ \langle \gamma _z , f- f^{\gamma _a} \rangle _{\kappa }, \label{eq:gzgidecom}
\end{align}
where we used the relations $0= \langle  \gamma _a , f \rangle + \langle  \gamma _a ^{-1} , f ^{\gamma _a} \rangle  $ and $0= \langle  \gamma _z , f \rangle +  e^{2\pi \sqrt{-1}\kappa } \langle  \gamma _z ^{-1} , f \rangle $. The formula obtained by replacing $a$ with $b$ also holds true.
Then analytic continuation of the transformed functions are calculated as follows;
\begin{prop}
\begin{align}
& \langle [\gamma _{z} ,\gamma _b] , f \rangle  ^{\gamma _a} = \langle [\gamma _{z} ,\gamma _b] , f \rangle + \langle [\gamma _{z} ,\gamma _a] , f ^{\gamma _b} \rangle - \langle [\gamma _{z} ,\gamma _a ],f  \rangle ,  \label{eq:gzgbga} \\
& \langle [\gamma _{z} ,\gamma _a] , f \rangle  ^{\gamma _a} = e^{2\pi \sqrt{-1}\kappa } \langle [\gamma _{z} ,\gamma _a] , f ^{\gamma _a} \rangle , \nonumber \\
& \langle [\gamma _{z} ,\gamma _a] , f \rangle  ^{\gamma _b} =
\langle [\gamma _{z} ,\gamma _a] , f \rangle +  e^{2\pi \sqrt{-1}\kappa } \langle [\gamma _{z} ,\gamma _b] , f ^{\gamma _a} \rangle -  e^{2\pi \sqrt{-1}\kappa } \langle [\gamma _{z} ,\gamma _b ],f  \rangle , \nonumber \\
& \langle [\gamma _{z} ,\gamma _b] , f \rangle  ^{\gamma _b} = e^{2\pi \sqrt{-1}\kappa } \langle [\gamma _{z} ,\gamma _b] , f ^{\gamma _b} \rangle . \nonumber 
\end{align}
\end{prop} 
\begin{proof}
We show the first equality.
It follows from Eq.(\ref{eq:gzgidecom}) that 
\begin{align}
& \langle [\gamma _{z} ,\gamma _b] , f \rangle  ^{\gamma _a} = \langle [\gamma _a \gamma _{z} \gamma _a ^{-1},\gamma _b] , f \rangle  = \langle \gamma _a ,f \rangle + \langle \gamma _z , f^{\gamma _a} \rangle -  e^{2\pi \sqrt{-1}\kappa } \langle \gamma _a , f \rangle  \\
& + e^{2\pi \sqrt{-1}\kappa } \langle \gamma _b ,f \rangle + e^{2\pi \sqrt{-1}\kappa } \langle \gamma _a ,f ^{\gamma _b}\rangle - \langle \gamma _z , f ^{\gamma _b \gamma _a} \rangle - \langle \gamma _a ,f^{\gamma _b} \rangle - \langle \gamma _b ,f \rangle . \nonumber 
\end{align}
By expanding the r.h.s. of the first equality, we have the same expression.
The second formula follows from
\begin{align}
& \langle [\gamma _{z} ,\gamma _a] , f \rangle  ^{\gamma _a} =  \langle \gamma _{a}\gamma _z [\gamma _{z} ,\gamma _a] (\gamma _{a}\gamma _z)^{-1}, f \rangle = e^{2\pi \sqrt{-1}\kappa } \langle [\gamma _{z} ,\gamma _a] , f ^{\gamma _a} \rangle . \nonumber
\end{align}
The other formulas are shown similarly.
\end{proof}
From now on, we assume that the function $y(w)$ is a solution of a second-order differential equation, the two points $a, b \in \Cplx $ are regular singularities of the differential equation whose exponents are $0$ and $\theta _a$, $0$ and $\theta _b$ respectively, where the case of Heun's equation and that of the differential equation $D_{y_1}(\theta _0, \theta _1, \theta _t, \theta _{\infty}; \lambda ,\mu )$ are included.
We put the base point $o$ of integrals in Eq.(\ref{eq:inttransDy1}) on the left side of the path $p _{ab}$.
Let $y^{(2)} (w)$ be a solution of the second-order differential equation such that the functions $y(w)$, $y^{(2)} (w)$ form a basis of solutions of the differential equation, and we denote the monodromy matrices around the singularity $w=a$ by
\begin{align}
(y(w)^{\gamma _a} , y^{(2)} (w) ^{\gamma _a})= (y(w) , y^{(2)}(w) )\left( 
\begin{array}{cc}
a'_{11} & a'_{12}\\
a'_{21} & a'_{22}
\end{array} \right) = (y (w) , y^{(2)} (w) )M'^{(a)}. \label{eq:monodAB} 
\end{align}
The eigenvalues of the monodromy matrix $M'^{(a)}$ in Eq.(\ref{eq:monodAB}) are $1$ and $e^{2\pi \sqrt{-1} \theta _a}$. 
If one of the exponents around the singularities $w=a$ is zero, then there exists a non-zero solution $f(w)$ that is holomorphic about $w=a$ and it follows from Eq.(\ref{eq:gzgidecom}) that the function $\langle [\gamma _{z} ,\gamma _a] , f \rangle _{\kappa } $ is zero.
Let $a'_1 y(w)+ a'_2y^{(2)}(w)$ be a non-zero holomorphic solution of the differential equation about $w=a$.
Then we have 
\begin{equation}
 a'_1 \langle [\gamma _{z} ,\gamma _a] , y \rangle _{\kappa }  + a'_2 \langle [\gamma _{z} ,\gamma _a] , y^{(2)} \rangle _{\kappa } =0. \label{eq:holp1p2}
\end{equation}
Since $a'_1 y_1(w)^{\gamma _a}+ a'_2 y^{(2)} (w)^{\gamma _a}=a'_1 y(w)+ a'_2 y^{(2)}(w)$, we have
\begin{align}
\left(
\begin{array}{cc}
a'_{11} -1& a'_{12}\\
a'_{21} & a'_{22} -1
\end{array} \right)
\left(
\begin{array}{cc}
a'_{1} \\
a'_{2}
\end{array} \right) 
=\left(
\begin{array}{cc}
0 \\
0
\end{array} \right) ,
\quad
\begin{array}{ll}
(a'_{11}-1)(a'_{22}-1)-a'_{12}a'_{21}=0,  \\
a'_{11} +a'_{22} =\mbox{tr} M'^{(a)} = 1+e^{2\pi \sqrt{-1} \theta _a}. 
\end{array}
\label{eq:pij} 
\end{align}
We use similar notations for the singularity $w=b$.
If $\langle [\gamma _{z} ,\gamma _a] , y \rangle _{\kappa } \neq 0$ and $\langle [\gamma _{z} ,\gamma _b] , y \rangle _{\kappa } \neq 0$, then $y(w)$ is not holomorphic about $w=a,b$ and $a'_2 \neq 0 \neq b'_2$.
We calculate the monodromy for the functions $\langle [\gamma _{z} ,\gamma _a] , y \rangle _{\kappa } $, $\langle [\gamma _{z} ,\gamma _b] , y \rangle _{\kappa } $ with respect to the cycles $\gamma _a$ and $\gamma _b$ for the case that $\langle [\gamma _{z} ,\gamma _a] , y \rangle _{\kappa } $, $\langle [\gamma _{z} ,\gamma _b] , y \rangle _{\kappa } $ are linearly independent.
Combining with Eqs.(\ref{eq:gzgbga}), we have
\begin{align}
& \langle [\gamma _{z} ,\gamma _a] , y  ^{\gamma _b}\rangle _{\kappa }
= \langle [\gamma _{z} ,\gamma _a] , b'_{11}y +b'_{21}y^{(2)}  \rangle _{\kappa }
= \langle [\gamma _{z} ,\gamma _a] , b'_{11}y -\frac{a'_1}{a'_2}b'_{21}y  \rangle _{\kappa } \label{eq:monodrels} \\
& \langle [\gamma _{z} ,\gamma _b] , y \rangle _{\kappa }  ^{\gamma _a} = \langle [\gamma _{z} ,\gamma _b] , y \rangle _{\kappa } + \langle [\gamma _{z} ,\gamma _a] , y ^{\gamma _b} \rangle _{\kappa } - \langle [\gamma _{z} ,\gamma _a ],y \rangle _{\kappa } ,\nonumber \\
& = \left( b'_{11} -1  -\frac{a'_1}{a'_2}b'_{21} \right) \langle [\gamma _{z} ,\gamma _a] , y \rangle _{\kappa } + \langle [\gamma _{z} ,\gamma _b] , y \rangle _{\kappa } ,\nonumber \\
& \langle [\gamma _{z} ,\gamma _a] , y\rangle _{\kappa }  ^{\gamma _a} = e^{2\pi \sqrt{-1}\kappa } \langle [\gamma _{z} ,\gamma _a] ,  a'_{11}y +a'_{21}y^{(2)} \rangle _{\kappa } \nonumber \\
& = e^{2\pi \sqrt{-1}\kappa } \langle [\gamma _{z} ,\gamma _a] ,  a'_{11}y -\frac{a'_1}{a'_2}a'_{21}y \rangle _{\kappa } = e^{2\pi \sqrt{-1}\kappa } (a'_{11}+a'_{22}-1 )\langle [\gamma _{z} ,\gamma _a] ,  y \rangle _{\kappa } ,\nonumber \\
& \langle [\gamma _{z} ,\gamma _a] , y  \rangle _{\kappa } ^{\gamma _b} =  \langle [\gamma _{z} ,\gamma _a] , y \rangle _{\kappa } + e^{2\pi \sqrt{-1}\kappa } \left( a'_{11} -1 -\frac{b'_{1}}{b'_{2}}a'_{21} \right) \langle [\gamma _{z} ,\gamma _b] , y \rangle _{\kappa } ,\nonumber \\
& \langle [\gamma _{z} ,\gamma _b ] , y \rangle _{\kappa } ^{\gamma _b} = e^{2\pi \sqrt{-1}\kappa } (b'_{11}+b'_{22}-1 )\langle [\gamma _{z} ,\gamma _b] ,  y \rangle _{\kappa } .\nonumber 
\end{align}
We denote the monodromy matrix on the cycle $\gamma _p$ with respect to the functions $\langle [\gamma _{z} ,\gamma _{a}] , y\rangle _{\kappa } ,  \langle [\gamma _{z} ,\gamma _{b}] , y\rangle _{\kappa } $ by $M^{(p)}_{a,b}$.
Then we have
\begin{align}
& M^{(a)}_{a,b}= \left( 
\begin{array}{cc}
e^{2\pi \sqrt{-1}\kappa } (a'_{11}+a'_{22}-1 ) &  b'_{11} -1 -\frac{a'_1}{a'_2}b'_{21} \\
0 & 1
\end{array} \right) , \label{eq:monodApBp} \\
&  M^{(b)}_{a,b} =\left( 
\begin{array}{cc}
1 &  0 \\
e^{2\pi \sqrt{-1}\kappa } \left( a'_{11} -1 -\frac{b'_{1}}{b'_{2}}a'_{21} \right) & e^{2\pi \sqrt{-1}\kappa } (b'_{11}+b'_{22}-1 )
\end{array} \right) , \nonumber
\end{align}
Thus tr$( M^{(a)}_{a,b}) =1+e^{2\pi \sqrt{-1} (\kappa +\theta _a)}$, tr$( M^{(b)}_{a,b}) =1+e^{2\pi \sqrt{-1} (\kappa +\theta _b)}$, and we have
\begin{align}
&  M^{(a)}_{a,b} M^{(b)}_{a,b}=e^{2\pi \sqrt{-1}\kappa } \left( 
\begin{array}{cc}
a'_{11} b'_{11}+a'_{12} b'_{21}+a'_{21} b'_{12}+a'_{22} b'_{22}\atop{-b'_{11}-b'_{22}+1} & (b'_{11}+b'_{22}-1)(b'_{11} -1 -\frac{a'_1}{a'_2}b'_{21}) \\
a'_{11} -1 -\frac{b'_1}{b'_2}a'_{21} & b'_{11}+b'_{22}-1 
\end{array}
\right) ,\\
& \mbox{tr}( M^{(a)}_{a,b} M^{(b)}_{a,b})=e^{2\pi \sqrt{-1}\kappa } \mbox{tr}(M'^{(a)}M'^{(b)}), \quad \mbox{det}( M^{(a)}_{a,b} M^{(b)}_{a,b})=e^{4\pi \sqrt{-1}\kappa } \mbox{det}(M'^{(a)}M'^{(b)}). \nonumber
\end{align}
Note that values of the trace and determinant are independent of the choice of basis.

We consider the case where $\langle [\gamma _{z} ,\gamma _a] , y \rangle _{\kappa }$, $\langle [\gamma _{z} ,\gamma _b] , y \rangle _{\kappa }$ are linearly dependent.
We further assume that the point $w=c$ is also a regular singularity, $\langle [\gamma _{z} ,\gamma _p] , y\rangle _{\kappa } \neq 0 $ for $p=a,b,c$ and $\langle [\gamma _{z} ,\gamma _a] , y \rangle _{\kappa }$, $\langle [\gamma _{z} ,\gamma _c] , y \rangle _{\kappa }$ are linearly independent.
Then $\langle [\gamma _{z} ,\gamma _b] , y\rangle _{\kappa } =d \langle [\gamma _{z} ,\gamma _a] , y\rangle _{\kappa } $ for some $d\neq 0$.
It follows from Eq.(\ref{eq:monodrels}) that $\langle [\gamma _{z} ,\gamma _a] , y\rangle _{\kappa }  ^{\gamma _a}= a_{11} \langle [\gamma _{z} ,\gamma _a] , y\rangle _{\kappa } = (1+a_{12}/d) \langle [\gamma _{z} ,\gamma _a] , y\rangle _{\kappa }$ and $\langle [\gamma _{z} ,\gamma _a] , y\rangle _{\kappa }  ^{\gamma _b}= b_{22} \langle [\gamma _{z} ,\gamma _a] , y\rangle _{\kappa } = (1+db_{21}) \langle [\gamma _{z} ,\gamma _a] , y\rangle _{\kappa }$, where  $a_{11}= e^{2\pi \sqrt{-1}\kappa } (a'_{11}+a'_{22}-1 )$, $a_{12}=  b'_{11} -1 -b'_{21} a'_1 / a'_2$, $b_{21} = e^{2\pi \sqrt{-1}\kappa } ( a'_{11} -1 -a'_{21}b'_{1}/b'_{2} )$, $b_{22}= e^{2\pi \sqrt{-1}\kappa } (b'_{11}+b'_{22}-1 )$.
Hence we have $a_{11} =1+a_{12}/d$, $b_{22} = 1+db_{21}$ and $a_{12}b_{21}=(a_{11}-1)(b_{22}-1)$.
By applying Eqs.(\ref{eq:pij}), (\ref{eq:monodApBp}), the relation $a_{12}b_{21}=(a_{11}-1)(b_{22}-1)$ can be written as
\begin{equation}
e^{4\pi \sqrt{-1}\kappa }\det (M'^{(a)}M'^{(b)})-e^{2\pi \sqrt{-1}\kappa }\mbox{tr}(M'^{(a)}M'^{(b)})+1=0.
\label{eq:detABtrAB}
\end{equation} 
Assume that the points $a$, $b$, $c$ are located anticlockwise with respect to the point $o$.
In a similar way as we obtained Eq.(\ref{eq:monodApBp}), we have 
\begin{align}
& ( \langle [\gamma _{z} ,\gamma _b] , y\rangle _{\kappa }  ^{\gamma _b} ,  \langle [\gamma _{z} ,\gamma _c] , y\rangle _{\kappa }  ^{\gamma _b} )= ( \langle [\gamma _{z} ,\gamma _b] , y\rangle _{\kappa } ,  \langle [\gamma _{z} ,\gamma _c] , y\rangle _{\kappa } ) \label{eq:monodBpCp} \\
& \quad \quad \quad \quad \quad \quad \quad \quad \quad \quad \left( 
\begin{array}{cc}
e^{2\pi \sqrt{-1}\kappa } (b'_{11}+b'_{22}-1 ) &  c'_{11} -1 -\frac{b'_1}{b'_2}c'_{21} \\
0 & 1
\end{array} \right) , \nonumber \\
& ( \langle [\gamma _{z} ,\gamma _b] , y\rangle _{\kappa }  ^{\gamma _c} ,  \langle [\gamma _{z} ,\gamma _c] , y\rangle _{\kappa }  ^{\gamma _c} )= ( \langle [\gamma _{z} ,\gamma _b] , y\rangle _{\kappa } ,  \langle [\gamma _{z} ,\gamma _c] , y\rangle _{\kappa } ) \nonumber \\
&\quad \quad \quad \quad \quad \left( 
\begin{array}{cc}
1 &  0 \\
e^{2\pi \sqrt{-1}\kappa } \left( b'_{11} -1 -\frac{c'_{1}}{c'_{2}}b'_{21} \right) & e^{2\pi \sqrt{-1}\kappa } (c'_{11}+c'_{22}-1 )
\end{array} \right) , \nonumber \\
& ( \langle [\gamma _{z} ,\gamma _c] , y\rangle _{\kappa }  ^{\gamma _c} ,  \langle [\gamma _{z} ,\gamma _a] , y\rangle _{\kappa }  ^{\gamma _c} )= ( \langle [\gamma _{z} ,\gamma _c] , y\rangle _{\kappa } ,  \langle [\gamma _{z} ,\gamma _a] , y\rangle _{\kappa } ) \nonumber \\
& \quad \quad \quad \quad \quad \quad \quad \quad \quad \quad \left( 
\begin{array}{cc}
e^{2\pi \sqrt{-1}\kappa } (c'_{11}+c'_{22}-1 ) &  a'_{11} -1 -\frac{c'_1}{c'_2}a'_{21} \\
0 & 1
\end{array} \right) , \nonumber \\
& ( \langle [\gamma _{z} ,\gamma _c] , y\rangle _{\kappa }  ^{\gamma _a} ,  \langle [\gamma _{z} ,\gamma _a] , y\rangle _{\kappa }  ^{\gamma _a} )= ( \langle [\gamma _{z} ,\gamma _c] , y\rangle _{\kappa } ,  \langle [\gamma _{z} ,\gamma _a] , y\rangle _{\kappa } ) \nonumber \\
&\quad \quad \quad \quad \quad \left( 
\begin{array}{cc}
1 &  0 \\
e^{2\pi \sqrt{-1}\kappa } \left( c'_{11} -1 -\frac{a'_{1}}{a'_{2}}c'_{21} \right) & e^{2\pi \sqrt{-1}\kappa } (a'_{11}+a'_{22}-1 )
\end{array} \right) . \nonumber
\end{align}
By applying $\langle [\gamma _{z} ,\gamma _b] , y\rangle _{\kappa } =d \langle [\gamma _{z} ,\gamma _a] , y\rangle _{\kappa } $, we obtain
\begin{align}
& ( \langle [\gamma _{z} ,\gamma _a] , y\rangle _{\kappa }  ^{\gamma _a} ,  \langle [\gamma _{z} ,\gamma _c] , y\rangle _{\kappa }  ^{\gamma _a} )= ( \langle [\gamma _{z} ,\gamma _a] , y\rangle _{\kappa } ,  \langle [\gamma _{z} ,\gamma _c] , y\rangle _{\kappa } )  \\
&\quad \quad \quad \quad \quad \left( 
\begin{array}{cc}
e^{2\pi \sqrt{-1}\kappa } (a'_{11}+a'_{22}-1 ) &  e^{2\pi \sqrt{-1}\kappa } \left( c'_{11} -1 -\frac{a'_{1}}{a'_{2}}c'_{21} \right) \\
0 & 1
\end{array} \right) , \nonumber \\
& ( \langle [\gamma _{z} ,\gamma _a] , y\rangle _{\kappa }  ^{\gamma _b} ,  \langle [\gamma _{z} ,\gamma _c] , y\rangle _{\kappa }  ^{\gamma _b} )= ( \langle [\gamma _{z} ,\gamma _a] , y\rangle _{\kappa } ,  \langle [\gamma _{z} ,\gamma _c] , y\rangle _{\kappa } ) \nonumber \\
& \quad \quad \quad \quad \quad \quad \quad \quad \quad \quad \left( 
\begin{array}{cc}
e^{2\pi \sqrt{-1}\kappa } (b'_{11}+b'_{22}-1 ) &  (c'_{11} -1 -\frac{b'_1}{b'_2}c'_{21})d \\
0 & 1
\end{array} \right) . \nonumber 
\end{align}
Then $\det (M^{(a)}_{a,c}M^{(b)}_{a,c})=e^{4\pi \sqrt{-1}\kappa }\det (M'^{(a)}M'^{(b)})$ and $\mbox{tr}(M^{(a)}_{a,c}M^{(b)}_{a,c}) = e^{4\pi \sqrt{-1}\kappa }\det (M'^{(a)}M'^{(b)}) +1$.
Combining with Eq.(\ref{eq:detABtrAB}), we have $\mbox{tr}(M^{(a)}_{a,c}M^{(b)}_{a,c}) = e^{2\pi \sqrt{-1}\kappa }\mbox{tr}(M'^{(a)}M'^{(b)})$, which is also shown for the case where the points $a$, $b$, $c$ are located clockwise with respect to the point $o$.

\section{Solutions and Monodromy} \label{sec:solmon}

In this section, we start with the proposition on the existence of a global simple solution of a second-order Fuchsian differential equation and the reducibility of the monodromy.
The monodromy representation of solutions of a Fuchsian differential equation is said to be reducible, iff the monodromy matrices $M^{\gamma }$ for fixed basis of solutions have a non-trivial invariant subspace which does not depend on the cycle $\gamma $, i.e. there exists a non-trivial subspace of the solutions which is invariant under analytic continuation along any cycle. 
\begin{prop} \label{prop:monoredpolynsol}
Let $Dy=0$ $(D=d^2/dw^2 +a_1(w)d/dw+a_2(w))$ be a second-order Fuchsian differential equation with singularities $\{ t_1, \dots , t_n , \infty (=t_{n+1}) \}$, and let $\theta _l ^{(1)}$ and  $\theta _l ^{(2)}$ be the exponents at the singularity $w=t_l$ $(l=1,\dots ,n+1)$.\\
(i) If the monodromy representation of solutions of the differential equation $Dy=0$ is reducible, there exists a non-zero solution $y(w)$ such that $y(w)= h(w) \prod _{l=1}^n (w-t_l) ^{\alpha _l} $, $h(w)$ is a polynomial in the variable $w$ and $ \prod _{l=1}^n h(t_l) \neq 0$.\\
(ii) If there exists a non-zero solution $y(w)$ of $Dy=0$ such that $y(w)= h(w) \prod _{l=1}^n (w-t_l) ^{\alpha _l} $, $h(w)$ is a polynomial of degree $k$ and $ \prod _{l=1}^n h(t_l) \neq 0$, then $\alpha _l \in \{ \theta _l ^{(1)}, \theta _l ^{(2)}\}$ for each $l =1,\dots ,n$ and $\theta _{n+1} ^{(1)}=-k-\sum _{l=1}^n \alpha _l $ or $\theta _{n+1} ^{(2)}=-k-\sum _{l=1}^n \alpha _l $.
\end{prop}
\begin{proof}
Assume that the monodromy representation of solutions of the Fuchsian equation $Dy=0$ is reducible.
Let $\tilde{\gamma }_l$ $(l=1,\dots ,n+1)$ be a cycle on the Riemann sphere which traces a path around the singularity $w=t_l$ anti-clockwise.
Since the dimension of the space of solutions of the differential equation $Dy=0$ is two, it follows from reducibility that there exists a non-zero solution $y(w)$ such that $y^{\tilde{\gamma }_l}(w) =e^{2\pi \sqrt{-1} \tilde{\alpha }_l}y(w)$ for some constants $\tilde{\alpha }_l $ and $l=1,\dots ,n+1$.
The monodromy of the function $y(w) \prod_{l=1}^n (w-t_l)^{-\tilde{\alpha }_l}$ on $\Cplx$ is trivial, because $\tilde{\gamma }_{n+1}$ is written as products of $\tilde{\gamma }_l ^{-1}$ $(l=1,\dots ,n)$.
Since $y(w)$ satisfies the Fuchsian equation $Dy=0$, the function $y(w) \prod_{l=1}^n (w-t_l)^{-\tilde{\alpha }_l}$ does not have any singularities except for $\{t_1 \dots ,t_{n+1} \}$, and the singularity $w=\infty $ is regular at most.
Hence the function $y(w) \prod_{l=1}^n (w-t_l)^{-\tilde{\alpha }_l}$ may have poles at $w=t_l$ $(l=1,\dots ,n)$, holomorphic on $\Cplx \setminus \{ t_1, \dots ,t_n\}$ and the regular singurality $w=\infty $ is apparent.
Therefore we have $y(w) \prod_{l=1}^n (w-t_l)^{-\tilde{\alpha }_l} = p(w) \prod_{l=1}^n (w-t_l)^{-m_l}$ for some integers $m_1, \dots , m_n$ and a polynomial $p(w)$.
Thus $y(w)$ may be written as $y(w)= h(w) \prod_{l=1}^n (w-t_l)^{\alpha _l}$, where $h(w)$ is a polynomial in the variable $w$ and $\prod _{l=1}^n h(t_l) \neq 0$.
Therefore we have (i).

If $y(w)= h(w) \prod_{l=1}^n (w-t_l)^{\alpha _l}$ ($h(w)$: a polynomial of degree $k$, $\prod _{l=1}^n h(t_l) \neq 0$) satisfies the Fuchsian equation $Dy=0$, it follows from that the exponents at $w=l$ ($l \in \{1,\dots ,n \}$) are $\theta _l ^{(1)}, \theta _l ^{(2)}$ and $h(p)\neq 0$ that $\alpha _l \in \{\theta _l ^{(1)}, \theta _l ^{(2)} \}$.
Write $h(w) =c_k w^k +c_{k-1} w^{k-1} +\dots +c_0$ $(c_k \neq 0)$.
Then we have the expansion $y(w) = (1/w)^{-k-\sum _{l=1}^n \alpha _l} (c_k +(c_{k-1}-c_k \sum_{l=1}^n t_l \alpha _l)/w +\dots )$ around $w =\infty $ and the index $-k-\sum _{l=1}^n \alpha _l$ must coincide with one of the exponents at $w=\infty $, which are $\theta _{n+1} ^{(1)}$ and $\theta _{n+1} ^{(2)}$.
Therefore we have $\theta _{n+1} ^{(1)}=-k-\sum _{l=1}^n \alpha _l $ or $\theta _{n+1} ^{(2)}=-k-\sum _{l=1}^n \alpha _l $.
\end{proof}
Proposition \ref{prop:monoredpolynsol} is applicable to the Fuchsian equation $D_{y_1}(\theta _0, \theta _1, \theta _t,\theta _{\infty}; \lambda ,\mu ) $ and Heun's equation (\ref{Heun01}) readily.

We introduce a proposition which connects polynomial-type solutions (the ones described in Proposition \ref{prop:monoredpolynsol}) to the ones such that one of the singularities is apparent through the Euler's integral transformation.
A regular singularity $z=p$ of a second-order linear differential equation is called apparent, if and only if the monodromy matrix along the singularity $z=p$ is a scalar matrix. 
Note that it is equivalent that the difference of the exponents is an integer and the logarithmic solution along $z=p$ disappear, which means $A^{\langle p \rangle}=0 $ in Eq.(\ref{eq:expansfpgp}) in the case that one of the exponents is zero.

For a cycle $\gamma$, we set
\begin{align}
\langle \gamma , y \rangle =
\left\{
\begin{array}{ll}
\langle \gamma  , y \rangle _{\kappa _2-1}, & \mbox{if } y(w) \mbox{ satisfies } D_{y_1}(\theta _0, \theta _1, \theta _t,\theta _{\infty}; \lambda ,\mu ), \\
\langle \gamma , y \rangle _{-\eta }, & \mbox{if } y(w) \mbox{ satisfies Heun's equation (\ref{Heun01})}.
\end{array}
\right.
\end{align}
It follows from Proposition \ref{prop:Nov} (resp. Proposition \ref{prop:Heunint}) that if $y(w)$ is a solution of $D_{y_1}(\theta _0, \theta _1, \theta _t,\theta _{\infty}; \lambda ,\mu )$ (resp. Eq.(\ref{Heun01})) then $\langle [\gamma _{z} ,\gamma _p] , y \rangle $ $(p\in \{0,1,t,\infty \} )$ is a solution of $D_{y_1}(\tilde{\theta }_0, \tilde{\theta }_1, \tilde{\theta }_t, \tilde{\theta }_{\infty} ; \tilde{\lambda } ,\tilde{\mu } )$ (resp. Eq.(\ref{Heun02})).
The local expansion of the function $\langle [\gamma _{z} ,\gamma _p] , y \rangle $ $(p=0,1,t,\infty )$ about $z=p$ can be calculated using Eqs.(\ref{eq:gzgpexpa}), (\ref{eq:gzginftyexpa}) for the case $\kappa = \kappa _2 -1 $, $\theta _{\infty} ^{(1)}=\kappa _1$, $\theta _{\infty} ^{(2)}=\kappa _2 +1$ (resp.  $\theta _p=1-\epsilon ' _p$ $(p=0,1,t)$, $\kappa = -\eta $, $\theta _{\infty} ^{(1)}=\alpha +\beta -2\eta +1$, $\theta _{\infty} ^{(2)}=2-\eta $) by using the local expansion of solutions of $D_{y_1}(\theta _0, \theta _1, \theta _t,\theta _{\infty}; \lambda ,\mu )$ (resp. Eq.(\ref{Heun01})) (see Eqs.(\ref{eq:ywfg}), (\ref{eq:expansfpgp}), (\ref{eq:yinftyCD}), (\ref{eq:expansfinftyginfty})).
They are used to obtain the condition that the function $\langle [\gamma _{z} ,\gamma _p] , y \rangle $ $(p\in \{0,1,t,\infty \} )$ is identically zero for all solutions $y(w)$ to $D_{y_1}(\theta _0, \theta _1, \theta _t,\theta _{\infty}; \lambda ,\mu )$ (resp. Eq.(\ref{Heun01})).
\begin{prop} \label{prop:nonzero}
(i) Let $p\in \{0,1,t \} $ and assume $\kappa _2 \not \in \Zint$ (resp. $\eta \not \in \Zint$).
The function $\langle [\gamma _{z} ,\gamma _p] , y \rangle  $ is identically zero for all solutions $y (w)$ of $D_{y_1}(\theta _0, \theta _1, \theta _t,\theta _{\infty}; \lambda ,\mu )$ (resp. Eq.(\ref{Heun01})), if and only if $\theta _p \in \Zint _{\geq 0}$ (resp. $\epsilon '_p \in \Zint _{\leq 1}$) and the singularity $w=p$ is apparent (i.e. $A^{\langle p \rangle} =0$ in Eq.(\ref{eq:expansfpgp})), or $\theta _p +\kappa _2 \in \Zint _{\leq -1} $ (resp. $\epsilon _p \in \Zint _{\geq 2}$) and the differential equation $D_{y_1}(\theta _0, \theta _1, \theta _t,\theta _{\infty}; \lambda ,\mu )$ (resp. Eq.(\ref{Heun01})) has a solution of the form of a product of $(w-p) ^{\theta _p }$ (resp. $(w-p) ^{1- \epsilon '_p }$) and a non-zero polynomial of degree no more than $-\theta _p -\kappa _2 -1$ (resp. $\epsilon _p -2$).\\
(ii) Under the assumption $\kappa _2 \not \in \Zint$ (resp. $\eta \not \in \Zint$), the function $\langle [\gamma _{z} ,\gamma _{\infty }] , y \rangle  $ is identically zero for all solutions $y (w)$ of $D_{y_1}(\theta _0, \theta _1, \theta _t,\theta _{\infty}; \lambda ,\mu )$ (resp. Eq.(\ref{Heun01})), if and only if $\theta _{\infty } \in \Zint _{\geq 1}$ (resp. $\alpha +\beta -\eta \in \Zint _{\geq 1} $) and the singularity $w={\infty }$ is apparent, or $\kappa _1 \in \Zint _{\leq 0} $ (resp. $\alpha +\beta -2\eta \in \Zint _{\leq -1}$) and the differential equation $D_{y_1}(\theta _0, \theta _1, \theta _t,\theta _{\infty}; \lambda ,\mu )$ (resp. Eq.(\ref{Heun01})) has a non-zero polynomial in the variable $w$ of degree $-\kappa _1$ (resp. $2\eta -\alpha -\beta -1$).
\end{prop}
Proposition \ref{prop:nonzero} will be proved in the appendix.

We investigate a sufficient condition for the functions $\langle [\gamma _{z} ,\gamma _0] , y \rangle $, $\langle [\gamma _{z} ,\gamma _1] , y \rangle $, $\langle [\gamma _{z} ,\gamma _t] , y \rangle $ to span the two-dimensional space of solutions of $D_{y_1}(\tilde{\theta }_0, \tilde{\theta }_1, \tilde{\theta }_t, \tilde{\theta }_{\infty} ; \tilde{\lambda } ,\tilde{\mu } )$ (resp. Eq.(\ref{Heun02})).
\begin{prop} \label{prop:spansp0}
There exists a solution $y(w)$ of $D_{y_1}(\theta _0, \theta _1, \theta _t,\theta _{\infty}; \lambda ,\mu )$ (resp. Eq.(\ref{Heun01})) such that  $\langle [\gamma _{z} ,\gamma _0] , y \rangle \neq 0$, $\langle [\gamma _{z} ,\gamma _1] , y \rangle \neq 0$, $\langle [\gamma _{z} ,\gamma _t] , y \rangle \neq 0$ and the functions $\langle [\gamma _{z} ,\gamma _0] , y \rangle $, $\langle [\gamma _{z} ,\gamma _1] , y \rangle $, $\langle [\gamma _{z} ,\gamma _t] , y \rangle $ span the two-dimensional space of solutions of $D_{y_1}(\tilde{\theta }_0, \tilde{\theta }_1, \tilde{\theta }_t, \tilde{\theta }_{\infty} ; \tilde{\lambda } ,\tilde{\mu } )$ (resp. Eq.(\ref{Heun02})), 
if $\kappa _2 \not \in \Zint$ and $\theta _p , \tilde{\theta }_p \not \in \Zint$ for all $p \in \{ 0,1,t, \infty \}$ (resp. $\eta , \epsilon _0 , \epsilon _1, \epsilon _t , \alpha -\beta , \epsilon _0 ',\epsilon _1 ',\epsilon _t ' ,\alpha '-\beta ' \not \in \Zint$).
\end{prop}
We will prove Proposition \ref{prop:spansp0} in the appendix with a more detailed proposition.

By applying the results on the monodromy of integral representations in section \ref{sec:MonInttr}, we have the following theorem for monodromy matrices:
\begin{thm} \label{thm:monodm}
Let $a,b \in \{ 0,1,t \}$ $(a\neq b)$, $M'^{(p)}$ be a monodromy matrix of a certain basis of solutions of $D_{y_1}( \theta _0, \theta _1, \theta _t, \theta _{\infty} ; \lambda ,\mu ) $ (resp. Eq.(\ref{Heun01})) on the cycle $\gamma _p$ ($p\in \{a,b\} $) and $M^{(p)}$ be a monodromy matrix of a certain basis of solutions of $D_{y_1}(\tilde{\theta }_0, \tilde{\theta }_1, \tilde{\theta }_t, \tilde{\theta }_{\infty} ; \tilde{\lambda } ,\tilde{\mu } ) $ (resp. Eq.(\ref{Heun02})) on the cycle $\gamma _p$.
Then we have
\begin{align}
\mbox{\rm tr} ((M^{(a)} M^{(b)})^n) =e^{2\pi \sqrt{-1} n\kappa _2} \mbox{\rm tr} ((M'^{(a)}M'^{(b)})^n) , & \label{eq:trinv} \\
\mbox{\rm (resp. tr} ((M^{(a)} M^{(b)})^n) =e^{-2\pi \sqrt{-1} n\eta} \mbox{\rm tr} ((M'^{(a)}M'^{(b)})^n) & , \nonumber 
\end{align}
for $n\in \Zint $.
\end{thm}
\begin{proof}
Set $\kappa =\kappa _2 -1 $ (resp. $\kappa =-\eta $) for the case of the differential equation $D_{y_1}( \theta _0, \theta _1, \theta _t, \theta _{\infty} ; \lambda ,\mu ) $ (resp. Eq.(\ref{Heun01})).
Firstly we show that $\mbox{\rm tr} (M^{(a)} M^{(b)}) =e^{2\pi \sqrt{-1} \kappa } \mbox{\rm tr} (M'^{(a)}M'^{(b)}) $ and $\mbox{\rm det} (M^{(a)} M^{(b)}) =e^{4\pi \sqrt{-1} \kappa } \mbox{\rm det} (M'^{(a)}M'^{(b)}) $ under the assumption of Proposition \ref{prop:spansp0}.
It follows from Proposition \ref{prop:spansp0} that there exists a solution $y(w)$ of $D_{y_1}(\theta _0, \theta _1, \theta _t,\theta _{\infty}; \lambda ,\mu )$ (resp. Eq.(\ref{Heun01})) such that $\langle [\gamma _{z} ,\gamma _0] , y \rangle \neq 0$, $\langle [\gamma _{z} ,\gamma _1] , y \rangle \neq 0$, $\langle [\gamma _{z} ,\gamma _t] , y \rangle \neq 0$ and the functions $\langle [\gamma _{z} ,\gamma _0] , y \rangle $, $\langle [\gamma _{z} ,\gamma _1] , y \rangle $, $\langle [\gamma _{z} ,\gamma _t] , y \rangle $ span the two-dimensional space of solutions of $D_{y_1}(\tilde{\theta }_0, \tilde{\theta }_1, \tilde{\theta }_t, \tilde{\theta }_{\infty} ; \tilde{\lambda } ,\tilde{\mu } )$ (resp. Eq.(\ref{Heun02})).
Let $c$ be the element in $\{0,1,t \}$ which is different from $a$ and $b$.
Then $\langle [\gamma _{z} ,\gamma _a] , y \rangle $, $\langle [\gamma _{z} ,\gamma _b] , y \rangle $ are linearly independent or  $\langle [\gamma _{z} ,\gamma _a] , y \rangle $, $\langle [\gamma _{z} ,\gamma _c] , y \rangle $ are linearly independent and  $\langle [\gamma _{z} ,\gamma _a] , y \rangle  =d \langle [\gamma _{z} ,\gamma _b] , y \rangle $ for some $d (\neq 0)$.
Hence it follows from the calculations of monodromy in section \ref{sec:MonInttr} that if $\kappa _2 \not \in \Zint$ and $\theta _p , \tilde{\theta }_p \not \in \Zint$ for all $p \in \{ 0,1,t, \infty \}$ (resp. $\eta , \epsilon _0 , \epsilon _1, \epsilon _t , \alpha -\beta , \epsilon _0 ',\epsilon _1',\epsilon _t ' ,\alpha '-\beta ' \not \in \Zint$) then $\mbox{\rm tr} (M^{(a)} M^{(b)}) =e^{2\pi \sqrt{-1} \kappa } \mbox{\rm tr} (M'^{(a)}M'^{(b)}) $ and $\mbox{\rm det} (M^{(a)} M^{(b)}) =e^{4\pi \sqrt{-1} \kappa } \mbox{\rm det} (M'^{(a)}M'^{(b)}) $.
It is known that continuity of the coefficients of the differential equation with respect to a parameter implies continuity of solutions of the differential equation and monodromy with respect to the parameter.
Hence we have $\mbox{tr}(M^{(a)} M^{(b)}) = e^{2\pi \sqrt{-1}\kappa }\mbox{tr}(M'^{(a)}M'^{(b)})$ and $\mbox{\rm det} (M^{(a)} M^{(b)}) =e^{4\pi \sqrt{-1} \kappa } \mbox{\rm det} (M'^{(a)}M'^{(b)}) $ for all cases by taking a limit from the case $\kappa _2 \not \in \Zint$ and $\theta _p , \tilde{\theta }_p \not \in \Zint$ for all $p \in \{ 0,1,t, \infty \}$ (resp. $\eta , \epsilon _0 , \epsilon _1, \epsilon _t , \alpha -\beta , \epsilon _0 ', \epsilon _1',\epsilon _t ' ,\alpha '-\beta ' \not \in \Zint$).
Let $l '_1$, $l '_2$ (resp. $l _1$, $l _2$) be the solutions of the quadratic equation $x^2 -\mbox{tr}(M'^{(a)} M'^{(b)}) x+\mbox{\rm det} (M'^{(a)} M'^{(b)})=0$ (resp. $x^2 -\mbox{tr}(M^{(a)} M^{(b)}) x +\mbox{\rm det} (M^{(a)} M^{(b)})=0$).
Then we have $\{ l _1, l _2 \} =\{ e^{2\pi \sqrt{-1}\kappa } l '_1, e^{2\pi \sqrt{-1}\kappa } l '_2 \} $ and $\mbox{\rm tr} ((M^{(a)} M^{(b)})^n) = (l _1)^n+(l _2)^n = e^{2\pi \sqrt{-1} n\kappa }( (l '_1)^n+(l '_2)^n) = e^{2\pi \sqrt{-1} n\kappa } \mbox{\rm tr} ((M'^{(a)}M'^{(b)})^n) $ for $n \in \Zint$.
\end{proof}
It follows from the relations $M'^{(0)}M'^{(1)}M'^{(t)}M'^{(\infty )} =1 $ and $M^{(0)}M^{(1)}M^{(t)}M^{(\infty )} =1 $ that $\mbox{\rm tr} ((M^{(p)} M^{(\infty )})^n) =e^{-2\pi \sqrt{-1} n\kappa _2} \mbox{\rm tr} ((M'^{(p)}M'^{(\infty )})^n)$ (resp. $\mbox{\rm tr} ((M^{(p)} M^{(\infty )})^n) =e^{2\pi \sqrt{-1} n\eta} \mbox{\rm tr} ((M'^{(p)}M'^{(\infty )})^n)$) for $p\in \{0,1,t\}$ and $n\in \Zint $.
It seems that we do not have a simple formula connecting tr$(M^{(a)}(M^{(b)})^{-1})$ and tr$(M'^{(a)}(M'^{(b)})^{-1})$ for $a,b \in \{ 0,1,t \}$, $a\neq b$.
Note that Eq.(\ref{eq:trinv}) can be written as tr$((M^{(a)}M^{(b)})^n)=$tr$((M'^{(a)}M'^{(b)})^n)$ for a $2\times 2$ $sl_2$-Fuchsian system with four singularities, and it was obtained by Inaba-Iwasaki-Saito \cite{IIS} and Boalch \cite{Boa}.

\section{Correspondence between polynomial-type solutions and apparency of a singularity} \label{sec:polnonlog}

In this section, we establish correspondences between polynomial-type solutions and apparency of a singularity, which are induced by integral transformations.

Let $y(w)$ be a solution of $D_{y_1}(\theta _0, \theta _1, \theta _t,\theta _{\infty}; \lambda ,\mu )$ (resp. Eq.(\ref{Heun01})), $p\in \{0,1,t\} $, and consider the local expansion of the solution about $w=p$ as Eqs.(\ref{eq:ywfg}), (\ref{eq:expansfpgp}) by setting $\kappa = \kappa _2 -1 $ (resp.  $\theta _p=1-\epsilon ' _p$, $\kappa = -\eta $).
It follows from the local expansion of $\langle [\gamma _{z} ,\gamma _p] , y \rangle $ about $z=p$ (see Eq.(\ref{eq:gzgpexpa}) for the case $\theta _p \in \Zint _{\leq -1}, \theta _p + \kappa \not \in \Zint _{\leq -2}$) that if $\theta _p \in \Zint _{\leq -1}$, $\kappa _2 \not \in \Zint$ (resp. $\epsilon '_p \in \Zint _{\geq 2}$, $\eta \not \in \Zint$) and the singularity $w=p$ of the differential equation $D_{y_1}(\theta _0, \theta _1, \theta _t,\theta _{\infty}; \lambda ,\mu )$ (resp. Eq.(\ref{Heun01})) is apparent, then $A^{\langle p \rangle} =0$ in Eqs.(\ref{eq:expansfpgp}), (\ref{eq:gzgpexpa}), the function $\langle [\gamma _{z} ,\gamma _p] , y \rangle $ is non-zero and it is a product of $ (z-p) ^{\theta _p +\kappa _2}$ (resp. $(z-p) ^{2-\epsilon ' _p -\eta }$) and a polynomial of degree no more than $-\theta _p-1$ (resp. $\epsilon ' _p -2$). 
Since $\langle [\gamma _{z} ,\gamma _p] , y \rangle $ satisfies $D_{y_1}(\tilde{\theta }_0, \tilde{\theta }_1, \tilde{\theta }_t, \tilde{\theta }_{\infty} ; \tilde{\lambda } ,\tilde{\mu } )$ (resp. Eq.(\ref{Heun02})), we have the following proposition:
\begin{prop} \label{prop:nonlog-polyn01}
Let $p \in \{ 0,1,t \} $.
If $\theta _p \in \Zint _{\leq -1}$, $\kappa _2 \not \in \Zint$ (resp. $\epsilon '_p \in \Zint _{\geq 2}$, $\eta \not \in \Zint$) and the singularity $w=p$ of the differential equation $D_{y_1}(\theta _0, \theta _1, \theta _t,\theta _{\infty}; \lambda ,\mu )$ (resp. Eq.(\ref{Heun01})) is apparent, then there exists a non-zero solution of $D_{y_1}(\tilde{\theta }_0, \tilde{\theta }_1, \tilde{\theta }_t, \tilde{\theta }_{\infty} ; \tilde{\lambda } ,\tilde{\mu } )$ (resp. Eq.(\ref{Heun02})) which can be written as $(z-p) ^{\theta _p +\kappa _2}  h(z) $ (resp. $(z-p) ^{1- \epsilon _p}  h(z) $) where $h(z)$ is a polynomial of degree no more than $-\theta _p-1$ (resp. $\epsilon ' _p -2$).
\end{prop}

The following theorem asserts various correspondences between polynomial-type solutions and apparency of a singularity for Heun's equation.

\begin{thm} \label{thm:nonlog-polynHeun}
Let $a,b,c$ be elements of $\{0,1,t \}$ such that $a \neq b \neq c \neq a$ and $\eta , \alpha ,\beta ,\epsilon _0 ,\epsilon _1 ,\epsilon _t,  \alpha ',\beta ',\epsilon '_0 ,\epsilon '_1 ,\epsilon '_t$ be the parameters defined in Eq.(\ref{eq:mualbe}) or Eq.(\ref{eq:mualbe18}).\\
(i) If $\epsilon ' _a \in \Zint _{\geq 2}$, $\eta \not \in \Zint$ and the singularity $w=a$ of Eq.(\ref{Heun01}) is apparent, then there exists a non-zero solution of Eq.(\ref{Heun02}) which can be written as $(z-a) ^{1- \epsilon _a}  h(z) $ where $h(z)$ is a polynomial of degree no more than $\epsilon ' _a -2$.
Moreover if $\alpha ' ,\beta ' \not \in \Zint$, then $\deg _E h(z)= \epsilon ' _a -2$.
\\
(ii) If $\epsilon '_a \in \Zint _{\leq 0}$, $ \eta  \not \in \Zint$, $\alpha ', \beta ' , \epsilon _b  , \epsilon _c \not \in \Zint $ and the singularity $w=a$ of Eq.(\ref{Heun01}) is apparent, then there exists a non-zero solution of Eq.(\ref{Heun02}) which can be written as $(z-b) ^{1-\epsilon _b}(z-c) ^{1-\epsilon _c}  h(z) $ where $h(z)$ is a polynomial with $\deg h(z)= -\epsilon ' _a$.\\
(iii) If $\epsilon _a \in \Zint _{\geq 2}$, $\alpha ,\beta  \not \in \Zint$ and there exists a non-zero solution of Eq.(\ref{Heun01}) which can be written as a product of $(w-a)^{1-\epsilon '_a}$ and a polynomial, then the singularity $z=a$ of Eq.(\ref{Heun02}) is apparent.\\
(iv) If $\epsilon _a \in \Zint _{\leq 0}$, $\alpha , \beta , \epsilon '_b , \epsilon '_c \not \in \Zint $ and there exists a non-zero solution of Eq.(\ref{Heun01}) which can be written as $(w-b) ^{1-\epsilon '_b}(w-c) ^{1-\epsilon '_c}h(w)$ where $h(w)$ is a polynomial, then the singularity $z=a$ of Eq.(\ref{Heun02}) is apparent.\\
(v) If $\alpha +\beta -\eta \in \Zint _{\leq 0}$, $\eta \not \in \Zint$ and the singularity $w=\infty $ of Eq.(\ref{Heun01}) is apparent, then there exists a non-zero solution of Eq.(\ref{Heun02}) which can be written as a polynomial of degree $\eta -\alpha -\beta $.\\
(vi) If $ \alpha +\beta -\eta \in \Zint _{\geq 2}$, $\eta ,\epsilon _0 ,\epsilon _1 ,\epsilon _t \not \in \Zint$ and the singularity $w=\infty $ of Eq.(\ref{Heun01}) is apparent, then there exists a non-zero solution of Eq.(\ref{Heun02}) which can be written as $z ^{1-\epsilon _0 } (z-1) ^{1-\epsilon _1} (z-t) ^{1-\epsilon _t}  h(z)$ where $h(z)$ is a polynomial of degree $\alpha +\beta -\eta -2$.\\
(vii) If $\alpha +\beta -2\eta \in \Zint _{\leq -1}$, $\eta ,\epsilon '_0 \not \in \Zint$ and there exists a non-zero solution of Eq.(\ref{Heun01}) written as a polynomial in $w$, then the singularity $z=\infty $ of Eq.(\ref{Heun02}) is apparent.\\
(viii) If $\alpha +\beta -2\eta \in \Zint _{\geq 1}$, $\eta ,\epsilon '_0 , \epsilon '_1 ,\epsilon '_t \not \in \Zint$ and there exists a non-zero solution of Eq.(\ref{Heun01}) which can be written as $w^{1-\epsilon ' _0} (w-1) ^{1-\epsilon '_1 }(w-t) ^{1-\epsilon ' _t }  h(w)$ where $h(w)$ is a polynomial, then the singularity $z=\infty $ of Eq.(\ref{Heun02}) is apparent.
\end{thm}
We will prove Theorem \ref{thm:nonlog-polynHeun} in the appendix with a more detailed proposition.

Correspondences between polynomial-type solutions and apparency of a singularity for the differential equations $D_{y_1}(\theta _0, \theta _1, \theta _t,\theta _{\infty}; \lambda ,\mu )$ and $D_{y_1}(\tilde{\theta }_0, \tilde{\theta }_1, \tilde{\theta }_t, \tilde{\theta }_{\infty} ; \tilde{\lambda } ,\tilde{\mu } )$ can also be described as follows:
\begin{thm} \label{thm:nonlog-polyn}
Set $\kappa _1= (\theta _{\infty } -\theta _0 -\theta _1 -\theta _t)/2$, $\kappa _2= -(\theta _{\infty } +\theta _0 +\theta _1 +\theta _t)/2$, $\tilde{\theta } _p= \kappa _2 + \theta _p$ $(p=0,1,t,\infty )$ and
\begin{align}
& \tilde{\lambda} =\lambda -\frac{\kappa _2}{\mu -\frac{\theta _0}{\lambda }-\frac{\theta _1}{\lambda -1}-\frac{\theta _t}{\lambda -t}} ,\quad  \tilde{\mu} =\frac{\kappa _2 +\theta _0}{\tilde{\lambda }} +\frac{\kappa _2 +\theta _1}{\tilde{\lambda }-1} +\frac{\kappa _2 +\theta _t}{\tilde{\lambda }-t} +\frac{\kappa _2 }{\lambda - \tilde{\lambda }} . \label{eq:tlamtmu2}
\end{align}
Let $a,b,c$ be elements of $\{0,1,t \}$ such that $a \neq b \neq c \neq a$.
Assume that $\lambda , \tilde{\lambda} \not \in \{ 0,1,t,\infty \}$.\\
(i) If $\theta _a \in \Zint _{\leq -1}$, $\kappa _2 \not \in \Zint$ and the singularity $w=a$ of the differential equation $D_{y_1}(\theta _0, \theta _1, \theta _t,\theta _{\infty}; \lambda ,\mu )$ in the variable $w$ is apparent, then there exists a non-zero solution of the differential equation $D_{y_1}(\tilde{\theta }_0, \tilde{\theta }_1, \tilde{\theta }_t, \tilde{\theta }_{\infty} ; \tilde{\lambda } ,\tilde{\mu } )$ in the variable $z$ which can be written as $(z-a) ^{\tilde{\theta }_a}  h(z) $ where $h(z)$ is a polynomial of degree no more than $-\theta _a-1$.
Moreover if $\kappa _1 \not \in \Zint$, then $\deg _E h(z)= -\theta _a-1$.
\\
(ii) If $\theta _a \in \Zint _{\geq 0}$, $\kappa _1 , \kappa _2 , \tilde{\theta }_b , \tilde{\theta }_c \not \in \Zint$ and the singularity $w=a$ of the differential equation $D_{y_1}(\theta _0, \theta _1, \theta _t,\theta _{\infty}; \lambda ,\mu )$ is apparent, then there exists a non-zero solution of the differential equation $D_{y_1}(\tilde{\theta }_0, \tilde{\theta }_1, \tilde{\theta }_t, \tilde{\theta }_{\infty} ; \tilde{\lambda } ,\tilde{\mu } )$ which can be written as $(z-b) ^{\tilde{\theta }_b}(z-c) ^{\tilde{\theta }_c}  h(z) $ where $h(z)$ is a polynomial with $\deg h(z)= \theta _a$.\\
(iii) If $\tilde{\theta }_a \in \Zint _{\leq 0}$, $\kappa _{2 } , \theta _{\infty } \not \in \Zint $ and there exists a non-zero solution of $D_{y_1}(\theta _0, \theta _1, \theta _t,\theta _{\infty}; \lambda ,\mu )$ which can be written as a product of $(w-a)^{\theta _a}$ and a polynomial, then the singularity $z=a$ of $D_{y_1}(\tilde{\theta }_0, \tilde{\theta }_1, \tilde{\theta }_t, \tilde{\theta }_{\infty} ; \tilde{\lambda } ,\tilde{\mu } )$ is apparent.\\
(iv) If $\tilde{\theta }_a \in \Zint _{\geq 1}$, $\kappa _2 ,\theta _{\infty }, \theta _b, \theta _c \not \in \Zint$ and there exists a non-zero solution of $D_{y_1}(\theta _0, \theta _1, \theta _t,\theta _{\infty}; \lambda ,\mu )$ which can be written as $(w-b) ^{\theta _b}(w-c) ^{\theta _c}h(w)$ where $h(w)$ is a polynomial, then the singularity $z=a$ of $D_{y_1}(\tilde{\theta }_0, \tilde{\theta }_1, \tilde{\theta }_t, \tilde{\theta }_{\infty} ; \tilde{\lambda } ,\tilde{\mu } )$ is apparent.\\
(v) If $\theta _{\infty } \in \Zint _{\leq 0}$, $\kappa _2 \not \in \Zint$ and the singularity $w=\infty $ of $D_{y_1}(\theta _0, \theta _1, \theta _t,\theta _{\infty}; \lambda ,\mu )$ is apparent, then there exists a non-zero solution of $D_{y_1}(\tilde{\theta }_0, \tilde{\theta }_1, \tilde{\theta }_t, \tilde{\theta }_{\infty} ; \tilde{\lambda } ,\tilde{\mu } )$ which can be written as a polynomial of degree $-\theta _{\infty }$.\\
(vi) If $\theta _{\infty } \in \Zint _{\geq 1}$, $\kappa _2 ,\tilde{\theta }_0, \tilde{\theta }_1, \tilde{\theta }_t \not \in \Zint $ and the singularity $w=\infty $ of $D_{y_1}(\theta _0, \theta _1, \theta _t,\theta _{\infty}; \lambda ,\mu )$ is apparent, then there exists a non-zero solution of $D_{y_1}(\tilde{\theta }_0, \tilde{\theta }_1, \tilde{\theta }_t, \tilde{\theta }_{\infty} ; \tilde{\lambda } ,\tilde{\mu } )$ which can be written as $z ^{\tilde{\theta }_0 } (z-1) ^{\tilde{\theta }_1} (z-t) ^{\tilde{\theta }_t}  h(z)$ where $h(z)$ is a polynomial of degree $\theta _{\infty }-1$.\\
(vii) If $\kappa _1 \in \Zint _{\leq 0}$, $\kappa _2 ,\theta _0 \not \in \Zint $  and there exists a non-zero solution of $D_{y_1}(\theta _0, \theta _1, \theta _t,\theta _{\infty}; \lambda ,\mu )$ written as a polynomial in $w$, then the singularity $z=\infty $ of $D_{y_1}(\tilde{\theta }_0, \tilde{\theta }_1, \tilde{\theta }_t, \tilde{\theta }_{\infty} ; \tilde{\lambda } ,\tilde{\mu } )$ is apparent.\\
(viii) If $\kappa _1 \in \Zint _{\geq 1}$, $\kappa _2 , \theta _0 ,\theta _1 ,\theta _t \not \in \Zint$ and there exists a non-zero solution of $D_{y_1}(\theta _0, \theta _1, \theta _t,\theta _{\infty}; \lambda ,\mu )$ which can be written as $w^{\theta _0} (w-1) ^{\theta _1 }(w-t) ^{\theta _t }  h(w)$ where $h(w)$ is a polynomial, then the singularity $z=\infty $ of $D_{y_1}(\tilde{\theta }_0, \tilde{\theta }_1, \tilde{\theta }_t, \tilde{\theta }_{\infty} ; \tilde{\lambda } ,\tilde{\mu } )$ is apparent.
\end{thm}

\section{Quasi-solvability and apparency of a singularity for Heun's equation} \label{sec:QSnonlogHeun}

We recall the quasi-solvability of Heun's equation.
\begin{prop} $($\cite{Ron,Tak2} etc.$)$ \label{prop:findim0}
Let $\nu _j \in \{ 0,1-\epsilon '_j \}$ for $j=0,1,t$, $\eta ' \in \{ \alpha ', \beta ' \} $ and assume that $-(\eta ' + \nu _0 +\nu _1 +\nu _t ) \in \Zint _{\geq 0}$.
Set $n= -(\eta  ' +\nu _0 +\nu _1 +\nu _t )$.
Then there exists a polynomial $P(q')$ of degree $n+1$ in the variable $q'$ such that if $q'$ satisfies $P(q')=0$ then there exists a solution of Eq.(\ref{Heun01}) written as $w^{\nu _0 }(w-1)^{\nu _1}(w-t)^{\nu _t} p(w)$, where $p(w)$ is a polynomial of degree no more than $n$ in the variable $w$. 
\end{prop}
\begin{exa} \label{exa:QES}
We investigate polynomial-type solutions of Heun's equation for the case $\epsilon '_0 - \beta ' + 2=0$.
Set $\nu _0= 0$, $\nu_1 = 1-\epsilon '_1$, $\nu_t = 1-\epsilon '_t$ and $\eta ' = \alpha ' $ in Proposition \ref{prop:findim0}.
Then $n= -(\alpha ' + 2- \epsilon '_1 - \epsilon '_t) = -(\epsilon '_0 -\beta ' +1)=1$.
We look for a solution of Eq.(\ref{Heun01}) of the form $(w-1)^{ 1-\epsilon '_1}(w-t)^{ 1-\epsilon '_t} (c+w)$.
By substituting it into Eq.(\ref{Heun01}), we have
\begin{align}
& c(q'-\beta ' \epsilon ' _1 t+\beta ' t+2\epsilon ' _1 t-2t-2+\beta '  -\beta ' \epsilon ' _t+2\epsilon ' _t) +2t-\beta '  t=0 , \\
& c(\beta ' -\epsilon ' _t -\epsilon ' _1 +2) +q' +\epsilon ' _1 t+2\beta ' t-2t-\beta ' \epsilon ' _1 t+2\beta '  -2-\beta '  \epsilon ' _t+\epsilon ' _t =0 . \nonumber
\end{align}
Hence 
\begin{align}
& c=\frac{q'-(\beta ' -1)((\epsilon ' _1-2)t+\epsilon ' _t-2)}{\epsilon ' _t +\epsilon ' _1 -\beta '  -2}, \label{eq:cqp} \\
& (q')^2+((-2\beta '  \epsilon ' _1+ 3\beta '  +3\epsilon ' _1-4)t+(-2\beta ' \epsilon ' _t+3\beta ' +3\epsilon ' _t -4))q' \label{eq:qpquad} \\
& \; + (\beta '  -2)[ (\epsilon ' _1-1)(\epsilon ' _1-2)(\beta '  -1)t^2t \nonumber \\
& \qquad + \{ (\beta ' -1) (2\epsilon ' _t\epsilon ' _1-3\epsilon ' _1-3\epsilon ' _t+5)-\epsilon ' _1-\epsilon ' _t +3 \} +(\epsilon ' _t-1)(\epsilon ' _t-2)(\beta '  -1)] =0. \nonumber
\end{align}
Therefore, if $q'$ satisfies the quadratic equation in Eq.(\ref{eq:qpquad}), then the function  $(w-1)^{ 1-\epsilon '_1}(w-t)^{ 1-\epsilon '_t} (c+w)$ satisfies Eq.(\ref{Heun01}) where $c$ is chosen as Eq.(\ref{eq:cqp}).
\end{exa}

We are going to obtain explicit expressions for solutions of Heun's equation which have an apparent singularity by using solutions which are expressed by quasi-solvability.

\begin{thm} \label{thm:exprnonlogsol}
Let $a,b,c$ be elements of $\{0,1,t \}$ such that $a \neq b \neq c \neq a$ and $\eta , \alpha ,\beta ,\epsilon _0 ,\epsilon _1 ,\epsilon _t,  \alpha ',\beta ',\epsilon '_0 ,\epsilon '_1 ,\epsilon '_t$ be the parameters defined in Eq.(\ref{eq:mualbe}) or Eq.(\ref{eq:mualbe18}).\\
(i) If $\epsilon _a \in \Zint _{\leq 0}$, $\alpha , \beta , \epsilon '_b , \epsilon '_c \not \in \Zint $ and the singularity $z=a$ of Eq.(\ref{Heun02}) is apparent, then there exists a non-zero solution of Eq.(\ref{Heun01}) which can be written as $(w-b) ^{1-\epsilon '_b}(w-c) ^{1-\epsilon '_c}h(w)$ where $h(w)$ is a polynomial of degree $-\epsilon _a $ and the functions
\begin{align}
& \int _{[\gamma _z ,\gamma _p]} (w-b) ^{1-\epsilon '_b}(w-c) ^{1-\epsilon '_c}h(w) (z-w)^{-\eta } dw , \quad (p=b,c),
\label{eq:intpbc}
\end{align}
are non-zero solutions of Eq.(\ref{Heun02}).\\
(ii) If $\alpha +\beta -2\eta \in \Zint _{\geq 1}$, $\eta ,\epsilon '_0 , \epsilon '_1 ,\epsilon '_t \not \in \Zint$ and the singularity $z=\infty $ of Eq.(\ref{Heun02}) is apparent, then there exists a non-zero solution of Eq.(\ref{Heun01}) which can be written as $w^{1-\epsilon ' _0} (w-1) ^{1-\epsilon '_1 }(w-t) ^{1-\epsilon ' _t }  h(w)$ where $h(w)$ is a polynomial of degree $\alpha +\beta -2\eta -1$ and the functions  
\begin{align}
& \int _{[\gamma _z ,\gamma _p]} w^{1-\epsilon ' _0} (w-1) ^{1-\epsilon '_1 }(w-t) ^{1-\epsilon ' _t }  h(w) (z-w)^{-\eta } dw , \quad (p=0,1,t),
\label{eq:intp01t}
\end{align}
are non-zero solutions of Eq.(\ref{Heun02}).
\end{thm}
\begin{proof}
By Theorem \ref{thm:nonlog-polynHeun} (ii) (resp. Theorem \ref{thm:nonlog-polynHeun} (vi)) and the duality of the parameters $(\alpha , \beta , \epsilon _0, \epsilon _1, \epsilon _t, \eta )$ and $(\alpha ', \beta ', \epsilon '_0, \epsilon '_1, \epsilon '_t, \eta ')$ in Eqs.(\ref{eq:mualbe}), (\ref{eq:mualbe18}), we obtain the existence of a non-zero solution of Eq.(\ref{Heun01}) which can be written as $(w-b) ^{1-\epsilon '_b}(w-c) ^{1-\epsilon '_c}h(w)$ (resp. $w^{1-\epsilon ' _0} (w-1) ^{1-\epsilon '_1 }(w-t) ^{1-\epsilon ' _t }  h(w)$) where $h(w)$ is a polynomial of degree $-\epsilon _a $ (resp. $\alpha +\beta -2\eta -1$).
It follows from Proposition \ref{prop:Heunint} that Eq.(\ref{eq:intpbc}) (resp. Eq.(\ref{eq:intp01t})) is a solution of Eq.(\ref{Heun02}).
We show that Eq.(\ref{eq:intpbc}) for $p=b$ is non-zero.
We expand $(w-c) ^{1-\epsilon '_c}h(w) $ about $w=b$ as $\sum _{j=0} ^{\infty } \tilde{c} _j (w-b)^j$.
Then there are infinitely many terms such that $\tilde{c} _j \neq 0$, because $\epsilon '_c \not \in \Zint$.
Eq.(\ref{eq:intpbc}) for $p=b$ is written as Eq.(\ref{eq:gzgpexpa}) for the case $\theta _b = 1- \epsilon ' _b \not \in \Zint _{\leq 0}$, $\theta _b +\kappa  +1 = 2- \epsilon '_b -\eta = 1- \epsilon _b \in \Zint _{\leq -1}$, and it is not identically zero.
It is also shown that Eq.(\ref{eq:intpbc}) for $p=c$ and Eq.(\ref{eq:intp01t}) for $p=0,1,t$ are not identically zero.\\
\end{proof}
\begin{exa}
We investigate solutions of Eq.(\ref{Heun02}) for the case $\epsilon _0 =-1$ and the singularity $z=0$ of Eq.(\ref{Heun02}) is apparent.
The condition that the singularity $z=0$ is apparent is described as an algebraic equation of $q$ by following the method in the appendix, and it is written as
\begin{equation}
q^2+(\epsilon _1 t+\epsilon _t-t-1)q+\alpha \beta t=0, \label{eq:e0-1nonlog}
\end{equation}
which is equivalent to Eq.(\ref{eq:qpquad}) by applying Eq.(\ref{eq:mualbe}) for $\eta =\beta $.
If Eq.(\ref{eq:e0-1nonlog}) is satified, then the function $(w-1)^{ 1-\epsilon '_1}(w-t)^{ 1-\epsilon '_t} (w+q/\alpha )$ satisfies Eq.(\ref{Heun01}), which follows from Example \ref{exa:QES}.
By applying the integral transformation, the functions
\begin{align}
& \int _{[\gamma _z ,\gamma _p]} (w-1) ^{1-\epsilon '_1 }(w-t) ^{1-\epsilon ' _t } \left( w+\frac{q}{\alpha } \right) (z-w)^{-\beta } dw ,
\end{align}
are solutions of Eq.(\ref{Heun02}), if $q$ satisfies Eq.(\ref{eq:e0-1nonlog}).
\end{exa}

If $\epsilon _a \in \Zint _{\geq 2}$, $\alpha ,\beta  \not \in \Zint$ and the singularity $z=a$ of Eq.(\ref{Heun02}) is apparent, then there exists a non-zero solution of Eq.(\ref{Heun01}) which can be written as $(w-a)^{1-\epsilon '_a} h(w)$ where $h(w)$ is a polynomial of degree $\epsilon _a -2$, and the functions
\begin{align}
& \int _{[\gamma _z ,\gamma _p]} (w-a)^{1-\epsilon '_a} h(w) (z-w)^{-\eta } dw , \quad (p=0,1,t,\infty ),
\label{eq:intwaezero}
\end{align}
are solutions of Eq.(\ref{Heun02}). But it is shown that Eq.(\ref{eq:intwaezero}) is identically zero for $p=0,1,t,\infty$. (For the case $p=a$, it follows from Eq.(\ref{eq:gzgpexpa}) on the case $\theta _p \not \in \Zint  $, $\theta _p +\kappa  \in \Zint _{\leq -2}$. For the case $p=b,c$, it follows from holomorphy of $(w-a)^{1-\epsilon '_a} h(w) $ about $p=b,c$. For the case $p=\infty $, it follows from $\gamma _0 \gamma _1 \gamma _t \gamma _{\infty }=1$.)
We have a similar situation for the case $\alpha +\beta -2\eta \in \Zint _{\leq -1}$, $\eta ,\epsilon '_0 \not \in \Zint$.
To obtain non-vanishing expressions of integrals, we apply the following proposition.
\begin{prop} \label{prop:Heunintother}
Let $\eta , \epsilon _0 ,\epsilon _1 ,\epsilon _t ,\epsilon '_0 ,\epsilon '_1 ,\epsilon '_t$ be the parameters defined in Eq.(\ref{eq:mualbe}) or Eq.(\ref{eq:mualbe18}) and $v(w)$ be a solution of Eq.(\ref{Heun01}).
Then the function 
\begin{align}
& y(z)=z^{1- \epsilon _0 } (z-1)^{1- \epsilon _1 } (z-t)^{1- \epsilon _t }\int _{[\gamma _z ,\gamma _p]} w^{\epsilon _0 '-1} (w-1)^{\epsilon _1 '-1} (w-t)^{\epsilon _t '-1} v(w) (z-w)^{\eta -2} dw 
\end{align}
is a solution of Eq.(\ref{Heun02}) for $p \in \{ 0,1,t,\infty \}$.
\end{prop}
\begin{proof}
Let $v(w)$ be a solution of Eq.(\ref{Heun01}).
Then the function $\tilde{v}(w)=  w^{\epsilon _0 '-1} (w-1)^{\epsilon _1 '-1} (w-t)^{\epsilon _t '-1} v(w)$ is a solution of 
\begin{align}
& \frac{d^2\tilde{v}}{dw^2} + \left( \frac{2-\epsilon _0 '}{w}+\frac{2-\epsilon _1'}{w-1}+\frac{2-\epsilon _t '}{w-t}\right) \frac{d\tilde{v}}{dw} +\frac{(2-\alpha ')(2- \beta ' )w -\tilde{q}'}{w(w-1)(w-t)} \tilde{v}=0, \\
& \tilde{q}'=q'-(\epsilon _0 '+\epsilon _t '-2)-(\epsilon _0 '+\epsilon _1 '-2)t . \nonumber 
\end{align}
It follows from Proposition \ref{prop:Heunint} that the function $\tilde{y}(z)= \int _{[\gamma _z ,\gamma _p]} \tilde{v}(w) (z-w)^{-(2-\eta )} dw $ $(p=0,1,t,\infty)$ is a solution of 
\begin{align}
& \frac{d^2\tilde{y}}{dz^2} + \left( \frac{2-\epsilon _0}{z}+\frac{2-\epsilon _1}{z-1}+\frac{2-\epsilon _t}{z-t}\right) \frac{d\tilde{y}}{dz} +\frac{(2-\alpha )(2-\beta )z -\tilde{q}}{z(z-1)(z-t)} \tilde{y}=0, \\
& \tilde{q}=\tilde{q}' +(1-\eta )\left\{ 2-\epsilon _t '+(2-\epsilon _1 ')t +(2-\epsilon _0 '-\eta ) (t+1) \right\} . \nonumber 
\end{align}
By setting $y(z)= z^{1- \epsilon _0 } (z-1)^{1- \epsilon _1 } (z-t)^{1- \epsilon _t }\tilde{y}(z)$, it follows that $y(z)$ is a solution of Eq.(\ref{Heun02}).
\end{proof}
\begin{thm} \label{thm:exprnonlogsol2}
Let $a,b,c$ be elements of $\{0,1,t \}$ such that $a \neq b \neq c \neq a$ and $\eta , \alpha ,\beta ,\epsilon _0 ,\epsilon _1 ,\epsilon _t,\epsilon '_0 ,\epsilon '_1 ,\epsilon '_t$ be the parameters defined in Eq.(\ref{eq:mualbe}) or Eq.(\ref{eq:mualbe18}).\\
(i) If $\epsilon _a \in \Zint _{\geq 2}$, $\alpha ,\beta ,\epsilon _b ' ,\epsilon _c ' \not \in \Zint$ and the singularity $z=a$ of Eq.(\ref{Heun02}) is apparent, then there exists a non-zero solution of Eq.(\ref{Heun01}) which can be written as $(w-a)^{1-\epsilon '_a} h(w)$ where $h(w)$ is a polynomial of degree $\epsilon _a -2$, and the functions
\begin{align}
& z^{1- \epsilon _0 } (z-1)^{1- \epsilon _1 } (z-t)^{1- \epsilon _t }\int _{[\gamma _z ,\gamma _p]} (w-b) ^{\epsilon '_b -1}(w-c) ^{\epsilon '_c -1}h(w) (z-w)^{\eta -2} dw , \quad (p=b,c),
\label{eq:intpbcD}
\end{align}
are non-zero solutions of Eq.(\ref{Heun02}).\\
(ii) If $\alpha +\beta -2\eta \in \Zint _{\leq -1}$, $\eta ,\epsilon '_0 , \epsilon '_1 ,\epsilon '_t \not \in \Zint$ and the singularity $z=\infty $ of Eq.(\ref{Heun02}) is apparent, then there exists a non-zero solution of Eq.(\ref{Heun01}) which can be written as $ h(w)$ where $h(w)$ is a polynomial of degree $2\eta -\alpha -\beta -1$ and the functions  
\begin{align}
& z^{1- \epsilon _0 } (z-1)^{1- \epsilon _1 } (z-t)^{1- \epsilon _t }\int _{[\gamma _z ,\gamma _p]} w^{\epsilon _0 '-1} (w-1)^{\epsilon _1 '-1} (w-t)^{\epsilon _t '-1} h(w) (z-w)^{\eta -2} dw ,
\label{eq:intp01tD}
\end{align}
$(p=0,1,t)$ are non-zero solutions of Eq.(\ref{Heun02}).
\end{thm}
\begin{proof}
By Theorem \ref{thm:nonlog-polynHeun} (i) (resp. Theorem \ref{thm:nonlog-polynHeun} (v)) and the duality of the parameters in Eqs.(\ref{eq:mualbe}), (\ref{eq:mualbe18}), we obtain the existence of a non-zero solution of Eq.(\ref{Heun01}) which can be written as $(w-a) ^{1-\epsilon '_a}h(w)$ (resp. $h(w)$) where $h(w)$ is a polynomial of degree $\epsilon _a -2$ (resp. $2\eta -\alpha -\beta -1$).
It follows from Proposition \ref{prop:Heunintother} that Eq.(\ref{eq:intpbcD}) (resp. Eq.(\ref{eq:intp01tD})) is a solution of Eq.(\ref{Heun02}).
It can be shown by a similar argument to that in the proof of Theorem \ref{thm:exprnonlogsol} that Eq.(\ref{eq:intpbcD}) for $p=b, c$ and Eq.(\ref{eq:intp01tD}) for $p=0,1,t$ are not identically zero.
\end{proof}

\section{Elliptical representation of Heun's equation} \label{sec:ell}

Heun's differential equation has an elliptical representation as we mentioned in the introduction.
In this section, we rewrite several results on the integral transformation of Heun's equation to the elliptical representation form.

We review the elliptical representation of Heun's differential equation.
Set
\begin{equation}
H^{(l'_0,l'_1,l'_2,l'_3)}= -\frac{d^2}{dx^2} + \sum_{i=0}^3 l'_i(l'_i+1)\wp (x+\omega_i).
\label{Ino}
\end{equation}
Let $\alpha '_i$ be a number such that $\alpha '_i= -l'_i$ or $\alpha '_i= l'_i+1$ for each $i\in \{ 0,1,2,3\} $.
By setting 
\begin{align}
& z=\frac{\wp (x) -e_1}{e_2-e_1}, \quad  t=\frac{e_3-e_1}{e_2-e_1}, \quad f(x)= v z^{\frac{\alpha '_1}{2}}(z-1)^{\frac{\alpha '_2}{2}}(z-t)^{\frac{\alpha '_3}{2}};
\label{eq:zxtrans}
\end{align}
the equation
\begin{equation}
\frac{d^2v}{dz^2} + \left( \frac{\epsilon _0 '}{z}+\frac{\epsilon _1'}{z-1}+\frac{\epsilon _t '}{z-t}\right) \frac{dv}{dz} +\frac{\alpha ' \beta ' z -q'}{z(z-1)(z-t)} v=0.
\label{Heun01-}
\end{equation}
is transformed to
\begin{equation}
H^{(l'_0,l'_1,l'_2,l'_3)} f(x) = E' f(x),
\label{InoEF0+}
\end{equation}
where
\begin{align}
& \{ \alpha ' , \: \beta ' \} = \{ (\alpha '_1 +\alpha '_2 +\alpha '_3 +\alpha '_0)/2, \: (\alpha '_1 +\alpha '_2 +\alpha '_3 +1 - \alpha '_0)/2 \} , \label{eq:zxtransparam} \\
& \epsilon '_0= \alpha '_1+1/2, \quad \epsilon '_1= \alpha '_2 +1/2, \quad \epsilon '_t =\alpha '_3 +1/2, \nonumber \\
& q'=- E'/(4(e_2-e_1))+ ( -(\alpha '-\beta ')^2 +2(\epsilon '_0 )^2  -4\epsilon '_0 +1) (t+1)/12\nonumber \\
& \quad +
(6\epsilon '_0 \epsilon '_t +2(\epsilon '_t )^2 -4\epsilon '_t  -(\epsilon '_1)^2 +2\epsilon '_1)/12 +(6\epsilon '_0 \epsilon '_1+2 (\epsilon '_1 )^2  -4\epsilon '_1-(\epsilon '_t) ^2+2\epsilon '_t)t/12. \nonumber
\end{align}

We investigate a correspondence of cycles on the Riemann sphere and the torus.
For the transformation $z=(\wp (x) -e_1)/(e_2-e_1)$, the path from $x$ to $-x$ (resp. $-x+2\omega _1$, $-x+2\omega _2$, $-x+2\omega _3$) which traces a semicircle around $\omega _0$ (resp. $\omega _1$, $\omega _2$, $\omega _3$) 
corresponds to a cycle which surrounds $\infty $ (resp. $0$, $1$, $t$) on the Riemann sphere $\Cplx \cup \{ \infty \}$ whose coordinate is $z$.
Let $\gamma _0$, (resp. $\gamma _1$, $\gamma _t$, $\gamma _{\infty}$) be a cycle on the Riemann sphere which surrounds the point $z=0$ (resp. $z=1$, $z=t$, $z=\infty $) anticlockwise.
We choose the cycles so that $\gamma _0 \gamma _1 \gamma _t \gamma _{\infty} \sim \mbox{id}$.
Then the shift of the period $x \rightarrow x+2\omega _1$  corresponds to a cycle which is homotopic to $\gamma _t \gamma _1 $, $\gamma _1 \gamma _t $, $\gamma _t ^{-1} \gamma _1 ^{-1}$ or $\gamma _1 ^{-1} \gamma _t ^{-1}$ on the punctured Riemann sphere, whose choice is dependent on specifying the point $x$ and the zone where the shift $x \rightarrow x+2\omega _1 $ passes (see Figure 2).
\begin{center}
\begin{picture}(400,60)(0,0)
\put(30,40){\circle*{3}}
\put(90,40){\circle*{3}}
\put(150,40){\circle*{3}}
\put(25,48){$\omega _3$}
\put(75,48){$\omega _1 +\omega _3$}
\put(125,48){$2\omega _1 +\omega _3$}
\put(30,40){\oval(20,20)[b]}
\put(90,40){\oval(20,20)[b]}
\put(40,40){\line(1,0){40}}
\put(100,40){\line(1,0){40}}
\put(60,40){\vector(1,0){1}}
\put(120,40){\vector(1,0){1}}
\put(31,30){\vector(1,0){1}}
\put(91,30){\vector(1,0){1}}
\put(30,10){\circle*{3}}
\put(90,10){\circle*{3}}
\put(150,10){\circle*{3}}
\put(25,0){$0(=\omega _0)$}
\put(85,0){$\omega _1$}
\put(145,0){$2\omega _1$}
\qbezier(17,37)(20,40)(23,43)
\qbezier(23,37)(20,40)(17,43)
\qbezier(137,37)(140,40)(143,43)
\qbezier(143,37)(140,40)(137,43)
\put(180,50){\line(0,1){10}}
\put(180,50){\line(1,0){10}}
\put(183,53){$x$}

\put(200,25){$\Rightarrow$}
\put(180,10){$z=\frac{\wp (x) -e_1}{e_2-e_1}$}

\put(280,20){\circle*{3}}
\put(370,20){\circle*{3}}
\put(278,10){$t $}
\put(368,10){$1$}
\qbezier(295,22)(295,35)(280,35)
\qbezier(265,20)(265,35)(280,35)
\qbezier(295,18)(295,5)(280,5)
\qbezier(265,20)(265,5)(280,5)
\qbezier(355,22)(355,35)(370,35)
\qbezier(385,20)(385,35)(370,35)
\qbezier(355,18)(355,5)(370,5)
\qbezier(385,20)(385,5)(370,5)
\put(295,22){\line(1,0){60}}
\put(295,18){\line(1,0){60}}
\qbezier(292,19)(295,22)(298,24)
\qbezier(298,19)(295,22)(292,24)
\put(265,18){\vector(0,-1){1}}
\put(385,22){\vector(0,1){1}}
\put(320,18){\vector(1,0){1}}
\put(330,22){\vector(-1,0){1}}
\put(290,35){$\gamma _t$}
\put(347,35){$\gamma _1$}
\put(390,45){\line(0,1){10}}
\put(390,45){\line(1,0){10}}
\put(393,48){$z$}
\end{picture}
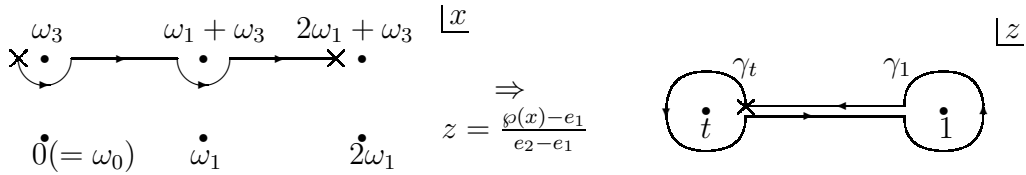
Figure 2. Correspondence of cycles.
\end{center}
It is also shown that the shift of the period $x \rightarrow x+2\omega _3$  corresponds to the cycle which is homotopic to $\gamma _0 \gamma _1 $, $\gamma _1 \gamma _0 $, $\gamma _0 ^{-1} \gamma _1 ^{-1}$ or $\gamma _1 ^{-1} \gamma _0 ^{-1}$ on the punctured Riemann sphere, whose choice is dependent on specifying the point $x$ and the zone where the shift $x \rightarrow x+2\omega _3 $ passes.

We rewrite the integral transformation of Heun's equation (i.e. Proposition \ref{prop:Heunint}) in elliptical representation form, which was announced in \cite{TakID}.
It is remarkable that the eigenvalue $E$ is unchanged by the integral transformation.
\begin{thm} \label{thm:ellipinttras0}
Let $\sigma (x)$ be the Weierstrass sigma function, $\sigma _i (x)$ $(i=1,2,3)$ be the Weierstrass co-sigma function which has a zero at $x=\omega _i$, and $I_i$ $(i=0,1,2,3)$ be the cycle on the complex plane with the variable $\xi $ such that points $\xi =x$ and $\xi =-x +2\omega _i$ are contained and the half-periods $\Zint \omega _1 +\Zint \omega _3$ are not contained inside the cycle.
Let $\alpha '_i$ be a number such that $\alpha' _i= -l'_i$ or $\alpha '_i= l'_i+1$ for each $i\in \{ 0,1,2,3\} $.
Set $d=-\sum_{i=0}^3 \alpha '_i /2$ and $\eta =d+2$.
If $\tilde{f}(x)$ satisfies 
\begin{equation}
H^{(l'_0,l'_1,l'_2,l'_3)} \tilde{f}(x) =E \tilde{f}(x), \label{InoEF00}
\end{equation} 
then the functions 
\begin{align}
& f(x)=\sigma (x) ^{\alpha '_0 +d+1 } \sigma _1 (x) ^{\alpha '_1 +d+1 } \sigma _2 (x) ^{\alpha '_2 +d+1} \sigma _3 (x) ^{\alpha '_3 +d+1} \cdot \label{eq:Hintell2} \\
& \quad \quad \int _{I_i } \tilde{f} (\xi ) \sigma (\xi ) ^{1 -\alpha '_0}  \sigma _1 (\xi ) ^{1 -\alpha '_1} \sigma _2 (\xi ) ^{1 -\alpha '_2} \sigma _3 (\xi ) ^{1 -\alpha' _3} (\sigma (x+\xi ) \sigma (x-\xi ) )^{-\eta } d\xi \nonumber 
\end{align}
$(i \in \{ 0,1,2,3 \})$ satisfy
\begin{equation}
H^{(\alpha '_0 +d ,\alpha '_1 +d ,\alpha '_2 +d ,\alpha '_3 +d )}  f(x) =E f(x). 
\label{eq:HalfEf}
\end{equation}
\end{thm}
\begin{proof}
Let $\tilde{f}(x)$ be a solution of $H^{(l'_0,l'_1,l'_2,l'_3)} \tilde{f}(x) =E' \tilde{f}(x) $.
By the transformation given by Eq.(\ref{eq:zxtrans}), the function $ \underline{f}(w) = \tilde{f}(\tilde{\wp}^{-1}(w)) w^{\frac{-\alpha '_1}{2}}(w-1)^{\frac{-\alpha '_2}{2}}(w-t)^{\frac{-\alpha '_3}{2}}$ is a solution of Eq.(\ref{Heun01-}) where $\tilde{\wp}^{-1}(w) $ is the inverse function of $w= \tilde{\wp }(\xi ) = (\wp (\xi ) -e_1)/(e_2-e_1 ) $, the parameters are given by Eq.(\ref{eq:zxtransparam}) and we choose $\beta '= (\alpha '_1 +\alpha '_2 +\alpha '_3 +\alpha '_0)/2=-d $.
Next we apply Proposition \ref{prop:Heunint} with the parameter $\eta = 2-\beta '$.
Then the functions $\int _{[\gamma _z ,\gamma _p]}  \underline{f}(w) (z-w)^{-\eta } dw  $ $(p=0,1,t,\infty )$ are solutions of Eq.(\ref{Heun02}) and we have $\epsilon _0  = \epsilon _0 '-\eta '+1 = \alpha '_1+ d+ 3/2$, $\epsilon _1 = \alpha '_2+ d+ 3/2$, $\epsilon _t  = \alpha '_3+ d+ 3/2$ and $ \{ \alpha , \beta \} = \{ 2-\beta ' , -\alpha '+\beta ' +1 \} =  \{ 2+d , \alpha _0 '+1/2 \}$.
The value $q$ is expressed in term of $E'$ and other parameters.
We set $\alpha _i= \alpha '_i+ d +1 (\in \{ -(\alpha '_i+ d ),  \alpha '_i+ d +1 \} ) $ $(i=0,1,2,3)$ and transform to the elliptical form by Eqs.(\ref{eq:zxtrans}), (\ref{eq:zxtransparam}) where the prime $(')$ is omitted.
It is shown by a direct calculation that the value $E$ coincides with the original value $E'$, and the functions
\begin{align}
& f(x)= z^{\frac{\alpha '_1+ d +1}{2}}(z-1)^{\frac{\alpha '_2+ d +1}{2}}(z-t)^{\frac{\alpha '_3+ d +1}{2}} \cdot \label{eq:Hintell0} \\
& \qquad \int _{[\gamma _z ,\gamma _p]} \tilde{f}(\tilde{\wp}^{-1}(w)) w^{\frac{-\alpha '_1}{2}}(w-1)^{\frac{-\alpha '_2}{2}}(w-t)^{\frac{-\alpha '_3}{2}}  (z-w)^{-\eta } dw \nonumber
\end{align}
are solutions of $H^{(\alpha '_0 +d ,\alpha '_1 +d ,\alpha '_2 +d ,\alpha '_3 +d )}  f(x) =E' f(x)$ for $p \in \{ 0,1,t,\infty \}$, where $z=(\wp (x) -e_1)/(e_2-e_1)$.
For the transformation $w=(\wp (\xi ) -e_1)/(e_2-e_1)$, the cycles $[\gamma _z ,\gamma _{\infty}]$, $[\gamma _z ,\gamma _{0}]$, $[\gamma _z ,\gamma _{1}]$, $[\gamma _z ,\gamma _{t}]$ correspond to the cycles $I_0$, $I_1$, $I_2$, $I_3$.
By changing the variable as $w=(\wp (\xi ) -e_1)/(e_2-e_1)$ in Eq.(\ref{eq:Hintell0}) and applying the relations $\sqrt{\wp (\xi ) -e_i} =\sigma _i(\xi )/\sigma (\xi )$ $(i=1,2,3)$, $\wp (x)- \wp (\xi )= -\sigma (x+\xi ) \sigma (x-\xi ) / (\sigma (x) \sigma (\xi ))^2$ and $\wp'(\xi )=-2\sigma _1 (\xi ) \sigma _2 (\xi ) \sigma _3 (\xi ) /\sigma (\xi ) ^3$,
we obtain the proposition.
\end{proof}
\begin{prop} \label{prop:ellipinttras}
Set 
\begin{align}
& \alpha _0 \in \{ -l_0, l_0 +1 \} , \; l'_0=\frac{-\alpha _0-l_1-l_2-l_3}{2}-1, \; l'_1=\frac{\alpha _0+l_1-l_2-l_3}{2}-1,  \label{eq:l'alpha} \\
& l'_2=\frac{\alpha _0 -l_1 +l_2-l_3}{2}-1, \; l'_3=\frac{\alpha _0 -l_1-l_2+l_3}{2}-1, \; \eta =\frac{\alpha _0 -l_1-l_2-l_3}{2} . \nonumber 
\end{align}
If $\tilde{f}(x)$ satisfies $H^{(l'_0,l'_1,l'_2,l'_3)} \tilde{f}(x) =E \tilde{f}(x)$, then the functions
\begin{align}
& f(x)=\sigma (x) ^{\alpha _0} \sigma _1 (x) ^{-l _1} \sigma _2 (x) ^{-l _2} \sigma _3 (x) ^{-l _3} \cdot \label{eq:Hintell} \\
& \; \int _{I_i } \tilde{f}(\xi ) \sigma (\xi ) ^{l' _0 +1}  \sigma _1 (\xi ) ^{l' _1 +1} \sigma _2 (\xi ) ^{l' _2 +1} \sigma _3 (\xi ) ^{l' _3 +1} (\sigma (x+\xi ) \sigma (x-\xi ) )^{-\eta } d\xi \nonumber 
\end{align}
($i \in \{ 0,1,2,3 \}$) satisfy 
\begin{equation}
H^{(l_0,l_1,l_2,l_3)} f(x) =E f(x) .
\label{InoEF}
\end{equation}
\end{prop}
\begin{proof}
We obtain the proposition by applying Theorem \ref{thm:ellipinttras0} for $\alpha '_0= 1+ (\alpha _0+l_1+l_2+l_3)/2$,  $\alpha '_1= 1+(-\alpha _0-l_1+l_2+l_3)/2$, $\alpha '_2= 1+(-\alpha _0 +l_1 -l_2+l_3)/2$, $\alpha '_3= 1+(-\alpha _0 +l_1+l_2-l_3)/2$.
Note that $H^{(-l_0-1, -l_1-1, -l_2-1, -l_3-1 )}=H^{(l_0, l_1, l_2, l_3 )}$.
\end{proof}

We review an aspect of the monodromy of a differential equation with periodic potential. 
Let $q(x)$ be a periodic function with a period $T$ and $\{ f_ 1(x)$, $f_2 (x) \}$ be a basis of solutions of the differential equation
\begin{equation} 
\left( -\frac{d^2}{dx^2} +q(x)  \right) f(x)= Ef(x).
\end{equation}
Then $ f_ 1(x+T)$ and $f_2 (x+T)$ are also solutions.
Let $M_T$ be a monodromy matrix for the shift $x \rightarrow x+T$ with respect to the basis $\{ f_ 1(x)$, $f_2 (x) \}$, i.e.
\begin{equation}
(f_ 1(x+T), f_2 (x+T)) = (f_ 1(x), f_2 (x))M_T= (f_ 1(x), f_2 (x))
\left(
\begin{array}{ll}
m_{11} & m_{12} \\
m_{21} & m_{22} 
\end{array}
\right) .
\end{equation}
Then we have $\det M_T= 1$.
\begin{proof}
Since $f'' _i(x) = (q(x) -E)f_i(x)$ $(i=1,2)$, we have $(f'_1(x)f_2(x)-f_1(x)f'_2(x))'=0$. Hence $f'_1(x)f_2(x)-f_1(x)f'_2(x) =C$ for some constant $C\neq 0$ which follows from the linear independence of $f_ 1(x), f_2 (x)$.
We have 
\begin{align}
& C= f'_1(x+T)f_2(x+T)-f_1(x+T)f'_2(x+T) \\
& =(m_{11} m_{22} -m_{12} m_{21})(f'_1(x)f_2(x)-f_1(x)f'_2(x)) =(m_{11} m_{22} -m_{12} m_{21}) C. \nonumber
\end{align}
Hence $\det M_T= 1$.
\end{proof}
Note that this situation is applicable to the elliptical representation of Heun's equation by setting $T= 2\omega _1$ or $2\omega _3 $ (or any period of the elliptic function $\wp (x)$).
If $\mbox {tr}M_T >2 $ or $\mbox {tr}M_T <-2 $ (resp. $-2<\mbox {tr}M_T <2 $), then there exists a basis of solutions $f_+(x) , f_-(x) $ such that $f _{\pm }(x+T) =e^{\pm \nu} f _{\pm }(x)$ (resp. $f _{\pm }(x+T) =e^{\pm \sqrt{-1} \nu} f _{\pm }(x))$ for some $\nu \in \Rea$ such that $e^{2 \nu} - (\mbox {tr}M_T )e^{ \nu} +1=0$ (resp. $e^{2 \sqrt{-1} \nu} - (\mbox {tr}M_T) e^{ \sqrt{-1} \nu} +1=0$).
If $\mbox {tr}M_T =2 $ (resp. $\mbox {tr}M_T =-2 $), then there exists a non-zero periodic (anti-periodic) solution, i.e. a solution $f(x)$ such that $f(x+T)=f(x)$ (resp. $f(x+T)=-f(x)$).
It does not simply follow from $\mbox {tr}M_T =2 $ (resp. $\mbox {tr}M_T =-2 $) that every solution is periodic (resp. anti-periodic).
Whether this is the case is determined by the Jordan normal form of $M_T$.

\begin{thm}  \label{thm:pp0}
Let $k \in \{1,3\}$ and $M_{2\omega _k} ^{(l_0,l_1,l_2,l_3)}(E)$ be the monodromy matrix by the shift of the period $x \rightarrow x+2\omega _k$ with respect to a certain basis of solutions to $H^{(l_0,l_1,l_2,l_3)}f(x)= E f(x)$.
Let $\alpha '_i$ be a number such that $\alpha '_i= -l'_i$ or $\alpha '_i= l'_i+1$ for each $i\in \{ 0,1,2,3\} $ and set $d=-\sum_{i=0}^3 \alpha '_i /2$.
Then 
\begin{equation}
\mbox{\rm{tr}}M_{2\omega _k} ^{(l'_0,l'_1,l'_2,l'_3)} (E)= \mbox{\rm{tr}}M _{2\omega _k} ^{(\alpha '_0 +d ,\alpha '_1 +d ,\alpha '_2 +d ,\alpha '_3 +d )}(E).
\label{eq:trM}
\end{equation}
\end{thm}
\begin{proof}
We prove the case $k=1$ such that the shift of the period $x \rightarrow x+2\omega _1$  corresponds to a cycle which is homotopic to $\gamma _t \gamma _1 $.
Let $\tilde{f}(x)$ (resp. $f(x)$) be a solution of Eq.(\ref{InoEF00}) (resp. Eq.(\ref{eq:HalfEf})).
Then the function 
$\tilde{f}(\tilde{\wp}^{-1}(w)) w^{-\alpha '_1/2}(w-1)^{-\alpha '_2/2}(w-t)^{-\alpha '_3/2}$ (resp. $f(\tilde{\wp}^{-1}(z)) z^{-(\alpha ' 1+d+1)/2}(z-1)^{-(\alpha ' 2+d+1)/2}(z-t)^{-(\alpha ' 3+d+1)/2}$) is a solution of Eq.(\ref{Heun01}) (resp. Eq.(\ref{Heun02})).
Thus we have $\exp (-2\pi \sqrt{-1} (\alpha ' _2+\alpha ' _3 )/2 ) \mbox{\rm{tr}} M_{2\omega _1} ^{(l'_0,l'_1,l'_2,l'_3)} (E) =  \mbox{\rm{tr}} M'^{(t)} M'^{(1)}$ and $\exp (-2\pi \sqrt{-1} (\alpha ' _2+\alpha ' _3 +2d+2)/2 ) \mbox{\rm{tr}} M _{2\omega _1} ^{(\alpha '_0 +d ,\alpha '_1 +d ,\alpha '_2 +d ,\alpha '_3 +d )} (E) = \mbox{\rm{tr}} M^{(t)} M^{(1)}$.
It follows from Theorem \ref{thm:monodm} that $\mbox{\rm tr} (M^{(t)} M^{(1)}) =\exp (- 2\pi \sqrt{-1} \eta ) \mbox{\rm tr} (M'^{(t)}M'^{(1)})$.
Combining these relations with the relation $\eta =d+2$ in the proof of Theorem \ref{thm:ellipinttras0}, we obtain Eq.(\ref{eq:trM}).
The other cases can be proved similarly.
\end{proof}
\begin{cor} \label{thmcor:pp}
Assume that the parameters $l_0$, $l_1$, $l_2$, $l_3$, $l'_0$, $l'_1$, $l'_2$, $l'_3$ satisfy Eq.(\ref{eq:l'alpha}).
Let $k \in \{1,3\}$.
Then 
\begin{equation}
\mbox{\rm{tr}}M_{2\omega _k} ^{(l'_0,l'_1,l'_2,l'_3)} (E)= \mbox{\rm{tr}}M _{2\omega _k} ^{(l_0 ,l_1 ,l_2  ,l_3  )}(E).
\label{eq:trM0}
\end{equation}
\end{cor}
\begin{cor}  \label{cor:pp0}
We keep the notations in Theorem \ref{thm:pp0}.
Let $k \in \{1,3\}$.
If there exists a non-zero solution $\tilde{f}(x,E)$ of $(H^{(l'_0,l'_1,l'_2,l'_3)} -E) \tilde{f}(x,E)=0$ such that $\tilde{f}(x+2\omega _k,E) =C_k(E) \tilde{f}(x,E) $, then there exists a non-zero solution $f(x,E)$ of $(H^{(\alpha '_0 +d ,\alpha '_1 +d ,\alpha '_2 +d ,\alpha '_3 +d )}-E) f(x,E)=0$ such that $f(x+2\omega _k,E) =C_k(E) f(x,E)$.
In other word, periodicity is preserved by the integral transformation.
\end{cor}
\begin{proof}
Let $t'_k$ (resp. $t_k$) be a solution of the quadratic equation $(t'_k)^2-\mbox{\rm{tr}}M_{2\omega _k} ^{(l'_0,l'_1,l'_2,l'_3)} (E)t'_k +1=0$ (resp. $t_k^2-\mbox{\rm{tr}}M _{2\omega _k} ^{(\alpha '_0 +d ,\alpha '_1 +d ,\alpha '_2 +d ,\alpha '_3 +d )}(E)t_k +1=0$).
Since $\det M_{2\omega _k} ^{(l'_0,l'_1,l'_2,l'_3)} (E) =\det M _{2\omega _k} ^{(\alpha '_0 +d ,\alpha '_1 +d ,\alpha '_2 +d ,\alpha '_3 +d )}(E) =1$, the value $t'_k$ (resp. $t_k$) is an eigenvalue of the monodromy matrix $M_{2\omega _k} ^{(l'_0,l'_1,l'_2,l'_3)} (E)$ (resp. $M _{2\omega _k} ^{(\alpha '_0 +d ,\alpha '_1 +d ,\alpha '_2 +d ,\alpha '_3 +d )}(E)$).
Thus Corollary \ref{cor:pp} follows from $\mbox{\rm{tr}}M_{2\omega _k} ^{(l'_0,l'_1,l'_2,l'_3)} (E)= \mbox{\rm{tr}}M _{2\omega _k} ^{(\alpha '_0 +d ,\alpha '_1 +d ,\alpha '_2 +d ,\alpha '_3 +d )}(E) $.
\end{proof}
\begin{cor} \label{cor:pp}
Assume that the parameters $l_0$, $l_1$, $l_2$, $l_3$, $l'_0$, $l'_1$, $l'_2$, $l'_3$ satisfy Eq.(\ref{eq:l'alpha}).
Let $k \in \{1,3\}$.
If there exists a non-zero solution $\tilde{f}(x,E)$ of  $(H^{(l'_0,l'_1,l'_2,l'_3)} -E) \tilde{f}(x,E)=0$ such that $\tilde{f}(x+2\omega _k,E) =C_k(E) \tilde{f}(x,E) $ such that $\tilde{f}(x+2\omega _k,E) =C_k(E) \tilde{f}(x,E)$, then there exists a non-zero solution $f(x,E)$ of Eq.(\ref{InoEF}) such that $f(x+2\omega _k,E) =C_k(E) f(x,E)$.
\end{cor}
If $\omega _1 \in \Rea _{\neq 0}$ and $\omega _3 \in \sqrt{-1}\Rea _{\neq 0}$, then the potential $\sum_{i=0}^3 l_i(l_i+1)\wp (x+\omega _i)$ in Eq.(\ref{InoEF0}) is real-valued for $x \in \Rea$.
From the viewpoint of quantum mechanics, we are interested in finding square-integrable eigenstates in a suitable Hilbert space for the elliptical representation of Heun's equation, and periodicity with respect to the shift $x \rightarrow x+2\omega _1$ is related to square-integrable eigenstates (see \cite{Tak2,Tak3}).
Ruijsenaars \cite{Rui} established that the spectrum of Eq.(\ref{InoEF00}) coincides with that of Eq.(\ref{InoEF}) by investigating a certain Hilbert-Schmidt operator.
Theorem \ref{thm:pp0} can be regarded as a complex-functional version of Ruijsenaars' result.
Khare and Sukhatme \cite{KhS} earlier made a conjecture about correspondences between quasi-solvable solutions of Eq.(\ref{InoEF00}) and those of Eq.(\ref{InoEF}), and Corollary \ref{cor:pp0} gives an approach for a reformulation of their conjecture in terms of monodromy.

For elliptical representations, quasi-solvability is described as follows:
\begin{prop} \label{findim} $($\cite[Proposition 5.1]{Tak2}$)$
Let $\beta '_i$ be a number such that $\beta '_i= -l'_i$ or $\beta '_i= l'_i+1$ for each $i\in \{ 0,1,2,3\} $, and set $\tilde{d}=-\sum_{i=0}^3 \beta '_i /2$.
Suppose that $\tilde{d} \in \Zint_{\geq 0}$, and let $V_{\beta '_0, \beta '_1, \beta '_2, \beta '_3}$ be the $\tilde{d}+1$-th dimensional space spanned by 
\begin{equation}
\left\{ \widehat{\Phi}(\wp (x)) \wp(x)^n\right\} _{n=0, \dots ,\tilde{d}}, \label{eq:Vqsol}
\end{equation} 
where $\widehat{\Phi}(z)=(z-e_1)^{\beta '_1/2}(z-e_2)^{\beta '_2/2}(z-e_3)^{\beta '_3/2}$.
Then the operator $H^{(l'_0,l'_1,l'_2, l'_3)}$ preserves the space $V_{\beta '_0, \beta '_1, \beta '_2, \beta '_3}$.
\end{prop}
To find eigenvalues of the operator $H^{(l'_0,l'_1,l'_2, l'_3)}$ on the space $V_{\beta '_0, \beta '_1, \beta '_2, \beta '_3}$, we obtain an algebraic equation of order $\tilde{d}+1$ in the variable $E$, which is related to $ P(q')=0$ in Proposition \ref{prop:findim0} for the case $\nu _0 = (\beta '_1 -\alpha '_1)/2$, $\nu _1 = (\beta '_2 -\alpha '_2)/2$, $\nu _t = (\beta '_3 -\alpha '_3)/2$,
where $\alpha '_i \in \{  -l'_i , l'_i+1 \}$ and the transformation between Proposition \ref{prop:findim0} and Proposition \ref{findim} is determined by Eq.(\ref{eq:zxtrans}).
The eigenvector corresponding to the eigenvalue $E$ can be written as a product of $\widehat{\Phi}(\wp (x)) $ and the polynomial in the variable $\wp (x)$ of degree no more than $\tilde{d}$.

Since the functions $\wp (x+2\omega _i)$ $(i=0,1,2,3)$ are even and doubly periodic, the solutions of Eq.(\ref{InoEF}) about $x=\omega _i$ $(i=0,1,2,3)$ can be expanded as
\begin{align}
& f (x)= \left\{ 
\begin{array}{ll}
\displaystyle C ^{\langle i \rangle} \sum _{j=0} ^{\infty } c^{(i)} _j (x-\omega _i)^{-l_i+ 2j}  +D ^{\langle i \rangle} \sum _{j=0} ^{\infty } \tilde{c}^{(i)} _j (x-\omega _i)^{l_i +1 + 2j}, & l_i \not \in 1/2 +\Zint , \\
\displaystyle C ^{\langle i \rangle} \sum _{j=0} ^{\infty } c^{(i)} _j (x-\omega _i)^{|l_i +1/2|+ 1/2+ 2j}  +D ^{\langle i \rangle} \cdot & l_i  \in 1/2 +\Zint , \\
\displaystyle  \quad \left( \sum _{j=0} ^{\infty } \tilde{c}^{(i)} _j (x-\omega _i)^{-|l_i +1/2| + 1/2 +2j} \right. \left. +A^{\langle i \rangle} \sum _{j=0} ^{\infty } c^{(i)} _j (x-\omega _i)^{|l_i +1/2|+1/2+2j} \right) ,  \!  \! \! \! \! \! \! \! \! \! \! \! \! \! \! \! \! \! \! \! \! \! \! \! 
\end{array}
\right.
\label{eq:expell}
\end{align}
where $C ^{\langle i \rangle}$ and $D ^{\langle i \rangle}$ are constants, $c^{(i)} _0= \tilde{c}^{(i)} _0=1$, and $c^{(i)} _j$ and $\tilde{c}^{(i)} _j$  $(j=1,2,\dots )$ are determined recursively.
If $l_i  \in 1/2 + \Zint $ and $A^{\langle i \rangle} \neq 0$ (resp. $A^{\langle i \rangle} = 0)$, then the singularity $x=\omega _i$ is non-apparent (resp. apparent).
Note that if $l_i= -1/2$, then the singularity $x=\omega _i$ is always logarithmic, i.e. $A^{\langle i \rangle} \neq 0$.
By the transformation given by Eq.(\ref{eq:zxtrans}), the condition that $l_0 \in 1/2 + \Zint $ (resp. $l_1 \in 1/2 +\Zint $, $l_2 \in 1/2 +\Zint $, $l_3 \in 1/2 +\Zint $) and the singularity $x=0 $ (resp. $x=\omega _1$, $x=\omega _2$, $x=\omega _3$) is (non-)apparent is equivalent to that $\alpha -\beta \in \Zint$ (resp. $\epsilon _0 \in \Zint$, $\epsilon _1 \in \Zint$, $\epsilon _t \in \Zint$) and the singularity $z=\infty $ (resp. $z=0$, $z=1$, $z=t$) is (non-)apparent.
The condition that the singularity $x=\omega _i$ $(i\in \{0,1,2,3 \})$ is apparent (i.e. $A^{\langle i \rangle} = 0$) for the case $l_i  \in -1/2 + \Zint _{\neq 0}$ is described as follows:
Set $j_0= -|l_i +1/2| +1/2$, $\tilde{c}^{(i)} _0=1$, $f(x)= \sum _{j=0} ^{\infty } \tilde{c}^{(i)} _j (x-\omega _i)^{j_0 +2j}$.
By substituting $f(x)$ into Eq.(\ref{InoEF}) and expanding Eq.(\ref{InoEF}) as a series in $x- \omega _i$, we obtain an equation for $\tilde{c}^{(i)} _0, \tilde{c}^{(i)} _1, \dots \tilde{c}^{(i)} _{j}$ for the coefficients of $(x-\omega _i)^{j_0 +2j-2}$.
We determine $\tilde{c}^{(i)} _{j'}$ $(j'=1,\dots , |l_i +1/2| -1 )$ by solving the equation for $\tilde{c}^{(i)} _0, \tilde{c}^{(i)} _1, \dots \tilde{c}^{(i)} _{j'}$ recursively for each $j'$ and we have $\deg _E \tilde{c}^{(i)} _{j'} =j'$.
For the coefficient of $(x-\omega _i)^{j_0 + |2l_i +1|-2}$, the term concerned with $\tilde{c}^{(i)} _{|l_i +1/2|} $ disappears and we have an algebraic equation of degree $|l_i +1/2|$ with respect to the variable $E$, which we denote by $P^{(i)}(E)=0$, where $P^{(i)}(E)$ is monic.
Then the condition that the singularity $x=\omega _i$ $(i\in \{0,1,2,3 \})$ is apparent is equivalent to the eigenvalue $E$ satisfying $P^{(i)}(E)=0$.
The following proposition can be proved by rewriting Theorem \ref{thm:nonlog-polynHeun} in its elliptical form.
\begin{prop} \label{prop:nonlog-polynell}
Let $\alpha '_i$ be a number such that $\alpha' _i= -l'_i$ or $\alpha '_i= l'_i+1$ for each $i\in \{ 0,1,2,3\} $.
Set $d=-\sum_{i=0}^3 \alpha '_i /2$.\\
(i) If $\alpha '_0 \in 3/2 + \Zint _{\geq 0}$ (resp. $\alpha ' _1 \in 3/2 +\Zint _{\geq 0}$, $\alpha ' _2 \in 3/2 +\Zint _{\geq 0}$, $\alpha ' _3 \in 3/2 +\Zint _{\geq 0}$), $d \not \in \Zint$ and the singularity $x=0$ (resp. $x=\omega _1$, $x=\omega _2$, $x=\omega _3$) of Eq.(\ref{InoEF00}) is apparent, then there exists a non-zero solution of Eq.(\ref{eq:HalfEf})
 which belongs to the space $V_{-\alpha '_0 -d, \alpha '_1 +d + 1,  \alpha '_2 +d + 1 ,  \alpha '_3 +d + 1 }$ (resp. $V_{ \alpha '_0 +d + 1 ,  -\alpha '_1 - d,  \alpha '_2 +d + 1 ,  \alpha '_3 +d + 1 }$, $V_{ \alpha '_0 +d + 1 ,  \alpha '_1 +d + 1,  -\alpha '_2 -d  ,  \alpha '_3 +d + 1 }$, $V_{ \alpha '_0 +d + 1 ,  \alpha '_1 +d + 1,  \alpha '_2 +d + 1,  -\alpha '_3 -d  }$).\\
(ii) If $\alpha '_0 \in -1/2 + \Zint _{\leq 0}$ (resp. $\alpha ' _1 \in -1/2 +\Zint _{\leq 0}$, $\alpha ' _2 \in -1/2 +\Zint _{\leq 0}$, $\alpha ' _3 \in -1/2 +\Zint _{\leq 0}$), $\alpha '_1 +d, \alpha '_2 +d, \alpha '_3 +d \not \in 1/2+ \Zint $ (resp. $\alpha '_0 +d ,\alpha '_2 +d,\alpha '_3 +d \not \in 1/2+ \Zint $, $\alpha '_0 +d, \alpha '_1 +d, \alpha '_3 +d  \not \in 1/2+ \Zint $, $\alpha '_0 +d, \alpha '_1 +d, \alpha '_2 +d \not \in 1/2+ \Zint $), $d \not \in \Zint$ and the singularity $x=0$ (resp. $x=\omega _1$, $x=\omega _2$, $x=\omega _3$) of Eq.(\ref{InoEF00}) is apparent,
 then there exists a non-zero solution of Eq.(\ref{eq:HalfEf}) which belongs to the space $V_{\alpha '_0 +d +1, -\alpha '_1 -d ,  -\alpha '_2 -d ,  -\alpha '_3 -d  }$ (resp. $V_{ -\alpha '_0 -d  ,  \alpha '_1 + d +1,  -\alpha '_2 -d  , -\alpha '_3 -d }$, $V_{ -\alpha '_0 -d  ,  -\alpha '_1 -d ,  \alpha '_2 +d +1 ,  -\alpha '_3 -d }$, $V_{ -\alpha '_0 -d  ,  -\alpha '_1 -d ,  -\alpha '_2 -d , \alpha '_3 +d +1 }$).\\
(iii) If $\alpha '_0 +d \in 1/2 +\Zint _{\geq 0} $ (resp. $\alpha '_1 +d \in 1/2 +\Zint _{\geq 0}$, $\alpha '_2 +d \in 1/2 +\Zint _{\geq 0}$, $\alpha '_3 +d \in 1/2 +\Zint _{\geq 0} $), $l'_1 , l'_2 , l'_3 \not \in 1/2+ \Zint $ (resp. $l '_0 , l'_2 , l'_3 \not \in 1/2 +\Zint$, $l '_0 , l'_1, l'_3 \not \in 1/2+ \Zint$, $l '_0 , l'_1, l'_2 \not \in 1/2 +\Zint$), $d \not \in \Zint$ and there exists a non-zero solution of Eq.(\ref{InoEF00}) which belongs to the space $V_{1- \alpha '_0, \alpha '_1, \alpha '_2 , \alpha '_3 }$ (resp. $V_{\alpha '_0 , 1-\alpha '_1 , \alpha '_2, \alpha '_3 }$, $V_{\alpha '_0 , \alpha '_1, 1-\alpha '_2 , \alpha '_3 }$, $V_{\alpha '_0 , \alpha '_1, \alpha '_2, 1-\alpha '_3 }$), then the singularity $x=0$ (resp. $x=\omega _1$, $x=\omega _2$, $x=\omega _3$) of Eq.(\ref{eq:HalfEf}) is apparent.\\
(iv)  If $\alpha '_0 +d \in -3/2 +\Zint _{\leq 0} $ (resp. $\alpha '_1 +d \in -3/2 +\Zint _{\leq 0}$, $\alpha '_2 +d \in -3/2 +\Zint _{\leq 0}$, $\alpha '_3 +d \in -3/2 +\Zint _{\leq 0} $), $l'_1, l'_2 , l'_3 \not \in 1/2 + \Zint$ (resp. $l '_0 , l'_2 , l'_3 \not \in 1/2 + \Zint$, $l '_0 , l'_1, l'_3 \not \in 1/2 + \Zint$, $l '_0 , l'_1, l'_2 \not \in 1/2 + \Zint$), $d \not \in \Zint$ and there exists a non-zero solution of Eq.(\ref{InoEF00}) which belongs to the space $V_{ \alpha '_0 , 1-\alpha '_1 , 1-\alpha '_2 , 1-\alpha '_3 }$ (resp. $V_{1-\alpha '_0 , \alpha '_1, 1-\alpha '_2 , 1-\alpha '_3 }$, $V_{1- \alpha '_0 , 1-\alpha '_1 , \alpha '_2 , 1-\alpha '_3 }$, $V_{1-\alpha '_0 , 1-\alpha '_1 , 1-\alpha '_2 , \alpha '_3 }$), then the singularity $x=0$ (resp. $x=\omega _1$, $x=\omega _2$, $x=\omega _3$) of Eq.(\ref{eq:HalfEf}) is apparent.
\end{prop}
With respect to the elliptical representation of Heun's equation, Theorems \ref{thm:exprnonlogsol} and \ref{thm:exprnonlogsol2} can be rewritten as follows:
\begin{prop} \label{prop:exprnonlogsolell}
Let $\alpha _0 \in \{ -l_0, l_0 +1 \}$ and $ l'_0, l'_1, l'_2, l'_3, \eta $ be the parameters defined in Eq.(\ref{eq:l'alpha}).\\
(i) If $-\alpha _0 \in 1/2 + \Zint _{\geq 0}$ (resp. $l _1 \in 1/2 + \Zint _{\geq 0}$, $l _2 \in 1/2 + \Zint _{\geq 0}$, $l _3 \in 1/2 + \Zint _{\geq 0}$), $l'_1, l'_2 , l'_3 \not \in 1/2 + \Zint$ (resp. $l '_0 , l'_2 , l'_3 \not \in 1/2 + \Zint$, $l '_0 , l'_1, l'_3 \not \in 1/2 + \Zint$, $l'_0 , l'_1, l'_2 \not \in 1/2 + \Zint$), $\eta \not \in \Zint$ and the singularity $x=0$ (resp. $x=\omega _1$, $x=\omega _2$, $x=\omega _3$) of Eq.(\ref{InoEF}) is apparent, then there exists a non-zero solution $\tilde{f}(x)$ of Eq.(\ref{InoEF00}) which belongs to the space $V_{-l '_0 , l'_1 +1, l'_2 +1, l'_3 +1}$ (resp. $V_{l '_0 +1, -l'_1, l'_2 +1, l'_3 +1}$, $V_{l '_0 +1, l'_1 +1, -l'_2 , l'_3 +1 }$, $V_{l'_0 +1, l'_1 +1, l'_2 +1, -l'_3 }$) and the functions 
\begin{align}
& f(x)=\sigma (x) ^{\alpha _0} \sigma _1 (x) ^{-l _1} \sigma _2 (x) ^{-l _2} \sigma _3 (x) ^{-l _3} \cdot \\
& \quad \quad \int _{I_i } \tilde{f}(y) \sigma (y) ^{l' _0 +1}  \sigma _1 (y) ^{l' _1 +1} \sigma _2 (y) ^{l' _2 +1} \sigma _3 (y) ^{l' _3 +1} (\sigma (x+y) \sigma (x-y) )^{-\eta } dy \nonumber 
\end{align}
for $i=1,2,3$ (resp. $i=2,3$, $i=1,3$, $i=1,2$) are non-zero solutions of Eq.(\ref{InoEF}).\\
(ii) 
If $-\alpha _0 \in -3/2 + \Zint _{\leq 0}$ (resp. $l _1 \in -3/2 +\Zint _{\leq 0}$, $l _2 \in -3/2 +\Zint _{\leq 0}$, $l _3 \in -3/2 +\Zint _{\leq 0}$), $l'_1, l'_2 , l'_3 \not \in 1/2 + \Zint$ (resp. $l'_0 , l'_2 , l'_3 \not \in 1/2 + \Zint$, $l'_0 , l'_1, l'_3 \not \in 1/2 + \Zint$, $l'_0 , l'_1, l'_2 \not \in 1/2 + \Zint$), $\eta \not \in \Zint$ 
 and the singularity $x=0$ (resp. $x=\omega _1$, $x=\omega _2$, $x=\omega _3$) of Eq.(\ref{InoEF}) is apparent, then there exists a non-zero solution $\tilde{f}(x)$ of Eq.(\ref{InoEF00}) which belongs to the space $V_{l '_0 +1, -l'_1, -l'_2 , -l'_3 }$ (resp. $V_{-l'_0 , l'_1+1, -l'_2, -l'_3 }$, $V_{-l'_0 , -l'_1, l'_2 +1, -l'_3 }$, $V_{-l '_0 , -l'_1, -l'_2, l'_3 +1 }$) and the functions 
\begin{align}
& f(x)=\sigma (x) ^{-\alpha _0 +1} \sigma _1 (x) ^{l _1 +1} \sigma _2 (x) ^{l _2 +1} \sigma _3 (x) ^{l _3 +1} \cdot \\
& \quad \quad \int _{I_i } \tilde{f}(y) \sigma (y) ^{-l' _0 }  \sigma _1 (y) ^{-l' _1 } \sigma _2 (y) ^{-l' _2 } \sigma _3 (y) ^{-l' _3 } (\sigma (x+y) \sigma (x-y) )^{\eta -2} dy \nonumber 
\end{align}
for $i=1,2,3$ (resp. $i=2,3$, $i=1,3$, $i=1,2$) are non-zero solutions of Eq.(\ref{InoEF}).\\
\end{prop}

\section{Finite-gap potentials and integral transformations} \label{sec:HeunFGIT}

We now review the definitions of a finite-gap potential and its properties.
\begin{df} \label{def:fin}
Assume $q(x)$ is real-valued and continuous for $x \in \Rea$.
We set $H=-d^2/dx^2+q(x)$. Let $\sigma _b(H)$ be the set such that 
$$
E \in \sigma _b(H) \; \Leftrightarrow \mbox{ All solutions of }(H-E)f(x)=0 \mbox{ are bounded on }x \in \Rea ,
$$
and $\overline{\sigma _b(H) }$ is the topological closure of $ \sigma _b(H)$ in $\Rea $.
If the set $\Rea \setminus \overline{\sigma _b(H) }$ can be written as
\begin{equation}
\Rea \setminus \overline{\sigma _b(H)}= (-\infty,E_{0} ) \cup (E_{1},E_{2}) \cup \dots \cup (E_{2g-1}, E_{2g}) ,
\end{equation}
with $E_0<E_{1}<\cdots <E_{2g}$
then $q(x)$ is called a finite-gap ($g$-gap) potential.
\end{df}
If $q(x)$ is real-valued and continuous for $x \in \Rea$ and periodic with period $T(>0)$, then $| \mbox {tr}M_T | >2 \Rightarrow  E \not \in \sigma _b(H)$ and $| \mbox {tr}M_T | <2  \Rightarrow  E \in \sigma _b(H)$, where $M_T$ is a monodromy matrix for the shift $x \rightarrow x+T$ with eigenvalue $E$.

\begin{df} \label{def:algfin}
If there exists an odd-order differential operator 
$A= \left( d/dx \right)^{2g+1} +\sum_{j=0}^{2g-1} b_j(x)\left( d/dx \right)^{2g-1-j}$ such that $[A, -d^2/dx^2+q(x)]=0$, then $q(x)$ is called an algebro-geometric finite-gap potential.
\end{df}
Note that the equation  $[A, -d^2/dx^2+q(x)]=0$ is equivalent to the function $q(x)$ being a solution of some stationary higher-order KdV equation.
It is known that if $q(x) $ is real-holomorphic on $\Rea$ and $q(x+T)=q(x)$, then $q(x)$ is a finite-gap potential if and only if $q(x)$ is an algebro-geometric finite-gap potential (see \cite{Nov0}).

For the elliptical representation of Heun's equation, the following theorem is known.
\begin{thm} $($\cite{TV}$)$
The potential $\sum_{i=0}^3 l'_i(l'_i+1)\wp (x+\omega_i)$ is algebro-geometric finite-gap, if and only if $l'_i \in \Zint$ for $i =0,1,2,3$.
\end{thm}
The function $ \sum_{i=0}^3 l'_i(l'_i+1)\wp (x+\omega_i)$ is called the Treibich-Verdier potential. 
Subsequently several other researchers have produced results on this subject (see \cite{GW,Smi,Tak1,Tak2,Tak3,Tak4,Tak5}).
If $l'_0=l'_1=0$, $\omega _1 \in \Rea _{\neq 0}$ and $\omega _3 \in \sqrt{-1}  \Rea _{\neq 0}$, then the potential is real-valued and holomorphic on $\Rea $, and we have the following corollary:
\begin{cor}
If $\omega _1 \in \Rea _{\neq 0}$, $\omega _3 \in \sqrt{-1}  \Rea _{\neq 0}$ and $l'_2, l'_3 \in \Zint$, then the potential $ l'_2(l'_2+1)\wp (x+ \omega _2) +l'_3(l'_3+1)\wp (x+\omega _3)$ is a finite-gap potential.
\end{cor}

We review a method for calculating the monodromy for the elliptical representation of Heun's equation for the case $l'_0, l'_1, l'_2, l'_3 \in \Zint$.
Note that Eq.(\ref{InoEF00}) is invariant under the change $l'_i \leftrightarrow -l'_i-1$ for each $i \in \{0,1,2,3 \}$.

Let $h(x)$ be  the product of any pair of solutions of the elliptical representation of Heun's equation. Then the function $h(x)$ satisfies the following third-order differential equation:
\begin{align}
& \left( \frac{d^3}{dx^3}-4\left( \sum_{i=0}^3 l'_i(l'_i+1)\wp (x+\omega_i)-E\right)\frac{d}{dx}-2\left(\sum_{i=0}^3 l'_i(l'_i+1)\wp '(x+\omega_i)\right) \right) h (x)=0.
\label{prodDE} 
\end{align}
It is known that if $l'_0, l'_1, l'_2, l'_3 \in \Zint $ then Eq.(\ref{prodDE}) has a non-zero doubly periodic solution for all $E$.
\begin{prop} $($\cite[Proposition 3.5]{Tak1}$)$ \label{prop:prod}
If $l'_0, l'_1, l'_2, l'_3 \in \Zint $, then Eq.(\ref{prodDE}) has a non-zero doubly periodic solution $\Xi (x,E)$, which has the expansion
\begin{equation}
\Xi (x,E)=c_0(E)+\sum_{i=0}^3 \sum_{j=0}^{\max(l'_i,-l'_i-1) -1} b^{(i)}_j (E)\wp (x+\omega_i)^{\max(l'_i,-l'_i-1) -j},
\label{Fx}
\end{equation}
where the coefficients $c_0(E)$ and $b^{(i)}_j(E)$ are polynomials in $E$, they do not have common divisors and the polynomial $c_0(E)$ is monic.
We set $g=\deg_E c_0(E)$. Then the coefficients satisfy $\deg _E b^{(i)}_j(E)<g$ for all $i$ and $j$.
\end{prop}
Set
\begin{align}
 & Q(E)=  \Xi (x,E)^2\left( E- \sum_{i=0}^3 l'_i(l'_i+1)\wp (x+\omega_i)\right) +\frac{1}{2}\Xi (x,E)\frac{d^2\Xi (x,E)}{dx^2}-\frac{1}{4}\left(\frac{d\Xi (x,E)}{dx} \right)^2. \label{const}
\end{align}
Then $Q(E)$ is independent of $x$ and it is a monic polynomial in $E$ of degree $2g+1$ (see \cite{Tak1}). 
Solutions of Heun's equations can be written using $\Xi (x,E)$ and $Q(E)$.
\begin{prop} $($\cite[Proposition 3.7]{Tak1}$)$
The functions 
\begin{equation}
\Lambda (x,E)=\sqrt{\Xi (x,E)}\exp \int \frac{\sqrt{-Q(E)}dx}{\Xi (x,E)}
\label{eqn:Lam}
\end{equation}
and $\Lambda (-x,E) $ are solutions of Eq.(\ref{InoEF00}).
\end{prop}
Write
\begin{equation}
\Xi (x,E)=c(E)+\sum_{i=0}^3 \sum_{j=0}^{\max(l'_i,-l'_i-1)-1 } a^{(i)}_j (E)\left( \frac{d}{dx} \right) ^{2j} \wp (x+\omega_i),
\label{FFx}
\end{equation}
and set
\begin{equation}
a(E)=\sum _{i=0}^3 a^{(i)} _0 (E).
\label{polaE}
\end{equation}
Then the monodromy with respect to the shift of a period can be written in terms of a hyperelliptic integral.
\begin{prop} $($\cite{Tak3,Tak4}$)$ \label{prop:Tak3} 
Assume $l'_0, l'_1, l'_2, l'_3\in \Zint$.\\
(i) If $Q(E_0)=0$, then there exists $q_k \in \{0,1\}$ such that $\Lambda (x+2\omega _k,E_0)=(-1)^{q_k} \Lambda (x,E_0)$ for each $k \in \{ 1,3\}$.\\
(ii) 
If $Q(E) \neq 0$, then the functions $ \Lambda (x,E)$ and $ \Lambda (-x,E)$ are linearly independent and we have 
\begin{equation}
\Lambda (\pm (x+2\omega _k),E)=(-1)^{q_k} \Lambda (\pm x,E) \exp \left( \mp \int_{E_0}^{E}\frac{\omega _k c(\tilde{E}) -\eta _k a(\tilde{E}) }{\sqrt{-Q(\tilde{E})}} d\tilde{E}\right) .
\label{hypellint}
\end{equation}
\end{prop}
We introduce another expression of monodromy arising from the Hermite-Krichever Ansatz \cite{Tak4}.
Set
\begin{equation}
\Phi _i(x,\alpha )= \frac{\sigma (x+\omega _i -\alpha ) }{ \sigma (x+\omega _i )} \exp (\zeta( \alpha )x), \quad \quad (i=0,1,2,3),
\label{Phii}
\end{equation}
where $\sigma (x)$ (resp. $\zeta (x)$) is the Weierstrass sigma (resp. zeta) function.
\begin{prop} $($\cite{Tak4}$)$ \label{prop:Tak4}
Assume $l'_0, l'_1, l'_2, l'_3\in \Zint$.
There exist polynomials $P_1(E), \dots ,P_6 (E)$ such that, if the eigenvalue $E$ satisfies $P_2(E) \neq 0$, then the function $\Lambda (x,E)$ in Eq.(\ref{eqn:Lam}) can be written as
\begin{align}
& \Lambda (x,E) = \exp \left( \kappa x \right) \left( \sum _{i=0}^3 \sum_{j=0}^{|l'_i+1/2|-3/2} \tilde{b} ^{(i)}_j \left( \frac{d}{dx} \right) ^{j} \Phi _i(x, \alpha ) \right) ,
\label{LalphaFG}
\end{align}
and the values $\alpha $ and $\kappa $ can be expressed as
\begin{equation}
 \wp (\alpha ) =\frac{P_1 (E)}{P_2 (E)}, \; \; \; \wp ' (\alpha ) =\frac{P_3 (E)}{P_4 (E)} \sqrt{-Q(E)} , \; \; \kappa  =\frac{P_5 (E)}{P_6 (E)} \sqrt{-Q(E)}.
\label{P1P6}
\end{equation}
The periodicity of the function $\Lambda (\pm x,E)  $ in  Eq.(\ref{LalphaFG}) is described as 
\begin{align}
& \Lambda (\pm (x+2\omega _k),E) = \exp (\pm (2\omega _k ( \zeta (\alpha ) +2 \kappa ) -2\eta _k \alpha )) \Lambda (\pm x,E) , \quad  (k=1,3). \label{ellintFG} 
\end{align}
If $P_2(E) = 0$, then the function $\Lambda (x,E)$ in Eq.(\ref{eqn:Lam}) can be expressed as a product of an exponential function and a doubly periodic function.
\end{prop}

We review a relationship between the polynomial $Q(E)$ and finite-dimensional invariant subspaces. 
We define a vector space $V$ by
\begin{align}
& V=
\left\{
\begin{array}{r}
U_{-l'_0,-l'_1,-l'_2,-l'_3}\oplus U_{-l'_0 ,-l'_1,l'_2+1 ,l'_3+1}\oplus U_{-l'_0,l'_1+1,-l'_2,l'_3+1}\oplus U_{-l'_0 ,l'_1+1,l'_2+1,-l'_3} \\
( l'_0 +l'_1 +l'_2 +l'_3 : \mbox{ even}); \\
U_{-l'_0,-l'_1,-l'_2,l'_3+1}\oplus U_{-l'_0 ,-l'_1,l'_2+1,-l'_3} \oplus U_{-l'_0 ,l'_1+1,-l'_2,-l'_3} \oplus U_{l'_0 +1,-l'_1,-l'_2,-l'_3} \\
 (l'_0 +l'_1 +l'_2 +l'_3 : \mbox{ odd}),
\end{array}
\right. \label{sp:V}
\end{align}
where $ U_{\alpha _0, \alpha _1, \alpha _2, \alpha _3}$ are defined by
\begin{align}
& U_{\alpha _0, \alpha _1, \alpha _2, \alpha _3}=
\left\{
\begin{array}{ll}
 V_{\alpha _0, \alpha _1, \alpha _2, \alpha _3}, & \sum_{i=0}^3 \alpha _i/2 \in \Zint_{\leq 0}; \\
V_{1-\alpha _0, 1-\alpha _1, 1-\alpha _2, 1-\alpha _3}, & \sum_{i=0}^3 \alpha _i /2\in \Zint_{\geq 2};\\
\{ 0 \} , & \mbox{otherwise} ,
\end{array}
\right. 
\end{align}
Then $H ^{(l'_0, l'_1, l'_2 , l'_3)}\cdot V \subset V$ and it can be shown that if $l'_0, l'_1, l'_2 , l'_3 \in \Zint $ then $V$ is the maximum finite-dimensional $H$-invariant subspace of the space spanned by the function $f(x)$ such that $f(x+2\omega _k )/f(x) \in \{\pm 1 \}$ for $k =1,3$.
Let $P(E)$ be the monic characteristic polynomial of the operator $H^{(l'_0, l'_1, l'_2 , l'_3)}$
on the space $V$, i.e. $P(E)=\det _V (E\cdot 1 -H^{(l'_0, l'_1, l'_2 , l'_3)})$. 
\begin{prop} $($\cite{Tak5}$)$ \label{thm:dist}
We have $P(E)=Q(E)$.
\end{prop}
The curve $\Gamma \! : \nu ^2=-Q(E)$ is called the spectral curve, which plays an important role in Eqs.(\ref{eqn:Lam}), (\ref{hypellint}).
It follows from Proposition \ref{thm:dist} that edges of the hyperelliptic curve $\Gamma $ are eigenvalues of the operator $H^{(l'_0, l'_1, l'_2 , l'_3)}$ on the invariant space $V$.
The genus of the curve $\Gamma $ is $g$, where $g$ is defined in Proposition \ref{prop:prod}. 

Let us consider the case $Q(E)=0$. 
Let $E_0$ be a zero of $Q(E)$.
Then we have $P(E_0)=0$, $\Lambda (x,E_0)= \sqrt{\Xi(x,E_0)} \in V$ and the functions $ \Lambda (x,E_0) $ and $\Lambda (-x,E_0 ) $ are linearly dependent.
Another solution of Eq.(\ref{InoEF00}) can be derived as  $\sqrt{\Xi(x,E_0)}\int \frac{dx}{\Xi(x,E_0)} (= \Lambda _2 (x,E_0))$.
The monodromy with respect to the shift of a period was calculated in \cite{TakH} and it can be written as
\begin{align}
& (\Lambda (x+2\omega _k ,E_0) , \Lambda _2(x+2\omega _k,E_0) )= \\
& (-1)^{q_k}  (\Lambda (x,E_0) , \Lambda _2(x,E_0) )\left( 
\begin{array}{cc}
1 & \left. \frac{2\omega _k c(E)  -2\eta _k a(E)}{\frac{d}{dE}Q(E)} \right|_{E \rightarrow E_0} \\
 0 & 1
\end{array}
\right) . \nonumber
\end{align}

\begin{exa}
The case $l'_0 =2$, $l'_1=l'_2=l'_3=0$.
The doubly periodic function $\Xi (x,E)$ which satisfies Eq.(\ref{prodDE}) and the polynomial $Q(E)$ are evaluated as
\begin{align}
& \Xi (x,E)= 9\wp (x)^2 +3E \wp (x) +E^2 -9g_2/4  ,\\
& Q(E)=(E^2-3g_2)\prod _{i=1}^3 (E-3e_i) .
\end{align}
The function $\Lambda (x,E)$ defined by Eq.(\ref{eqn:Lam}) is a solution of Eq.(\ref{InoEF00}).
For the monodromy with respect to the shift $x \rightarrow x+2\omega _k$ $(k=1,3)$, we have a formula described by a hyperelliptic integral of genus two.
\begin{align}
& \Lambda (x+2\omega _k,E)= \Lambda (x,E) \exp \left( -\frac{1}{2} \int_{\sqrt{3g_2}}^{E}\frac{ \omega _k (2\tilde{E}^2-3g_2)-6\eta _k \tilde{E}  }{\sqrt{-(\tilde{E}^2-3g_2) \prod _{i=1}^3 (\tilde{E}-3e_i)}} d\tilde{E}\right) .
\label{eq:HE2000}
\end{align}
The function $\Lambda (x,E)$ can be expressed in the form of the Hermite-Krichever Ansatz 
\begin{align}
& \Lambda (x,E) = \exp \left( \kappa x \right) \left(  \tilde{b} ^{(0)}_0  \Phi _0(x, \alpha ) + \tilde{b} ^{(0)}_1 \left( \frac{d}{dx} \right)  \Phi _0(x, \alpha ) \right) ,
\end{align}
and $\alpha $, $\kappa $ satisfy
\begin{align}
& \wp( \alpha )= e_1 -\frac{(E-3e_1)(E+6e_1)^2}{9(E^2-3g_2)}, \quad \kappa =\frac{2}{3(E^2-3g_2)}\sqrt{-Q(E)}.
\label{eq:HK2000}
\end{align}
Set 
\begin{align}
& V= V_{-2,0,0,0}\oplus V_{0,-1,-1,0}\oplus V_{0,-1,0,-1}\oplus V_{0,0, -1,-1}.
\end{align}
Then $\dim V= 2+1+1+1=5$ and $Q(E)$ is the characteristic polynomial of $H^{(2,0,0,0)}$ on the space $V$.
The characteristic polynomial of $H^{(2,0,0,0)}$ on $V_{-2,0,0,0}$ (resp. $V_{0,-1,-1,0}$, $V_{0,-1,0,-1}$, $V_{0,0, -1,-1}$) is $E^2-3g_2$ (resp. $E-3e_3$, $E-3e_2$, $E-3e_1$).
\end{exa}

By applying integral transformation to the case of a finite-gap potential (i.e. applying Theorem \ref{thm:ellipinttras0} for the case $l'_0, l'_1 , l'_2 ,l'_3 \in \Zint $ while choosing $\alpha ' _0  \in \{ -l'_0 ,l'_0 +1 \} $ to be $\eta \in 1/2 +\Zint $), we obtain Heun's equation for the case $l_0, l_1 , l_2 ,l_3 \in \Zint +1/2$ and $l_0+ l_1 + l_2 +l_3 \in 2\Zint +1$.
Conversely we can express solutions and monodromy for the case $l_0, l_1 , l_2 ,l_3 \in \Zint +1/2$ and $l_0+ l_1 + l_2 +l_3 \in 2\Zint +1$ by using solutions and monodromy calculated by the finite-gap potential method for the case $l'_0, l'_1 , l'_2 ,l'_3 \in \Zint $.
The following proposition is obtained by combining Proposition \ref{prop:ellipinttras}, Corollary \ref{thmcor:pp}, Propositions \ref{prop:Tak3} and \ref{prop:Tak4}.
\begin{prop} \label{prop:ellipinttrasfingap}
Let $\alpha _0 \in \{ -l_0, l_0 +1 \}$ and set
\begin{align}
& \eta =\frac{-\alpha _0 -l_1-l_2-l_3+1}{2}, \; l'_0=\frac{-\alpha _0+l_1+l_2+l_3+1}{2},  \label{eq:l'alpha1} \\
& l'_1=\frac{-\alpha _0+l_1-l_2-l_3-1}{2}, \; l'_2=\frac{-\alpha _0 -l_1+l_2-l_3-1}{2}, \; l'_3=\frac{-\alpha _0 -l_1-l_2+l_3-1}{2}. \nonumber 
\end{align}
If $l_0, l_1, l_2, l_3 \in \Zint +1/2 $ and  $l_0+ l_1 + l_2 +l_3 \in 2\Zint +1$, then we have $\eta \in \Zint +1/2 $ and $l'_0, l'_1 , l'_2 ,l'_3 \in \Zint $.
Let $M_{2\omega _k}$ $(k=1,3)$ be a monodromy matrix of solutions of Eq.(\ref{InoEF0}) with respect to the shift $x \rightarrow x+2\omega _k$ for the parameters $l_0, l_1, l_2, l_3 ,E$.
Then we have
\begin{align}
& {\rm tr} M_{2\omega _k} = 2(-1)^{q_k} \cos \left( \int_{E_0}^{E}\frac{\omega _k c(\tilde{E}) -\eta _k a(\tilde{E}) }{\sqrt{Q(\tilde{E})}} d\tilde{E}\right) ,
\label{eq:monodHEint}
\end{align}
where $c(E)$ and $a(E)$ are defined in Eqs.(\ref{FFx}), (\ref{polaE}) and $E_0$ is a zero of $Q(E)$ for the parameters $l'_0, l'_1, l'_2, l'_3 ,E$ such that $\Lambda (x+2\omega _k,E_0)=(-1)^{q_k} \Lambda (x,E_0)$ for each $k \in \{ 1,3\}$.
We also have 
\begin{align}
& {\rm tr} M_{2\omega _k} = 2\cos \left( \sqrt{-1} (2\omega _k ( \zeta (\alpha ) + \kappa ) -2\eta _k \alpha ) \right) ,
\label{eq:monodHKint}
\end{align}
where $\alpha $ and $\kappa $ are determined by Eq.(\ref{P1P6}) for the parameters $l'_0, l'_1, l'_2, l'_3 ,E$.
\end{prop}
Note that Heun's equation in Proposition \ref{prop:ellipinttrasfingap} for the parameter $l'_0$, $l'_1$, $l'_2$, $l'_3$ for the case $\alpha _0=-l_0$ is isomonodromic to the one for the parameter $l'_0$, $l'_1$, $l'_2$, $l'_3$ for the case $\alpha _0=l_0 +1$, and they are linked by the generalized Darboux transformation described in \cite{Tak5}.
If we replace the definition of the set $\overline{\sigma _b(H)}$ by the following; $E \in \overline{\sigma _b(H)} \Leftrightarrow -2 \leq \mbox {tr}M_{2\omega _1} \leq 2 $,
then the set $\Rea \setminus \overline{\sigma _b(H)}$ for the case $l_0, l_1 , l_2 ,l_3 \in \Zint +1/2$, $l_0+ l_1 + l_2 +l_3 \in 2\Zint +1 $ and $\omega _1 , \sqrt{-1} \omega _3 \in \Rea _{\neq 0}$ has finite gaps, which coincides with the one for the case $l'_0$, $l'_1$, $l'_2$, $l'_3$ in Proposition \ref{prop:ellipinttrasfingap}.
But the potential for the case  $l_0, l_1 , l_2 ,l_3 \in \Zint +1/2$ and $l_0+ l_1 + l_2 +l_3 \in 2\Zint +1 $ is not an algebro-geometric finite-gap potential.

It follows from Propositions \ref{prop:nonlog-polynell} and \ref{prop:exprnonlogsolell} that the eigenvalues of the four spaces for $l'_0$, $l'_1$, $l'_2$, $l'_3$ in Eq.(\ref{sp:V}) corresponds to eigenvalues such that one of the singularities $\{ 0,\omega _1, \omega _2, \omega _3 \}$ is apparent. 
By combining these remarks with Proposition \ref{thm:dist}, we have the following proposition:
\begin{prop} \label{prop:lplQE}
Let $\alpha _0 \in \{ -l_0, l_0 +1 \}$ and define the numbers $l'_0, l'_1 , l'_2 ,l'_3 $ by Eq.(\ref{eq:l'alpha1}).
Assume  $l_0, l_1, l_2, l_3 \in \Zint +1/2 $, $l_0+ l_1 + l_2 +l_3 \in 2\Zint +1$ and let $Q(E)$ be the polynomial in Eq.(\ref{const}) for the parameters $l'_0, l'_1 , l'_2 ,l'_3 (\in \Zint )$.\\
(i) The condition $Q(E_0)=0$ is equivalent to the condition that there exists $i \in \{ 0,1,2,3 \}$ such that the singularity $x =\omega _i$ is apparent in Eq.(\ref{InoEF}) for the parameters $l_0, l_1 , l_2 ,l_3 $.\\
(ii) If $l'_0+ l'_1 + l'_2 +l'_3 $ is even, then the characteristic polynomial of the operator $H^{(l'_0,l'_1,l'_2,l'_3)}$ on the space $U_{-l'_0,-l'_1,-l'_2,-l'_3}$ (resp. $U_{-l'_0 ,-l'_1,l'_2+1 ,l'_3+1}$, $U_{-l'_0,l'_1+1,-l'_2,l'_3+1}$, $U_{-l'_0 ,l'_1+1,l'_2+1,-l'_3}$) coincides with the polynomial $P^{(0)} (E)$ (resp. $P^{(1)} (E)$, $P^{(2)} (E)$, $P^{(3)} (E)$) for the parameters $l_0, l_1 , l_2 ,l_3 $ which are defined between Eq.(\ref{eq:expell}) and Proposition \ref{prop:exprnonlogsolell}.\\
(iii) If $l'_0 + l'_1 + l'_2 +l'_3 $ is odd, then the characteristic polynomial of the operator $H^{(l'_0,l'_1,l'_2,l'_3)}$ on the space $U_{-l'_0,-l'_1,-l'_2,l'_3+1}$ (resp. $U_{-l'_0 ,-l'_1,l'_2+1,-l'_3}$, $U_{-l'_0 ,l'_1+1,-l'_2,-l'_3}$, $U_{l'_0 +1,-l'_1,-l'_2,-l'_3}$) coincides with the polynomial $P^{(3)} (E)$ (resp. $P^{(2)} (E)$, $P^{(1)} (E)$, $P^{(0)} (E)$) for the parameters $l_0, l_1 , l_2 ,l_3 $.\\
(iv) We have $Q(E) =P^{(0)}(E) P^{(1)}(E) P^{(2)}(E) P^{(3)}(E) $.
\end{prop}
It was shown in \cite{Tak1} that if $l'_0, l'_1 , l'_2 ,l'_3 \in \Zint $, then any two spaces of the four spaces in Eq.(\ref{sp:V}) have no eigenvalues in common.
Hence we have
\begin{prop}
Assume $l_0, l_1, l_2, l_3 \in \Zint +1/2 $ and $l_0+ l_1 + l_2 + l_3 \in 2\Zint +1$.
Then any two of the four equations $P^{(0)} (E)=0$, $P^{(1)} (E)=0$, $P^{(2)} (E)=0$, $P^{(3)} (E)=0$ have no common solutions.
In other words, if one of the singularities $\{ 0,\omega _1, \omega _2, \omega _3 \}$ is apparent, then the other three singularities are non-apparent.
\end{prop}
Under the assumptions and notations in Proposition \ref{prop:lplQE}, we have $\deg _E Q(E)= \deg _E P^{(0)}(E) + \deg _E P^{(1)}(E) + \deg _E P^{(2)}(E) + \deg _E P^{(3)}(E) = |l_0 +1/2 |+|l_1 +1/2 |+|l_2 +1/2 |+|l_3 +1/2 |$.
Hence the genus of the curve $\Gamma \! : \nu ^2=-Q(E)$ for $l'_0, l'_1 , l'_2 ,l'_3 (\in \Zint )$ is obtained by applying Proposition \ref{prop:nonlog-polynell}, setting $\alpha '_0 =-l'_0$ (resp. $\alpha '_0 =l'_0 +1$) for the case that  $l'_0 + l'_1 + l'_2 +l'_3 $ is even (resp. odd).
\begin{prop} \label{prop:genus}
Assume $l'_0, l'_1 , l'_2 ,l'_3 \in \Zint $.
Let $g$ be the genus of the curve $\Gamma \! : \nu ^2=-Q(E)$.\\
(i) If $l'_0+ l'_1 + l'_2 +l'_3 $ is even, then
\begin{align}
& g= \frac{1}{2} \left( \left| \frac{l '_0+l'_1+l'_2+l'_3}{2}  \right| + \left| \frac{l'_0+l'_1-l'_2-l'_3}{2} \right|  \right. \\
& \qquad \qquad \qquad \qquad \left. + \left| \frac{l'_0-l'_1+l'_2-l'_3}{2} \right| + \left| \frac{l'_0-l'_1-l'_2+l'_3}{2} \right| \right) . \nonumber
\end{align}
(ii) If $l'_0+ l'_1 + l'_2 +l'_3 $ is odd, then
\begin{align}
& g= \frac{1}{2} \left( \left| \frac{-l '_0+l'_1+l'_2+l'_3+1}{2}  \right| + \left| \frac{l'_0-l'_1+l'_2+l'_3+1}{2} \right|  \right. \\
& \qquad \qquad \qquad \qquad \left. + \left| \frac{l'_0+l'_1-l'_2+l'_3+1}{2} \right| + \left| \frac{l'_0+l'_1+l'_2-l'_3+1}{2} \right| -1 \right) . \nonumber
\end{align}
\end{prop}
Note that the expression in Proposition \ref{prop:genus} is different from the one in \cite[Proposition 3.3]{Tak3}.

\begin{exa}
For the case $l'_0=l'_1=l'_2=l'_3=0$, Eq.(\ref{InoEF00}) is written as $(d^2/dx^2 +E) f(x)=0$, a basis of solutions can be written as $e^{ \kappa x}$, $e^{-\kappa x}$ for the case $E \neq 0$ by writing $E=-\kappa ^2 $.
Hence we have $\mbox{\rm tr} M'_{2\omega _k} =e^{2\kappa \omega _k} +e^{-2\kappa \omega _k}$ $(k=1,3)$, where $M' _{2\omega _k}$ is a monodromy matrix of solutions of Eq.(\ref{InoEF00}) for the case $l'_0=l'_1=l'_2=l'_3=0$.
There exists a non-zero periodic (resp. anti-periodic) solution with respect to the period $2\omega _1$, if and only of $E$ can be written as $E= \pi^2 n^2/ \omega _1 ^2$ (resp. $E= \pi^2 (2n+1)^2/ (2\omega _1 )^2$) for some $n \in \Zint _{\geq 0}$.
 
We apply an integral transformation of Theorem \ref{thm:ellipinttras0} for the case $\alpha '_0=1$, $\alpha '_1=\alpha '_2=\alpha '_3=0$.
By replacing the contour integral $I_i$ by twice of the integral from $-x+2\omega _i $ to $x$ and setting $E=-\kappa ^2$, it follows from Eq.(\ref{eq:Hintell}) that the function 
\begin{align}
 f(x)&= \left( \prod _{i=1}^3 (\wp (x) -e_i) \right) ^{1/4}\int ^{x}_{-x+2\omega _i} \frac{e^{\tilde{\kappa }\xi }\sigma (x)\sigma (\xi  )}{\sqrt{\sigma (x-\xi )\sigma (x+\xi )}}d\xi 
\end{align}
is a solution of Eq.(\ref{InoEF}) for the case $l_0=1/2$, $l_1=l_2=l_3=-1/2$ for $i \in \{0,1,2,3 \}$, which reproduces the result in \cite{TakI}.
By Corollary \ref{thmcor:pp} we have $\mbox{\rm tr} M_{2\omega _k} =e^{2\kappa \omega _k} +e^{-2\kappa \omega _k}$ $(k=1,3)$, where $M _{2\omega _k}$ is a monodromy matrix of solutions of Eq.(\ref{InoEF00}) for the case $l_0=1/2$, $l_1=l_2=l_3=-1/2$.
It follows from Corollary \ref{cor:pp} that there exists a non-zero periodic (resp. anti-periodic) solution with respect to the period $2\omega _1$, if and only of $E$ can be written as $E= \pi^2 n^2/ \omega _1 ^2$ (resp. $E= \pi^2 (2n+1)^2/ (2\omega _1 )^2$) for some $n \in \Zint _{\geq 0}$.
As a sequel, if $\omega _1 \in \Rea _{>0}$ and $\omega _3 \in \sqrt{-1} \Rea _{\neq 0}$, then the spectrum of the operator $H^{(1/2,-1/2,-1/2,-1/2)}$ (see Eq.(\ref{Ino})) with respect to the interval $[0,\omega _1]$ can be expressed as $L^2 ([0,\omega _1]) = \{ \pi^2 n^2/ (2\omega _1 )^2 \: | \: n\in \Zint _{\geq 0} \}$, which reproduces the result by Ruijsenaars which was presented at the Bonn conference in 2008 (see \cite{Rui}).
Note that Heun's equation for the case $l_0=1/2$, $l_1=l_2=l_3=-1/2$ was previously studied by Valent \cite{Val} to understand an eigenvalue problem related to certain birth and death processes.
\end{exa}
\begin{exa}
We apply Proposition \ref{prop:ellipinttrasfingap} to the case $l_0=3/2$, $l_1=l_2=l_3=1/2$.
By setting $\alpha _0= 5/2$, we have $l'_0=-3$, $l'_1=l'_2=l'_3=0$ and $\eta = 1/2$ in Proposition \ref{prop:ellipinttras}.
Let $M_{2\omega _k}$ $(k=1,3)$ be a monodromy matrix of solutions of Eq.(\ref{InoEF0}) with respect to the shift $x \rightarrow x+2\omega _k$ for the case $l_0=3/2$, $l_1=l_2=l_3=1/2$.
Since $H^{(-3,0,0,0)} =H^{(2,0,0,0)}$, it follows from Eqs.(\ref{eq:HE2000}), (\ref{eq:monodHEint}) that 
\begin{align}
& {\rm tr} M_{2\omega _k} = 2\cos  \left( -\frac{1}{2} \int_{\sqrt{3g_2}}^{E}\frac{ \omega _k (2\tilde{E}^2-3g_2)-6\eta _k \tilde{E}  }{\sqrt{(\tilde{E}^2-3g_2) \prod _{i=1}^3 (\tilde{E}-3e_i)}} d\tilde{E}\right) ,
\end{align}
and it follows from Corollary \ref{cor:pp} that for each $k\in \{ 1,3 \}$ there exists a non-zero solution $f_k(x,E)$ of Eq.(\ref{InoEF}) for the case $l_0=3/2$, $l_1=l_2=l_3=1/2$ such that 
\begin{align}
& f_k (x+2\omega _k,E)= f_k (x,E)  \exp \left( -\frac{1}{2} \int_{\sqrt{3g_2}}^{E}\frac{ \omega _k (2\tilde{E}^2-3g_2)-6\eta _k \tilde{E}  }{\sqrt{-(\tilde{E}^2-3g_2) \prod _{i=1}^3 (\tilde{E}-3e_i)}} d\tilde{E}\right) .
\end{align}
We also have 
\begin{align}
& {\rm tr} M_{2\omega _k} = 2\cos \left( \sqrt{-1} (2\omega _k ( \zeta (\alpha ) + \kappa ) -2\eta _k \alpha ) \right) ,
\end{align}
where $\alpha $ and $\kappa $ are defined by
\begin{align}
& \wp( \alpha )= e_1 -\frac{(E-3e_1)(E+6e_1)^2}{9(E^2-3g_2)}, \quad \kappa =\frac{2}{3}\sqrt{\frac{-\prod _{i=1}^3 (\tilde{E}-3e_i)}{E^2-3g_2}}.
\end{align}
It follows from Proposition \ref{prop:lplQE} that the singularity $x=0$ (resp. $x=\omega _1, \omega _2, \omega _3$) for Eq.(\ref{InoEF0}) on the case $l_0=3/2$, $l_1=l_2=l_3=1/2$ is apparent if and only if $E= \pm \sqrt{3g_2}$ (resp. $E=3e_1, 3e_2, 3e_3 $).

By setting $\alpha _0= -3/2$, we have $l'_0=-1$, $l'_1=l'_2=l'_3=-2$, which can be replaced by $l'_0=0$, $l'_1=l'_2=l'_3=1$. 
The case $(l'_0, l'_1, l'_2 ,l'_3 )=(0,1,1,1)$ is isomonodromic to the case $(l'_0, l'_1, l'_2 ,l'_3 )=(2,0,0,0)$, and the two cases are linked by the generalized Darboux transformation in \cite{Tak5}.
\end{exa}

\section{Summary and concluding remarks}
In this paper, we investigated correspondences of special solutions of Heun's differential equation (\ref{Heun}) and the second-order linear differential equation $D_{y_1}(\theta _0, \theta _1, \theta _t, \theta _{\infty}; \lambda ,\mu )$ (Eq.(\ref{eq:linP6})) by Euler's integral transformation.
Namely, polynomial-type solutions are connected to the solutions such that one of the regular singularities is apparent.
On the monodromy, the trace of a product of the local monodromy matrices is essentially preserved by the Euler's transformation (see Theorem \ref{thm:monodm}), and it is written more clearly on the elliptical representation of Heun's equation (see Theorem \ref{thm:ellipinttras0}).
The monodromy of the elliptical representation of Heun's equation (\ref{InoEF0}) in the case $l_0, l_1, l_2, l_3 \in \Zint +1/2 $ and $l_0+ l_1 + l_2 + l_3 \in 2\Zint +1$ can be calculated by using the results on finite-gap integration.

The results of this paper would also be valid for confluent families of Heun's equation (see \cite{Ron}) and the linear differential equation related with the other Painlev\'e equations. They should be presented clearly in a near future.

{\bf Acknowledgments}
The author thanks Professor S. N. M. Ruijsenaars for discussions and for sending a draft of his paper.
The author was supported by the Grant-in-Aid (No. 19740089, 22740107, 26400122) from the Japan Society for the Promotion of Science.

\appendix

\section{Local expansions and the proof of Propositions \ref{prop:nonzero}, \ref{prop:spansp0} and Theorems \ref{thm:nonlog-polynHeun}, \ref{thm:nonlog-polyn}} \label{sec:locexp}

We investigate local expansions of solutions of a second-order linear differential equation about a regular singularity and the image of the expansion mapped by an integral transformation.

We assume that the function $y(w)$ is a solution of a second-order linear differential equation about a regular singularity $w=p (\neq \infty )$,
and the exponents of the second-order linear differential equation at $w=p$ are $0$ and $\theta _p$.

Then the function $y(w)$ can be expanded as
\begin{equation}
y (w)= C ^{\langle p \rangle} f^{\langle p \rangle}(w)+ D^{\langle p \rangle} g^{\langle p \rangle}(w), \quad (C^{\langle p \rangle},D^{\langle p \rangle} \in \Cplx ),
\label{eq:ywfg}
\end{equation}
such that
\begin{align}
& f^{\langle p \rangle}(w)= \left\{
\begin{array}{ll}
\displaystyle \sum _{j=0} ^{\infty } c^{(p)} _j (w-p)^{j}, & \theta _p \not \in \Zint _{\geq 0} \\
\displaystyle (w-p) ^{\theta _p} \sum _{j=0} ^{\infty } c^{(p)} _j (w-p)^{j}, & \theta _p \in \Zint _{\geq 0}, \\
\end{array} 
\right. \label{eq:expansfpgp} \\
& g^{\langle p \rangle}(w)= \left\{
\begin{array}{ll}
\displaystyle (w-p) ^{\theta _p}  \sum _{j=0} ^{\infty } \tilde{c}^{(p)} _j (w-p)^{j}, & \theta _p \not \in \Zint \\
\displaystyle (w-p) ^{\theta _p} \left( \sum _{j=0} ^{\infty } \tilde{c}^{(p)} _j (w-p)^{j} \right) + A^{\langle p \rangle} f^{\langle p \rangle} (w) \log (w-p) , & \theta _p \in \Zint _{\leq -1} \\
\displaystyle \left( \sum _{j=0} ^{\infty } \tilde{c}^{(p)} _j (w-p)^{j} \right) + A^{\langle p \rangle} f^{\langle p \rangle} (w) \log (w-p) , & \theta _p \in \Zint _{\geq 0} ,\\
\end{array}
\right.  \nonumber
\end{align}
where $c^{(p)} _0= \tilde{c}^{(p)} _0=1$.
The function $f^{\langle p \rangle}(w)$ is holomorphic about $w=p$.
The function $g^{\langle p \rangle}(w)$ is branching about $w=p$, if $\theta _p \not \in \Zint _{\neq 0}$ or $A^{\langle p \rangle} \neq 0$.
If $\theta _p \in \Zint _{\neq 0}$ and $A^{\langle p \rangle} = 0$, the singularity $w=p$ is apparent

We now describe a criterion that the singularity $w=p$ is apparent for the case  $\theta _p \in \Zint _{\neq 0}$.
We denote the differential equation which the function $y(w)$ satisfies by 
\begin{equation}
\frac{d^2y}{dw^2} +\left( \sum _{i=0}^{\infty } r_i (w-p) ^{i-1} \right) \frac{dy}{dw} +\left( \sum _{i=0}^{\infty } s_i (w-p)^{i-2} \right) y=0.
\label{eq:DEwp}
\end{equation}
Let $F(\xi) =\xi^2 +(p_0-1) \xi +q_0 = \xi (\xi -\theta _p)$ be the characteristic polynomial about $w=p$.
If the function $y= w^{\rho } (\sum _{i=0}^{\infty} c_i z^i)$ $(c_0=1)$ satisfies Eq.(\ref{eq:DEwp}), then we have $F(\rho )=0$ and
\begin{equation}
F(\rho + n) c_n +\sum _{i=1}^n \{ (n-i+\rho )r_i +s_i \} c_{n-i} =0.
\label{eq:rho}
\end{equation}
If $F(\rho + n) \neq 0$ for all $n \in \Zint _{\geq 1}$, then the coefficients $c_n$ are determined recursively. 
In particular, the coefficients $c_n$ are determined recursively for the case $\theta  _p \not \in \Zint$.
We consider the case $\theta _p \in \Zint _{\geq 1}$.
Set $\rho =0$.
The coefficients $c_1 , \dots c_{\theta _p -1}$ are determined recursively.
We substitite $n= \theta _p$ in Eq.(\ref{eq:rho}).
Then 
\begin{equation}
\sum _{i=1}^{\theta _p} \{ (\theta _p-i )r_i +s_i \} c_{\theta _p-i} =0,
\label{eq:rho0}
\end{equation}
and it gives an eqivalent condition to that the singularity $w=p$ is apparent (i.e. $A^{\langle p \rangle} = 0$).
A condition of apparency of the singularity $w=p$ for the case $\theta _p \in \Zint _{\leq -1}$ is given by Eq.(\ref{eq:rho}) for the case $\rho = \theta _p$ and $n=-\theta _p$.

We investigate the local expansion of the function $\int _{[\gamma _{z} ,\gamma _p]} y(w) (z-w)^{\kappa } dw $ about $w=p$ for the case $\kappa  \not \in \Zint $.
Set 
\begin{align}
& d_{\alpha , \beta } =
\left\{
\begin{array}{ll}
\displaystyle (e^{2\pi \sqrt{-1} \alpha }-1)(e^{2\pi \sqrt{-1} \beta }-1) \frac{\Gamma (\alpha )\Gamma (\beta )}{\Gamma (\alpha +\beta )}, & \alpha \not \in \Zint ,\\
\displaystyle 2\pi \sqrt{-1} (e^{2\pi \sqrt{-1} \beta }-1) \frac{(-1)^{\alpha } \Gamma (\beta )}{(-\alpha ) ! \:\Gamma (\alpha +\beta )}, & \alpha \in \Zint _{\leq 0} , \\
\displaystyle 2\pi \sqrt{-1} (e^{2\pi \sqrt{-1} \beta }-1) \frac{\Gamma (\alpha )\Gamma (\beta )}{\Gamma (\alpha +\beta )}, & \alpha \in \Zint _{\geq 1}.
\end{array} \right. 
\label{eq:dalbe}
\end{align}
If $\beta \not \in \Zint $,
then $d_{\alpha , \beta } \neq 0 \Leftrightarrow \alpha + \beta \not \in \Zint _{\leq 0}$ and we have
\begin{align}
& \int _{[\gamma _{1} ,\gamma _0]} s^{\alpha -1} (1-s)^{\beta -1} ds = \left\{
\begin{array}{ll}
d_{\alpha , \beta } & \alpha \not \in \Zint _{\geq 1}\\
0 & \alpha \in \Zint _{\geq 1},
\end{array} \right. \label{eq:beta}\\
& \int _{[\gamma _{1} ,\gamma _0]} s^{n -1} (1-s)^{\beta -1} (\log s ) ds = d_{n , \beta },
\nonumber
\end{align}
for $n \in \Zint_{\geq 1}$.
For the function $y(w)$ in Eq.(\ref{eq:ywfg}), we have
\begin{align}
 \langle [\gamma _{z} ,\gamma _p] , y \rangle  = \int _{[\gamma _{z} ,\gamma _p]} y(w) (z-w)^{\kappa } dw \qquad \qquad \qquad \qquad \qquad \qquad \qquad \qquad 
\label{eq:gzgpexpa} \\
 = \left\{ 
\begin{array}{ll}
\displaystyle D^{\langle p \rangle}(z-p) ^{\theta _p +\kappa +1}  \sum _{j=0} ^{\infty } \tilde{c}^{(p)}_j d_{j+\theta _p +1, \kappa +1} (z-p)^{j}, & \displaystyle \theta _p \not \in \Zint ,\quad \quad \quad \; \atop{\displaystyle \theta _p +\kappa  \not \in \Zint _{\leq -2}}\\
\displaystyle D^{\langle p \rangle}\sum _{j=0} ^{\infty } \tilde{c}^{(p)}_{j-\kappa -\theta _p -1} d_{j-\kappa , \kappa +1} (z-p)^{j}, & \displaystyle \theta _p \not \in \Zint ,\quad \quad \quad \; \atop{\displaystyle \theta _p +\kappa  \in \Zint _{\leq -2}} \\
\displaystyle D^{\langle p \rangle}(z-p) ^{\theta _p +\kappa +1} \left\{ \displaystyle \sum _{j=0} ^{-\theta _p- 1} \tilde{c}^{(p)} _j d_{j+\theta _p +1, \kappa +1}(z-p)^{j} \quad \quad \quad \quad \atop{\displaystyle + A^{\langle p \rangle} \sum _{j=-\theta _p} ^{\infty } c^{(p)} _{j+\theta _p} d_{j+\theta _p +1, \kappa +1} (z-p)^{j} } \right\} , & \displaystyle \theta _p \in \Zint _{\leq -1} , \; \atop{\displaystyle \theta _p +\kappa  \not \in \Zint _{\leq -2}}\\
\displaystyle D^{\langle p \rangle} A^{\langle p \rangle} (z-p) ^{\theta _p +\kappa +1} \sum _{j=0} ^{\infty } c^{(p)} _j d_{j+ \theta _p +1, \kappa +1} (z-p)^{j} , & \displaystyle \theta _p \in \Zint _{\geq 0} , \quad \atop{\displaystyle \theta _p +\kappa  \not \in \Zint _{\leq -2} },
\end{array}
\right. \nonumber
\end{align}
by applying the transformation $w=p+(z-p)s$.
Hence
\begin{equation}
 \langle [\gamma _{z} ,\gamma _p] , y \rangle ^{\gamma _p} = e^{2\pi \sqrt{-1} (\theta _p +\kappa )} \langle [\gamma _{z} ,\gamma _p] , y \rangle .
\end{equation}
If $\theta _p +\kappa \in \Zint  $, then the function $ \langle [\gamma _{z} ,\gamma _p] , y \rangle  $ is holomorphic about $z=p$.
Under the assumption $\kappa \not \in \Zint$, the function $\langle [\gamma _{z} ,\gamma _p] , y \rangle  $ is identically zero for any function $y (w)$ written as Eq.(\ref{eq:ywfg}), if and only if $\theta _p \in \Zint _{\geq 0}$ and the singularity $w=p$ is apparent (i.e. $A^{\langle p \rangle} =0$), or $\theta _p +\kappa \in \Zint _{\leq -2} $ and the function $g^{\langle p \rangle}(w)$ in Eq.(\ref{eq:expansfpgp}) is a product of $(w-p) ^{\theta _p }$ and a non-zero polynomial of degree no more than $-\theta _p -\kappa -2$ (i.e. $\tilde{c}^{(p)}_{j}=0$ for $ j \geq -\theta _p -\kappa -1$).
By putting  $\kappa = \kappa _2 -1 $, (resp.  $\theta _p=1-\epsilon ' _p$ $(p=0,1,t)$, $\kappa = -\eta $), we obtain Proposition \ref{prop:nonzero} (i).
If $\theta _p \in \Zint _{\leq -1}$, $\kappa \not \in \Zint$ and the singularity $w=p$ is apparent, then $A^{\langle p \rangle} =0$ and the function $\langle [\gamma _{z} ,\gamma _p] , y \rangle $ is a product of $ (z-p) ^{\theta _p +\kappa +1}$ and a polynomial of degree no more than $-\theta _p-1$. 

Let us consider the local expansion about $w=\infty$.
We assume that the function $y(w)$ is a solution of a second-order linear differential equation about a regular singularity $w= \infty $,
and that the exponents of the second-order linear differential equation at $w=\infty $ are $\theta _{\infty } ^{(1)}$ and $\theta _{\infty} ^{(2)}$.
Then any solution $y (w)$ can be written as 
\begin{equation}
y (w)= C ^{\langle \infty \rangle} f^{\langle \infty \rangle}(w)+ D^{\langle \infty \rangle} g^{\langle \infty \rangle}(w), \quad (C^{\langle \infty  \rangle},D^{\langle \infty  \rangle} \in \Cplx ),
\label{eq:yinftyCD}
\end{equation}
such that
\begin{align}
& f^{\langle \infty \rangle}(w)= \left\{
\begin{array}{ll}
\displaystyle (1/w) ^{\theta _{\infty} ^{(2)}} \sum _{j=0} ^{\infty } c^{(\infty )} _j (1/w)^{j}, & \theta _{\infty } ^{(1)} -\theta _{\infty} ^{(2)}  \not \in \Zint _{\geq 0} \\
\displaystyle (1/w) ^{\theta _{\infty } ^{(1)}} \sum _{j=0} ^{\infty } c^{(\infty )} _j (1/w)^{j}, & \theta _{\infty } ^{(1)} -\theta _{\infty} ^{(2)} \in \Zint _{\geq 0} ,\\
\end{array} 
\right. \label{eq:expansfinftyginfty} \\
& g^{\langle \infty \rangle}(w)= \left\{
\begin{array}{ll}
\displaystyle (1/w) ^{\theta _{\infty } ^{(1)}}\sum _{j=0} ^{\infty } \tilde{c}^{(\infty )} _j (1/w)^{j}, & \theta _{\infty } ^{(1)} -\theta _{\infty} ^{(2)} \not \in \Zint \\
\displaystyle (1/w) ^{\theta _{\infty } ^{(1)}}\left( \sum _{j=0} ^{\infty } \tilde{c}^{(\infty )} _j (1/w)^{j} \right) + A^{\langle p \rangle} f^{\langle p \rangle} (w) \log (1/w) , & \theta _{\infty } ^{(1)} -\theta _{\infty} ^{(2)} \in \Zint _{\leq -1} \\
\displaystyle (1/w) ^{\theta _{\infty} ^{(2)} } \left( \sum _{j=0} ^{\infty } \tilde{c}^{(\infty )} _j (1/w)^{j} \right) + A^{\langle p \rangle} f^{\langle p \rangle} (w) \log (1/w),  & \theta _{\infty } ^{(1)} -\theta _{\infty} ^{(2)} \in \Zint _{\geq 0} ,\\
\end{array}
\right. \nonumber
\end{align}
where $c^{(\infty )} _0= \tilde{c}^{(\infty )} _0=1$.

We investigate the local expansion of the function $\int _{[\gamma _{z} ,\gamma _{\infty }]} y(w) (z-w)^{\theta _{\infty} ^{(2)} -2} dw $ about $w=\infty $ for the case $\theta _{\infty} ^{(2)} -1 \not \in \Zint $.
Since 
\begin{align}
& \int _{[\gamma _{z} ,\gamma _{\infty }]} (1/w) ^{\theta _{\infty} ^{(2)} -1 +\alpha } (z-w)^{\theta _{\infty} ^{(2)} -2} dw = e^{\pi \sqrt{-1} (\theta _{\infty} ^{(2)} -1)} (1/z)^{\alpha } d_{\alpha ,\theta _{\infty} ^{(2)} -1}, \\
& \int _{[\gamma _{z} ,\gamma _{\infty }]} (1/w) ^{\theta _{\infty} ^{(2)} -1 +n } (z-w)^{\theta _{\infty} ^{(2)} -2} (\log (1/w) ) dw = e^{\pi \sqrt{-1} (\theta _{\infty} ^{(2)} -1)} (1/z)^{n} d_{n ,\theta _{\infty} ^{(2)} -1}, \nonumber
\end{align}
for $n \in \Zint _{\geq 1}$, we have
\begin{align}
e^{\pi \sqrt{-1} (1-\theta _{\infty} ^{(2)}) }  \langle [\gamma _{z} ,\gamma _{\infty }] , y \rangle  = \qquad \qquad \qquad \qquad \qquad \qquad \qquad \qquad \qquad \qquad \label{eq:gzginftyexpa} \\
 \left\{ 
\begin{array}{ll}
\displaystyle D^{\langle \infty \rangle} (1/z) ^{\theta _{\infty } ^{(1)} -\theta _{\infty } ^{(2)} +1}  \sum _{j=0} ^{\infty } \tilde{c}^{({\infty })}_j d_{j+\theta _{\infty } ^{(1)} -\theta _{\infty } ^{(2)} +1, \theta _{\infty} ^{(2)} -1} (1/z)^{j}, & \displaystyle \theta _{\infty } ^{(1)} -\theta _{\infty } ^{(2)} \not \in \Zint ,\; \atop{\displaystyle \theta _{\infty } ^{(1)} \not \in \Zint _{\leq 0}}\\
\displaystyle D^{\langle \infty \rangle} (1/z) ^{-\theta _{\infty} ^{(2)}+2} \sum _{j=0} ^{\infty } \tilde{c}^{({\infty })}_{j-\theta _{\infty } ^{(1)}+1} d_{j -\theta _{\infty} ^{(2)} +2, \theta _{\infty} ^{(2)} -1} (1/z)^{j}, & \displaystyle \theta _{\infty } ^{(1)} -\theta _{\infty } ^{(2)} \not \in \Zint , \; \atop{\displaystyle \theta _{\infty } ^{(1)} \in \Zint _{\leq 0}} \\
\displaystyle D^{\langle \infty \rangle} (1/z) ^{\theta _{\infty } ^{(1)} -\theta _{\infty } ^{(2)} +1} \left\{ \displaystyle \sum _{j=0} ^{-\theta _{\infty } ^{(1)} +\theta _{\infty } ^{(2)} -1} \tilde{c}^{({\infty })} _j d_{j+\theta _{\infty } ^{(1)} -\theta _{\infty } ^{(2)} +1, \theta _{\infty} ^{(2)} -1}(1/z)^{j} \quad \quad \atop{\displaystyle + A^{\langle {\infty } \rangle} \! \! \! \! \! \! \! \! \! \! \sum _{j=-\theta _{\infty } ^{(1)} +\theta _{\infty } ^{(2)} } ^{\infty } \! \! \! \! \! \! \! \! \! c^{({\infty })} _{j+\theta _{\infty } ^{(1)} -\theta _{\infty } ^{(2)}} d_{j+\theta _{\infty } ^{(1)} -\theta _{\infty } ^{(2)} +1, \theta _{\infty} ^{(2)} -1} (1/z)^{j} } \right\} , & \displaystyle \theta _{\infty } ^{(1)} -\theta _{\infty } ^{(2)} \in \Zint _{\leq -1},\; \atop{\displaystyle \theta _{\infty } ^{(1)} \not \in \Zint _{\leq 0}}\\
\displaystyle D^{\langle \infty \rangle} A^{\langle {\infty } \rangle} (1/z) ^{\theta _{\infty } ^{(1)} -\theta _{\infty } ^{(2)} +1} \sum _{j=0} ^{\infty } c^{({\infty })} _j d_{j+\theta _{\infty } ^{(1)} -\theta _{\infty } ^{(2)} +1, \theta _{\infty} ^{(2)} -1} (1/z)^{j} , & \displaystyle \theta _{\infty } ^{(1)} -\theta _{\infty } ^{(2)} \in \Zint _{\geq 0} ,\; \atop{\displaystyle \theta _{\infty } ^{(1)} \not \in \Zint _{\leq 0}} .\\
\end{array}
\right. \nonumber
\end{align}
Hence
\begin{equation}
 \langle [\gamma _{z} ,\gamma _{\infty }] , y \rangle ^{\gamma _{\infty }} = e^{2\pi \sqrt{-1} (\theta _{\infty } ^{(1)} -\theta _{\infty } ^{(2)})} \langle [\gamma _{z} ,\gamma _{\infty }] , y \rangle .
\end{equation}
Under the assumption $\theta _{\infty} ^{(2)} \not \in \Zint$, the function $\langle [\gamma _{z} ,\gamma _{\infty }] , y \rangle  $ is identically zero for any function $y (w)$ written as in Eq.(\ref{eq:yinftyCD}), if and only if $\theta _{\infty } ^{(1)} -\theta _{\infty } ^{(2)}\in \Zint _{\geq 0}$ and the singularity $w={\infty }$ is apparent (i.e. $A^{\langle {\infty } \rangle} =0$), or $\theta _{\infty } ^{(1)} \in \Zint _{\leq 0} $ and the function $g^{\langle \infty  \rangle}(w)$ in Eq.(\ref{eq:expansfinftyginfty}) is a non-zero polynomial in the variable $w$ of degree $-\theta _{\infty } ^{(1)}$ (i.e. $\tilde{c}^{({\infty })}_{j}=0$ for $j \geq 1-\theta _{\infty } ^{(1)}$).
By putting $\theta _{\infty} ^{(1)}=\kappa _1$, $\theta _{\infty} ^{(2)}=\kappa _2 +1$ (resp. $\theta _{\infty} ^{(1)}=\alpha +\beta -2\eta +1$, $\theta _{\infty} ^{(2)}=2-\eta $), we obtain Proposition \ref{prop:nonzero} (ii).
If $\theta _{\infty } ^{(1)} -\theta _{\infty } ^{(2)} +1 \in \Zint _{\leq 0}$ and the singularity $w={\infty }$ is apparent, then $A^{\langle {\infty } \rangle} =0$ and the function $\langle [\gamma _{z} ,\gamma _{\infty }] , y \rangle $ is a polynomial in the variable $z$ of degree $-\theta _{\infty } ^{(1)} +\theta _{\infty } ^{(2)} -1$. 

We investigate a sufficient condition that the functions $\langle [\gamma _{z} ,\gamma _0] , y \rangle $, $\langle [\gamma _{z} ,\gamma _1] , y \rangle $, $\langle [\gamma _{z} ,\gamma _t] , y \rangle $ span the two-dimensional space of solutions of $D_{y_1}(\tilde{\theta }_0, \tilde{\theta }_1, \tilde{\theta }_t, \tilde{\theta }_{\infty} ; \tilde{\lambda } ,\tilde{\mu } )$ (resp. Eq.(\ref{Heun02})) for some solution $y(w)$ of $D_{y_1}(\theta _0, \theta _1, \theta _t,\theta _{\infty}; \lambda ,\mu )$ (resp. Eq.(\ref{Heun01})) for the case $\kappa _2 \not \in \Zint $ (resp. $\eta \not \in \Zint $).
\begin{prop} \label{prop:spansp}
Assume that $\kappa _2 \not \in \Zint $ (resp. $\eta \not \in \Zint $), there exists a branching solution of $D_{y_1}(\theta _0, \theta _1, \theta _t,\theta _{\infty}; \lambda ,\mu )$ (resp. Eq.(\ref{Heun01})) for each singularity $w=0,1,t$ (i.e. $\theta _p \not \in \Zint $ (resp. $\epsilon ' _p \not \in \Zint $) or $A^{\langle {p} \rangle} \neq 0$ for $p=0,1,t$), the differential equation $D_{y_1}(\theta _0, \theta _1, \theta _t,\theta _{\infty}; \lambda ,\mu )$ (resp. Eq.(\ref{Heun01})) does not have a solution written as a product of $(w-p) ^{\theta _p} $  (resp. $(w-p) ^{1- \epsilon ' _p}$ ) and a non-zero polynomial on the case $\kappa _2 +\theta _p \in \Zint _{\leq -1}$ (resp. $2-\eta -\epsilon ' _p =1-\epsilon _p \in \Zint _{\leq -1}$, $\epsilon _p \in \Zint _{\geq 2}$) for each $p \in \{0,1,t \}$, and the differential equation $D_{y_1}(\tilde{\theta }_0, \tilde{\theta }_1, \tilde{\theta }_t, \tilde{\theta }_{\infty} ; \tilde{\lambda } ,\tilde{\mu } )$ (resp. Eq.(\ref{Heun02})) does not have a solution written as a product of $z^{\kappa _2 + \theta _0} (z-1) ^{\kappa _2 + \theta _1}(z-t) ^{\kappa _2 + \theta _t}$ (resp. $z^{1 -\epsilon _0} (z-1) ^{1-\epsilon _1}(z-t) ^{1-\epsilon _t}$) and a non-zero polynomial.
Then there exists a solution $y(w)$ of the differential equation $D_{y_1} (\theta _0 , \theta _1 , \theta _t ,\theta _{\infty}; \lambda ,\mu )$ (resp. Eq.(\ref{Heun01})) such that the functions $\langle [\gamma _{z} ,\gamma _0] , y \rangle $, $\langle [\gamma _{z} ,\gamma _1] , y \rangle $, $\langle [\gamma _{z} ,\gamma _t] , y \rangle $ span the two-dimensional space of solutions of the differential equation $D_{y_1}(\tilde{\theta }_0, \tilde{\theta }_1, \tilde{\theta }_t, \tilde{\theta }_{\infty} ; \tilde{\lambda } ,\tilde{\mu } )$ (resp. Eq.(\ref{Heun02})), $\langle [\gamma _{z} ,\gamma _0] , y \rangle \neq 0$, $\langle [\gamma _{z} ,\gamma _1] , y \rangle \neq 0$ and $\langle [\gamma _{z} ,\gamma _t] , y \rangle \neq 0$.
\end{prop}
\begin{proof}
Set $\kappa = \kappa _2 -1$ (resp. $\kappa = -\eta $ and $\epsilon '_p =1 -\theta _p$ $(p = 0,1,t)$).
If $\langle [\gamma _{z} ,\gamma _p] , y \rangle =0$ $(p \in \{0,1,t \} )$ for all solutions of $D_{y_1}(\theta _0, \theta _1, \theta _t,\theta _{\infty}; \lambda ,\mu )$ (resp. Eq.(\ref{Heun01})). Then it follows from Proposition \ref{prop:nonzero} (i) that $\theta _p \in \Zint $ and the singularity $w=p$ is apparent, or $\theta _p +\kappa  +1  \in \Zint _{\leq -1} $ and the differential equation $D_{y_1}(\theta _0, \theta _1, \theta _t,\theta _{\infty}; \lambda ,\mu )$ (resp. Eq.(\ref{Heun01})) has a solution of the form which is a product of $(w-p) ^{\theta _p }$ and a non-zero polynomial of degree no more than $-\theta _p -\kappa -2$.
Hence it follows from the assumptions of Proposition \ref{prop:spansp} that $\langle [\gamma _{z} ,\gamma _p] , y^{(p)}\rangle  \neq 0$ for some solution $y^{(p)} (w)$ for each $p \in \{0,1,t \}$.
By setting $y(w)=c_0 y^{(0)}(w) +c_1 y^{(1)}(w) +c_t y^{(t)}(w)$ and choosing constants $c_0, c_1 ,c_t$ appropriately, we have $\langle [\gamma _{z} ,\gamma _p] , y\rangle  \neq 0$ for all $p \in \{0,1,t\}$.
Assume that the functions $\langle [\gamma _{z} ,\gamma _0] , y \rangle $, $\langle [\gamma _{z} ,\gamma _1] , y \rangle $, $\langle [\gamma _{z} ,\gamma _t] , y \rangle $ do not span the space of solutions of $D_{y_1}(\tilde{\theta }_0, \tilde{\theta }_1, \tilde{\theta }_t, \tilde{\theta }_{\infty} ; \tilde{\lambda } ,\tilde{\mu } )$ (resp. Eq.(\ref{Heun02})).
Then $\langle [\gamma _{z} ,\gamma _0] , y \rangle =d\langle [\gamma _{z} ,\gamma _1] , y \rangle =d'\langle [\gamma _{z} ,\gamma _t] , y \rangle $ for some $d \neq 0$ and $d' \neq 0$.
Since $\langle [\gamma _{z} ,\gamma _0] , y \rangle $ satisfies the differential equation $D_{y_1}(\tilde{\theta }_0, \tilde{\theta }_1, \tilde{\theta }_t, \tilde{\theta }_{\infty} ; \tilde{\lambda } ,\tilde{\mu } )$ (resp. Eq.(\ref{Heun02})), $\langle [\gamma _{z} ,\gamma _0] , y \rangle $ is locally holomorphic in $\Cplx \setminus \{0,1,t \}$, and it follows from Eq.(\ref{eq:gzgpexpa}) that the function $z^{-\theta _0-\kappa -1} (z-1)^{-\theta _1-\kappa -1}(z-t)^{-\theta _t-\kappa -1}\langle [\gamma _{z} ,\gamma _0] , y \rangle $ is holomorphic in $\Cplx $, and the singurality $z=\infty $ is regular at most and apparent.
Hence $z^{-\theta _0-\kappa -1} (z-1)^{-\theta _1-\kappa -1}(z-t)^{-\theta _t-\kappa -1 }\langle [\gamma _{z} ,\gamma _0] , y \rangle $ is a polynomial, and $\langle [\gamma _{z} ,\gamma _0] , y \rangle =z^{\theta _0+\kappa +1} (z-1)^{\theta _1+\kappa +1}(z-t)^{\theta _t+\kappa +1} h(z)$ for some polynomial $h(z)$.
But this contradicts the assumptions of the proposition.
\end{proof}
\begin{cor} \label{cor:spansp} (Proposition \ref{prop:spansp0})
There exists a solution $y(w)$ of $D_{y_1}(\theta _0, \theta _1, \theta _t,\theta _{\infty}; \lambda ,\mu )$ (resp. Eq.(\ref{Heun01})) such that  $\langle [\gamma _{z} ,\gamma _0] , y \rangle \neq 0$, $\langle [\gamma _{z} ,\gamma _1] , y \rangle \neq 0$, $\langle [\gamma _{z} ,\gamma _t] , y \rangle \neq 0$ and the functions $\langle [\gamma _{z} ,\gamma _0] , y \rangle $, $\langle [\gamma _{z} ,\gamma _1] , y \rangle $, $\langle [\gamma _{z} ,\gamma _t] , y \rangle $ span the two-dimensional space of solutions of $D_{y_1}(\tilde{\theta }_0, \tilde{\theta }_1, \tilde{\theta }_t, \tilde{\theta }_{\infty} ; \tilde{\lambda } ,\tilde{\mu } )$ (resp. Eq.(\ref{Heun02})), 
if $\kappa _2 \not \in \Zint$ and $\theta _p , \tilde{\theta }_p \not \in \Zint$ for all $p \in \{ 0,1,t, \infty \}$ (resp. $\eta , \epsilon _0 , \epsilon _1, \epsilon _t, \alpha -\beta , \epsilon _0 ',\epsilon _1',\epsilon _t ' ,\alpha '-\beta ' \not \in \Zint$).
\end{cor}
\begin{proof}
It follows from the fact that $\theta _0, \theta _1, \theta _t \not \in \Zint $ (resp. $\epsilon _0 ',\epsilon _1',\epsilon _t '\not \in \Zint$) that there exists a branching solution of $D_{y_1}(\theta _0, \theta _1, \theta _t,\theta _{\infty}; \lambda ,\mu )$ (resp. Eq.(\ref{Heun01})).
If there exists a solution of $D_{y_1}(\tilde{\theta }_0, \tilde{\theta }_1, \tilde{\theta }_t, \tilde{\theta }_{\infty} ; \tilde{\lambda } ,\tilde{\mu } )$ (resp. Eq.(\ref{Heun02})) that can be written as $z^{\kappa _2 + \theta _0} (z-1) ^{\kappa _2 + \theta _1}(z-t) ^{\kappa _2+ \theta _t } h(z)$ (resp. $z^{1-\epsilon _0} (z-1) ^{1 -\epsilon _1}(z-t) ^{1 -\epsilon _t } h(z)$) for some non-zero polynomial $h(z)$,
it follows from Proposition \ref{prop:monoredpolynsol} (ii) that
$\theta _{\infty } +(\kappa _2+\theta _0+\kappa _2+\theta _1+\kappa _2+\theta _t) = -\deg h(z) \in \Zint _{\leq 0}$ or $-\kappa _2 +1 +(\kappa _2+\theta _0+\kappa _2+\theta _1+\kappa _2+\theta _t)  = -\deg h(z) \in \Zint _{\leq 0}$ (resp. $\alpha + (3- \epsilon _0-\epsilon _1-\epsilon _t)= -\deg h(z) \in \Zint _{\leq 0}$ or $\beta + (3- \epsilon _0-\epsilon _1-\epsilon _t)= -\deg h(z) \in \Zint _{\leq 0}$), i.e. $\kappa _2 \in \Zint _{\leq 0}$ or $\theta _{\infty } \in \Zint _{\geq 1}$ (resp. $2-\beta \in \Zint _{\leq 0}$ or $2-\alpha \in \Zint _{\leq 0}$), which contradicts the assumption of the corollary.
The condition $\tilde{\theta }_p \in \Zint _{\leq -1}$ (resp. $1-\epsilon _p \in \Zint _{\leq -1}$) for $p =0,1,t$ is covered in the assumption of the corollary.
Thus, the assumption of Proposition \ref{prop:spansp} (i) follows from the assumption of the corollary, and the corollary is obtained by applying Proposition \ref{prop:spansp} (i).
\end{proof}

We derive the following proposition which is used to prove Theorem \ref{thm:nonlog-polynHeun}.
\begin{prop} \label{prop:nonlog-polynHeun}
Let $a,b,c$ be elements of $\{0,1,t \}$ such that $a \neq b \neq c \neq a$ and $\eta , \alpha ,\beta ,\epsilon _0 ,\epsilon _1 ,\epsilon _t,  \alpha ',\beta ',\epsilon '_0 ,\epsilon '_1 ,\epsilon '_t$ be the parameters defined in Eq.(\ref{eq:mualbe}) or Eq.(\ref{eq:mualbe18}).\\
(i) If $\epsilon ' _a \in \Zint _{\geq 2}$, $\eta \not \in \Zint$ and the singularity $w=a$ of Eq.(\ref{Heun01}) is apparent, then there exists a non-zero solution of Eq.(\ref{Heun02}) which can be written as $(z-a) ^{1- \epsilon _a}  h(z) $ where $h(z)$ is a polynomial of degree no more than $\epsilon ' _a -2$.
Moreover if $\alpha ' ,\beta ' \not \in \Zint$, then $\deg _E h(z)= \epsilon ' _a -2$.\\
(ii) If $\epsilon ' _a \in \Zint _{\leq 0}$, $\eta  \not \in \Zint$, the singularity $w=a$ of Eq.(\ref{Heun01}) is apparent and there do not exist any non-zero solutions of Eq.(\ref{Heun01}) written in the form $(w-b)^{\alpha _b}(w-c)^{\alpha _c} p(w)$, where $p(w)$ is a polynomial and $(\alpha _b, \alpha _c)=(0,0)$, $( 1-\epsilon ' _b, 0)$ or $(0,1-\epsilon ' _c )$, then there exists a non-zero solution of Eq.(\ref{Heun02}) which can be written as $(z-b) ^{1-\epsilon  _b}(z-c) ^{1-\epsilon  _c}  h(z) $ where $h(z)$ is a polynomial.
Moreover if $\alpha ' ,\beta ' \not \in \Zint$, then $\deg h(z)= -\epsilon ' _a$.\\
(iii) If $\epsilon _a \in \Zint _{\geq 2}$, $\eta \not \in \Zint$, there exists a non-zero solution of Eq.(\ref{Heun01}) which can be written as $(w-a) ^{1-\epsilon ' _a}h(w)$, where $h(w)$ is a polynomial and there do not exist any non-zero solutions of Eq.(\ref{Heun02}) written as polynomials in $z$, then the singularity $z=a$ of Eq.(\ref{Heun02}) is apparent.\\
(iv) If $\epsilon _a \in \Zint _{\leq 0}$, $\eta , \epsilon ' _b ,\epsilon ' _c \not \in \Zint$, there exists a non-zero solution of Eq.(\ref{Heun01}) written as a product of $(w-b) ^{ 1-\epsilon ' _b}(w-c) ^{1-\epsilon ' _c}$ and a polynomial, and there do not exist any non-zero solutions of Eq.(\ref{Heun02}) written as a product of $z^{1-\epsilon _0}(z-1)^{1-\epsilon _1}(z-t)^{1-\epsilon _t}$ and a polynomial, then the singularity $z=a$ of Eq.(\ref{Heun02}) is apparent.\\
(v) If $\alpha +\beta -\eta \in \Zint _{\leq 0}$, $\eta \not \in \Zint$ and the singularity $w=\infty $ of Eq.(\ref{Heun01}) is apparent, then there exists a non-zero solution of Eq.(\ref{Heun02}) which can be written as a polynomial of degree $\eta -\alpha -\beta $.\\
(vi) If $\alpha +\beta -\eta \in \Zint _{\geq 2}$, $\eta \not \in \Zint$, the singularity $w=\infty $ of Eq.(\ref{Heun01}) is apparent and there do not exist any non-zero solution of Eq.(\ref{Heun01}) written as $w^{\alpha _0} (w-1)^{\alpha _1} (w-t)^{\alpha _t} p(w)$ such that $p(w)$ is a polynomial and $(\alpha _0, \alpha _1, \alpha _t)=(1-\epsilon '  _0,0,0),$ $(0,1-\epsilon ' _1,0)$ or $(0,0,1-\epsilon ' _t)$, then there exists a non-zero solution of Eq.(\ref{Heun02}) which can be written as $z ^{1-\epsilon _0 } (z-1) ^{1-\epsilon _1} (z-t) ^{1-\epsilon _t}  h(z)$ where $h(z)$ is a polynomial of degree $\alpha +\beta -\eta -2$.\\
(vii) If $\alpha +\beta -2\eta \in \Zint _{\leq -1}$, $\eta \not \in \Zint$, there exists a non-zero solution of Eq.(\ref{Heun01}) which is written as a polynomial and there do not exist any non-zero solutions of Eq.(\ref{Heun02}) which are written as $(1/z)^{\eta }p(1/z)$ where $p(1/z) $ is a polynomial in $1/z$, then the singularity $z=\infty $ of Eq.(\ref{Heun02}) is apparent.\\
(viii) If $\alpha +\beta -2\eta \in \Zint _{\geq 1}$, $\eta , \epsilon ' _0 , \epsilon ' _1 ,\epsilon ' _t \not \in \Zint$, there exists a non-zero solution of Eq.(\ref{Heun01}) which can be written as a product of $w^{1-\epsilon ' _0} (w-1) ^{1-\epsilon ' _1 }(w-t) ^{1-\epsilon ' _t } $ and a polynomial, then the singularity $z=\infty $ of Eq.(\ref{Heun02}) is apparent.
\end{prop}
\begin{proof}
Set $\epsilon ' _p=1-\theta _p$, $\epsilon _p=1-\tilde{\theta }_p$ $(p=0,1,t)$.
Then $\tilde{\theta }_p = \theta _p- \eta +1$.
We apply local expansions in this appendix by setting $\kappa =- \eta $.

To prove (i), it follows from Proposition \ref{prop:nonlog-polyn01} that it remains to show that if $\alpha ' ,\beta ' \not \in \Zint$, then $\deg _E h(z)= \epsilon ' _a -2$.
If there exists a non-zero solution of Eq.(\ref{Heun02}) which is written as $(z-a) ^{1- \epsilon _a}  h(z) $ where $h(z)$ is a polynomial, then it follows from Proposition \ref{prop:monoredpolynsol} (ii) that $\deg _E h(z)= -\alpha -1 + \epsilon _a$ or $ -\beta -1 + \epsilon _a$, i.e. $\deg _E h(z)= -\eta -1 + \epsilon _a = \epsilon ' _a -2$ or $ -(\alpha + \beta -\eta )-1 + \epsilon _a = 2\eta -\alpha - \beta + \epsilon ' _a -2$.
If $\alpha ' ,\beta ' \not \in \Zint$, then we have $2\eta -\alpha - \beta  \not \in \Zint$ and $\deg _E h(z)= \epsilon ' _a -2$.

If $\theta _a \in \Zint _{\geq 1}$, $\eta  \not \in \Zint$ and the singularity $w=a$ of Eq.(\ref{Heun01}) is apparent, then it follows from Proposition \ref{prop:nonzero} (i) that $\langle [\gamma _{z} ,\gamma _a] , y \rangle=0 $ for all solutions $y(w)$ of Eq.(\ref{Heun01}).
Combining this result with Eq.(\ref{eq:gzgbga}) and a similar equality, we have 
\begin{align}
& \langle [\gamma _{z} ,\gamma _b] , y \rangle ^{\gamma _a}=\langle [\gamma _{z} ,\gamma _b] , y \rangle, \quad \langle [\gamma _{z} ,\gamma _c] , y \rangle ^{\gamma _a}=\langle [\gamma _{z} ,\gamma _c] , y \rangle .
\label{eq:bcgatriv}
\end{align}
If $\langle [\gamma _{z} ,\gamma _b] , y \rangle$, $\langle [\gamma _{z} ,\gamma _c] , y \rangle $ are linearly independent for some solution $y(w)$ of Eq.(\ref{Heun01}), then it follows from Eq.(\ref{eq:bcgatriv}) that the monodromy matrix about $z=a$ is a unit and the exponents of Eq.(\ref{Heun02}) at $z=a$ are integers.
Hence $\tilde{\theta }_a= \theta _a - \eta  +1 \in \Zint$, and this contradicts $\eta  \not \in \Zint$.
Therefore $\langle [\gamma _{z} ,\gamma _b] , y \rangle$, $\langle [\gamma _{z} ,\gamma _c] , y \rangle $ are linearly dependent for any solution $f(w)$.
If $\langle [\gamma _{z} ,\gamma _b] , y \rangle \neq 0$ for some solution $y(w)$ and $\langle [\gamma _{z} ,\gamma _c] , y \rangle \neq 0$ for some solution $y(w)$, then there exists a solution $y(w)$ of Eq.(\ref{Heun01}) such that $\langle [\gamma _{z} ,\gamma _c] , y \rangle =d' \langle [\gamma _{z} ,\gamma _b] , y \rangle \neq 0$ for some constant $d' \neq 0$.
It follows from local expansions (Eq.(\ref{eq:gzgpexpa})) for the case $p=b,c$, Eq.(\ref{eq:bcgatriv}) and the condition $\tilde{\theta }_a \not \in \Zint $ that the function $h(z)= (z-b)^{-\tilde{\theta }_b}(z-c)^{-\tilde{\theta }_c} \langle [\gamma _{z} ,\gamma _b] , y \rangle$ is holomorphic in $\Cplx $.
Hence $h(z)$ is a non-branching function in $\Cplx \cup \{ \infty \}$ which may have a pole at $z=\infty $, and Eq.(\ref{Heun02}) has a non-zero solution $(z-b)^{\tilde{\theta }_b}(z-c)^{\tilde{\theta }_c} h(z)$, where $h(z)$ is a polynomial. 
Let $k$ be the degree of $h(z)$.
It follows from Proposition \ref{prop:monoredpolynsol} (ii) that  $-k-\tilde{\theta }_b-\tilde{\theta }_c =\eta $ or $- \eta +\alpha +\beta$, and by applying the relation $\tilde{\theta }_a +\tilde{\theta }_b +\tilde{\theta }_b +\alpha +\beta =2$ we have $k=\tilde{\theta }_a +\alpha +\beta -\eta -2= \theta _a -2\eta +\alpha +\beta -1$ or $k=\tilde{\theta }_a +\eta -2 =\theta _a-1$.
Hence, if $\alpha ' ,\beta' \not \in \Zint $, then $- 2\eta +\alpha +\beta \not \in \Zint $ and we have $\deg h(z)=\theta _a -1 =-\epsilon _a '$.
If $\langle [\gamma _{z} ,\gamma _b] , y \rangle = 0$ for all solutions $y(w)$, 
then it follows from Proposition \ref{prop:nonzero} (i) that $\theta _b \in \Zint _{\geq 0}$ and the singularity $w=b$ is apparent or $\epsilon _b \in \Zint _{\geq 2} $ and the Eq.(\ref{Heun01}) has a solution of the form which is a product of $(w-b) ^{\theta _b}$ and a non-zero polynomial of degree no more than $\epsilon _b -2$.
For the case $\theta _b \in \Zint _{\geq 0}$ and the singularity $w=b$ is apparent, by taking a solution $y(w)$ of Eq.(\ref{Heun01}) which is holomorphic at $w=c$, the function $y(w)$ is holomorphic on the points $w=a,b,c$ and it is a polynomial in $w$, because the point $w=\infty $ is an apparent singularity.
Hence we have a polynomial solution $y(w)$ of Eq.(\ref{Heun01}).
By combining this result with a similar statement for the case $\langle [\gamma _{z} ,\gamma _c] , y \rangle = 0$ for all solutions $y(w)$,
it follows that if $\theta _a \in \Zint _{\geq 0}$, the singularity $w=a$ is apparent and $\langle [\gamma _{z} ,\gamma _b] , y \rangle =0$ or $\langle [\gamma _{z} ,\gamma _c] , y \rangle = 0$ for all solutions $y(w)$, then there exists a non-zero solution of Eq.(\ref{Heun01}) which can be written as $(w-b)^{\alpha _b}(w-c)^{\alpha _c} p(w)$ where $p(w)$ is a polynomial and $(\alpha _b, \alpha _c)=(0,0)$, $(\theta _b, 0)$ or $(0,\theta _c )$.
Therefore we obtain (ii).

We show that if $\theta _a (=1-\epsilon '_a ) \in \Zint _{\leq 0}$, $\eta  \not \in \Zint$, 
there exists a logarithmic solution of Eq.(\ref{Heun01}) about $w=a$ and there do not exist any non-zero solutions of Eq.(\ref{Heun01}) written as a polynomial, then there do not exist any non-zero solution of Eq.(\ref{Heun02}) written as $(z-a)^{1-\epsilon _a} p(z)$ such that $p(z)$ is a polynomial.
We write a logarithmic solution of Eq.(\ref{Heun01}) as in Eqs.(\ref{eq:ywfg}), (\ref{eq:expansfpgp}).
Then $D^{\langle a \rangle} \neq 0 $, $A^{\langle a \rangle} \neq 0 $ and it follows from the absence of a non-zero polynomial solution of Eq.(\ref{Heun01}) that $\forall K \in \Zint$, $\exists j \in \Zint _{\geq K}$ such that $c_j^{(a)} \neq 0$.
A solution $\langle [\gamma _{z} ,\gamma _a] , y \rangle  $ of Eq.(\ref{Heun02}) can be written as Eq.(\ref{eq:gzgpexpa}) for the case $\theta _a \in \Zint _{\leq 0}$, $\theta _a +\kappa  +1 \not \in \Zint _{\leq -1}$, and it cannot be written as $(z-a)^{\theta _a - \eta +1} p(z)$ such that $p(z)$ is a polynomial because $A^{\langle a \rangle} \neq 0 $ and $\forall K \in \Zint$, $\exists j \in \Zint _{\geq K}$ such that $c_j^{(a)} \neq 0$.
Since $(1-\epsilon _a = )\tilde{\theta }_a =  \theta _a - \eta +1 \not \in \Zint$, the space of solutions of Eq.(\ref{Heun02}) that are written as $(z-a)^{\tilde{\theta }_a} h(z)$ such that $h(z)$ is holomorphic about $z=a$ is one-dimensional.
Hence there does not exist a non-zero solution of Eq.(\ref{Heun02}) written as $(z-a)^{1-\epsilon _a} p(z)$ such that $p(z)$ is a polynomial.
It follows from the duality of the parameters $(\epsilon _0, \epsilon _1, \epsilon _t, \eta )$ and $(\epsilon '_0, \epsilon '_1, \epsilon '_t, \eta ')$ in Eqs.(\ref{eq:mualbe}), (\ref{eq:mualbe18}) that we obtain (iii) for the case $1- \epsilon _a \in \Zint _{\leq 0}$ by contraposition.

We show that if $\eta  \not \in \Zint$, $\tilde{\theta }_a \in \Zint _{\geq 1}$, $\theta _b ,\theta _c \not \in \Zint$,
there exists a logarithmic solution of Eq.(\ref{Heun02}) about $z=a$,
there exists a non-zero solution of Eq.(\ref{Heun01}) written as $(w-b)^{\theta _b} (w-c)^{\theta _c} h(w)$ such that $h(w)$ is a polynomial,
and there do not exist any non-zero solutions of Eq.(\ref{Heun02}) written as $(z-a)^{\tilde{\theta }_a} (z-b)^{\tilde{\theta }_b} (z-c)^{\tilde{\theta }_c} \tilde{p}_0(z)$ where $\tilde{p}_0(z)$ is a polynomial,
then we have a contradiction.
Assume that there exists a non-zero solution of Eq.(\ref{Heun01}) written as $y(w)=(w-b)^{\theta _b} (w-c)^{\theta _c} h(w)$ such that $h(w)$ is a polynomial,
Then the functions $\langle [\gamma _{z} ,\gamma _b] , y \rangle  $, $\langle [\gamma _{z} ,\gamma _c] , y \rangle $ are solutions of Eq.(\ref{Heun02}) and they are non-zero, which follows from $\theta _b ,\theta _c \not \in \Zint$ and Eq.(\ref{eq:gzgpexpa}).
Since it has been shown that $y ^{\gamma _b} =e^{2\pi \sqrt{-1} \theta _b}y$, $y ^{\gamma _c} =e^{2\pi \sqrt{-1}\theta _c}y$ and $\langle [\gamma _{z} ,\gamma _a] , y \rangle  =0$, we have $\langle [\gamma _{z} ,\gamma _b] , y \rangle ^{\gamma _a} =\langle [\gamma _{z} ,\gamma _b] , y \rangle$, $\langle [\gamma _{z} ,\gamma _c] , y \rangle ^{\gamma _a} =\langle [\gamma _{z} ,\gamma _c] , y \rangle$.
If $\langle [\gamma _{z} ,\gamma _b] , y \rangle  $, $\langle [\gamma _{z} ,\gamma _c] , y \rangle $ are linearly independent, the monodromy matrix about $z=a$ is a unit, and this contradicts the existence of a logarithmic solution.
Hence $\langle [\gamma _{z} ,\gamma _b] , y \rangle  $, $\langle [\gamma _{z} ,\gamma _c] , y \rangle $ are linearly dependent.
It follows from a similar argument to the proof of (ii) that there exists a non-zero solution $\tilde{y}(z)$ of Eq.(\ref{Heun02}) written as $\tilde{y}(z)=(z-b)^{\tilde{\theta }_b} (z-c)^{\tilde{\theta }_c} \tilde{p}(z)$ such that $\tilde{p}(z)$ is a polynomial.
Since $\tilde{\theta }_a \in \Zint _{\geq 0}$ and there exists a logarithmic solution of Eq.(\ref{Heun02}) about $z=a$, it follows from Eq.(\ref{eq:expansfpgp}) that $y(z)$ can be expressed as $y(w)=(z-a)^{\tilde{\theta }_a} (z-b)^{\tilde{\theta }_b} (z-c)^{\tilde{\theta }_c} \tilde{p}_0(z)$ such that $\tilde{p}_0(z)$ is a polynomial, and we have a contradiction.
Hence we obtain that if $\eta  \not \in \Zint$, $\tilde{\theta }_a \in \Zint _{\geq 0}$, $\theta _b ,\theta _c \not \in \Zint$,
there exists a non-zero solution of Eq.(\ref{Heun01}) written as $(w-b)^{\theta _b} (w-c)^{\theta _c} h(w)$ such that $h(w)$ is a polynomial,
and there do not exist any non-zero solutions of Eq.(\ref{Heun02}) written in the form $(z-a)^{\tilde{\theta }_a} (z-b)^{\tilde{\theta }_b} (z-c)^{\tilde{\theta }_c} \tilde{p}_0(z)$, where $\tilde{p}_0(z)$ is a polynomial,
then the singularity $z=a$ of Eq.(\ref{Heun02}) is apparent.
Therefore we have (iv).

We show (v) and (vi).
We apply the expansions in Eqs.(\ref{eq:expansfinftyginfty}), (\ref{eq:gzginftyexpa}) by setting $\theta _{\infty} ^{(1)}= \alpha +\beta -2\eta +1$ and $\theta _{\infty} ^{(2)}= 2- \eta$.
If $\alpha +\beta -\eta  =\theta _{\infty} ^{(1)}-\theta _{\infty} ^{(2)} + 1 \in \Zint _{\leq 0}$, $\eta  \not \in \Zint$ and the singularity $w={\infty }$ of Eq.(\ref{Heun01}) is apparent, then $\theta _{\infty} ^{(1)} \not \in \Zint _{\leq 0}$ and the function $\langle [\gamma _{z} ,\gamma _{\infty }] , y \rangle $ in Eq.(\ref{eq:gzginftyexpa}) is a product of $ (1/z) ^{\alpha +\beta -\eta }$ and a polynomial in the variable $1/z$ of degree no more than $\eta -\alpha -\beta $, and it satisfies Eq.(\ref{Heun02}).
Hence there exists a solution of Eq.(\ref{Heun02}) which is a polynomial in $z$ of degree no more than $\eta -\alpha -\beta $.
If there exists a solution of Eq.(\ref{Heun02}) which is a polynomial in $z$, then the degree of the polynomial is $-\alpha $ or $-\beta $, i.e., $-\eta (\not \in \Zint)$ or $\eta -\alpha -\beta (\in \Zint)$.
Therefore we have (v).

If $- \eta +\alpha +\beta  = \theta _{\infty} ^{(1)}-\theta _{\infty} ^{(2)} \in \Zint _{\geq 0}$, $\eta  \not \in \Zint$ and the singularity $w=\infty $ of Eq.(\ref{Heun01}) is apparent, then it follows from Proposition \ref{prop:nonzero} (ii) and a similar argument to obtain Eq.(\ref{eq:bcgatriv}) that $\langle [\gamma _{z} ,\gamma _\infty ] , y \rangle=0 $ for all solutions $y(w)$ of Eq.(\ref{Heun01}) and 
\begin{align}
& \langle [\gamma _{z} ,\gamma _p] , y \rangle ^{\gamma _{\infty}}=e^{2\pi \sqrt{-1} \eta } \langle [\gamma _{z} ,\gamma _p] , y \rangle, \quad p=0,1,t.
\label{eq:bcginftytriv} 
\end{align}
If $\langle [\gamma _{z} ,\gamma _a] , y \rangle$, $\langle [\gamma _{z} ,\gamma _b] , y \rangle $ $(a,b \in \{0,1,t \}$, $a\neq b)$ are linearly independent for some solution $y(w)$ of Eq.(\ref{Heun01}), it follows from Eq.(\ref{eq:bcginftytriv}) that the monodromy matrix of Eq.(\ref{Heun02}) about $z=\infty$ is scalar,
the difference between the exponents of Eq.(\ref{Heun02}) at $z = \infty $ (i.e. $\theta _{\infty } ^{(1)} -\theta _{\infty } ^{(2)}+1$ and $2-\theta _{\infty } ^{(2)}$) is an integer, and this contradicts $\theta _{\infty} ^{(1)}-\theta _{\infty} ^{(2)} \in \Zint $ and $ \eta =2- \theta _{\infty} ^{(2)} \not \in \Zint$.
Therefore $\langle [\gamma _{z} ,\gamma _a] , y \rangle$, $\langle [\gamma _{z} ,\gamma _b] , y \rangle $ are linearly dependent for any solution $y(w)$ of Eq.(\ref{Heun01}) and $a,b \in \{0,1,t\}$ such that $a\neq b$.
If there exists a solution $y^{(p)}(w)$ of Eq.(\ref{Heun01}) such that $\langle [\gamma _{z} ,\gamma _p] , y ^{(p)} \rangle \neq 0$ for each $p \in \{0,1,t\}$, then there exists a solution $y(w)$ of Eq.(\ref{Heun01}) such that $\langle [\gamma _{z} ,\gamma _p] , y \rangle \neq 0$ for any $p \in \{0,1,t \}$ by setting $y(w)=c_0y^{(0)}(w)+c_1y^{(1)}(w)+c_ty^{(t)}(w)$ and choosing $c_0, c_1, c_t$ appropriately.
It follows from  $\langle [\gamma _{z} ,\gamma _0] , y \rangle =d \langle [\gamma _{z} ,\gamma _1] , y \rangle =d' \langle [\gamma _{z} ,\gamma _t] , y \rangle \neq 0$ for some constants $d,d' \neq 0$.
It is shown that the function $z^{-\tilde{\theta }_0}(z-1)^{-\tilde{\theta }_1}(z-t)^{-\tilde{\theta }_t} \langle [\gamma _{z} ,\gamma _0] , y \rangle$ is holomorphic in $\Cplx $, and Eq.(\ref{Heun02}) has a non-zero solution $z^{\tilde{\theta }_0}(z-1)^{\tilde{\theta }_1}(z-t)^{\tilde{\theta }_t} h(z)$ where $h(z)$ is a polynomial. 
Let $k$ be the degree of $h(z)$.
It follows from Proposition \ref{prop:monoredpolynsol} (ii) that $-k +\alpha +\beta -2= \eta $ or $\alpha +\beta -\eta $.
Since $\eta \not \in \Zint$, $\deg h(z)=\alpha +\beta  -\eta -2$.
If $\langle [\gamma _{z} ,\gamma _p] , y \rangle = 0$ for all solutions $y(w)$ and some $p \in \{0,1,t\}$,
then there exists a solution of Eq.(\ref{Heun01}) which can be expressed as a product of $(w-p)^{\theta _p}$ and a polynomial, or $\theta _p \in \Zint _{\geq 1}$ and there are no logarithmic solutions about $w=p$.
Assume that  $\theta _p \in \Zint _{\geq 1}$ and there are no logarithmic solutions about $w=p$.
Let $p' \in \{0,1,t \}$ such that $p' \neq p$ and $y(w)$ be a  solution of Eq.(\ref{Heun01}) which is holomorphic at $w=p'$.
Then $y^{\gamma _{p'}}(w)=y(w) $, $y^{\gamma _{p}}(w)=y(w) $.
Since the singularity $w=\infty $ is apparent, we have $y^{\gamma _{\infty }}(w)=e^{2\pi \sqrt{-1} \theta _{\infty }^{(2)}}y(w) =e^{-2\pi \sqrt{-1} \eta }y(w) $ and it follows that $y^{\gamma _{p''}}(w)=e^{2\pi \sqrt{-1} \eta }y(w) $ $(p'' \in \{ 0,1,t\},$ $p\neq p''\neq p' )$,
and then, since $e^{2\pi \sqrt{-1} \eta } \neq 1$, that $y^{\gamma _{p''}}(w)=e^{2\pi \sqrt{-1} \theta _{p''}}y(w) $.
Hence the function $y(w)$ can be expressed as $y(w)=(w-p'') ^{\theta _{p''}} h(w)$ such that $h(w)$ is a polynomial, which follows from the monodromy of $y(w)$.
Therefore if $\langle [\gamma _{z} ,\gamma _p] , y \rangle = 0$ for all solution $y(w)$ and some $p \in \{0,1,t\}$ then there exists a solution $y(w)$ of Eq.(\ref{Heun01}) such that $y(w) =w^{{\alpha}_0}(w-1)^{{\alpha}_1}(w-t)^{\alpha_t} h(w)$, $h(w)$ is a polynomial and $(\alpha _0, \alpha _1, \alpha _t)=(\theta _0,0,0),$ $(0,\theta _1,0)$ or $(0,0,\theta _t)$, and we have (vi).

We show (vii) and (viii).
We apply the expansions in Eqs.(\ref{eq:expansfinftyginfty}), (\ref{eq:gzginftyexpa}) by setting $\theta _{\infty} ^{(1)}= \alpha +\beta -2\eta +1$ and $\theta _{\infty} ^{(2)}= 2- \eta$.
We show that if $\eta  (=2-\eta '=2- \theta _{\infty} ^{(2)}) \not \in \Zint$, $\alpha +\beta -\eta (=\alpha ' +\beta ' -2\eta ' +1= \theta _{\infty} ^{(1)}-\theta _{\infty} ^{(2)} + 1)\in \Zint _{\leq 0}$, 
there exists a logarithmic solution of Eq.(\ref{Heun01}) about $w={\infty }$ and there do not exist any non-zero solutions of Eq.(\ref{Heun01}) written as a product of $(1/w)^{\eta '}(=(1/w)^{\theta _{\infty} ^{(2)}})$ and a polynomial in $1/w$, then there do not exist any non-zero solutions of Eq.(\ref{Heun02}) written as a polynomial.
We write a solution of Eq.(\ref{Heun01}) as in Eqs.(\ref{eq:yinftyCD}), (\ref{eq:expansfinftyginfty}).
Then $D^{\langle {\infty } \rangle} \neq 0 $, $A^{\langle {\infty } \rangle} \neq 0 $ and $\forall K \in \Zint$, $\exists j \in \Zint _{\geq K}$ such that $c_j^{({\infty })} \neq 0$ in Eq.(\ref{eq:expansfinftyginfty}).
It can be shown as in the proof of (iii) that the function $\langle [\gamma _{z} ,\gamma _{\infty }] , y \rangle $ can be written as in Eq.(\ref{eq:gzginftyexpa}) for the case $\theta _{\infty} ^{(1)}-\theta _{\infty} ^{(2)} \in \Zint _{\leq -1}$, $\theta _{\infty} ^{(1)} \not \in \Zint _{\leq 0}$, and it is not written as $(1/z)^{\theta _{\infty} ^{(1)}-\theta _{\infty} ^{(2)} +1 } p(1/z)$ such that $p(z)$ is a polynomial, and
it follows from $\theta _{\infty} ^{(1)}-\theta _{\infty} ^{(2)} + 1-(2-\theta _{\infty} ^{(2)} ) \not \in \Zint$ that
there do not exist any non-zero solutions of Eq.(\ref{Heun02}) written as $(1/z)^{\theta _{\infty} ^{(1)}-\theta _{\infty} ^{(2)} + 1 } p(1/z)$ such that $p(z)$ is a polynomial.
If there exists a non-zero solution of Eq.(\ref{Heun02}) written as a polynomial $h(z)$, then it follows from Proposition \ref{prop:monoredpolynsol} (ii) that $\deg h(z) =\alpha +\beta -\eta  $ or $\eta $.
Because $\eta  \not \in \Zint$, $\deg h(z) =\alpha +\beta -\eta = \theta _{\infty} ^{(1)}-\theta _{\infty} ^{(2)} + 1$  and $h(z)$ can be written as $(1/z)^{\theta _{\infty} ^{(1)}-\theta _{\infty} ^{(2)} + 1} p(1/z)$ where $p(z)$ is a polynomial of degree no more than $\theta _{\infty} ^{(1)}-\theta _{\infty} ^{(2)} + 1$.
Therefore we obtain the result that there do not exist any non-zero solutions of Eq.(\ref{Heun02}) written as a polynomial.
It follows from the duality of the parameters $(\epsilon _0, \epsilon _1, \epsilon _t, \eta )$ and $(\epsilon '_0, \epsilon '_1, \epsilon '_t, \eta ')$ in Eqs.(\ref{eq:mualbe}), (\ref{eq:mualbe18}) that we obtain (vii) by contraposition.

We show that if $\eta , \theta _0, \theta _1, \theta _t  \not \in \Zint$, $\alpha +\beta -2\eta (=\theta _{\infty} ^{(1)}-\theta _{\infty} ^{(2)} + 1-(2-\theta _{\infty} ^{(2)} ))\in \Zint _{\geq 1}$,
there exists a logarithmic solution of Eq.(\ref{Heun02}) about $z=\infty $,
and there do not exist any non-zero solutions of Eq.(\ref{Heun01}) written as $w^{\theta _0} (w-1)^{\theta _1} (w-t)^{\theta _t} p(w)$ such that $p(w)$ is a polynomial,
then we have a contradiction.
Assume that there exists a non-zero solution $y(z) $ of Eq.(\ref{Heun01}) written as $y(w)=w^{\theta _0} (w-1)^{\theta _1} (w-t)^{\theta _t } p(w)$ such that $p(w)$ is a polynomial.
Then the exponent of $y(w)$ at $w=\infty $ is $\theta _{\infty} ^{(1)}+ \theta _{\infty} ^{(2)} -\deg p(w) -2$ and the function $y(w)$ can be expressed as $f^{\langle \infty \rangle } (w)$ in Eq.(\ref{eq:yinftyCD}) for the case $\theta _{\infty} ^{(1)}- \theta _{\infty} ^{(2)} \not \in \Zint$,
Hence $\langle [\gamma _{z} ,\gamma _{\infty }] , y \rangle  =0$ and we have $\langle [\gamma _{z} ,\gamma _p] , y \rangle ^{\gamma _{\infty }} =e^{2\pi \sqrt{-1} \eta }\langle [\gamma _{z} ,\gamma _p] , y \rangle$ $(p=0,1,t) $.
Since there exists a logarithmic solution about $z=\infty$, any two of $\langle [\gamma _{z} ,\gamma _0] , y \rangle $, $\langle [\gamma _{z} ,\gamma _1] , y \rangle $, $\langle [\gamma _{z} ,\gamma _t] , y \rangle $ are linearly dependent (see the proof of (iv)) and it follows from $\theta _p \not \in \Zint $ $(p=0,1,t )$ that 
there exists a solution $y(z)$ of Eq.(\ref{Heun02}) written as $z^{\tilde{\theta }_0} (z-1)^{\tilde{\theta }_1} (z-t)^{\tilde{\theta }_t} p(z)$ such that $p(z)$ is a polynomial.
Then we have $\deg p(z)=\eta -2$ or $\alpha +\beta -\eta -2$ and this contradicts $\eta  \not \in  \Zint$ and $\alpha +\beta - 2\eta \in \Zint$.
Hence if $\alpha +\beta -2\eta \in \Zint _{\geq 1}$, $\eta , \theta _0 , \theta _1 ,\theta _t \not \in \Zint$, there exists a non-zero solution of Eq.(\ref{Heun01}) which can be written as a product of $w^{\theta _0} (w-1) ^{\theta _1 }(w-t) ^{\theta _t } $ and a polynomial, then the singularity $z=\infty $ of Eq.(\ref{Heun02}) is apparent.
Therefore we have (viii).
\end{proof}
Theorem \ref{thm:nonlog-polynHeun} (i), (v), (viii) follows from Proposition \ref{prop:nonlog-polynHeun} (i), (v), (viii).

We show Theorem \ref{thm:nonlog-polynHeun} (ii).
Assume that there exists a non-zero solution of Eq.(\ref{Heun01}) which is written as $p(w)$ (resp. $(w-p)^{1- \epsilon ' _p} p(w)$) where $p(w)$ is a polynomial.
It follows from Proposition \ref{prop:monoredpolynsol} (ii) that $\deg p(w)=-\alpha '$ or $-\beta '$ (resp. $\deg p(w)=\epsilon '_p  -\alpha '-1$ or $\epsilon '_p  -\beta '-1$).
Thus $\alpha ' \in \Zint _{\leq 0}$ or $\beta ' \in \Zint _{\leq 0}$ (resp. $\epsilon '_p  -\alpha ' \in \Zint _{\geq 1}$ or $\epsilon '_p  -\beta ' \in \Zint _{\geq 1}$).
Therefore, if $\alpha' ,\beta ' \not \in \Zint$ (resp. $\epsilon '_p -\alpha ' , \epsilon '_p -\beta ' \not \in \Zint $) then there do not exist any  non-zero solutions of Eq.(\ref{Heun01}) written in the form $p(w)$ (resp. $(w-p)^{1- \epsilon ' _p} p(w)$) where $p(w)$ is a polynomial.
It follows from $\alpha' ,\beta ' \not \in \Zint$ that $\eta ' \not \in \Zint$ and  $\eta \not \in \Zint$.
If $\epsilon '_a \in \Zint $ and $\epsilon _b =\epsilon '_b -\eta '+1 \not \in \Zint $ (resp. $\epsilon _c \not \in \Zint $), then $\epsilon '_c -(\alpha ' +\beta '-\eta ')=- \epsilon ' _a- \epsilon '_b +\eta ' +1 \not \in \Zint $ (resp. $\epsilon '_b -(\alpha ' +\beta '-\eta ') \not \in \Zint $).
By combining with Proposition \ref{prop:nonlog-polynHeun} (ii), we have Theorem \ref{thm:nonlog-polynHeun} (ii).

We show Theorem \ref{thm:nonlog-polynHeun} (iii) and (iv).
It follows from $\alpha ,\beta \not \in \Zint$ that $\eta \not \in \Zint$.
If there exists a non-zero solution of Eq.(\ref{Heun02}) which is written as $p(z)$ (resp. $z^{1-\epsilon _0} (z-1)^{1-\epsilon _1} (z-t)^{1-\epsilon _t} p(z)$) where $p(z)$ is a polynomial, then $\deg p(z)=-\alpha $ or $-\beta $ (resp. $\deg p(z)=\alpha -2$ or $\beta -2$).
Hence if $\alpha, \beta \not \in \Zint$, then there do not exist any non-zero solutions of Eq.(\ref{Heun02}) written as a polynomial nor as a product of $z^{1-\epsilon _0} (z-1)^{1-\epsilon _1} (z-t)^{1-\epsilon _t} p(z)$ and a polynomial.
By combining with Proposition \ref{prop:nonlog-polynHeun} (iii), (iv), we have Theorem \ref{thm:nonlog-polynHeun} (iii) and (iv).

If $\alpha +\beta -\eta = \alpha '+\beta '-2\eta ' +1 \in \Zint$ and $\epsilon _p \not \in \Zint$, then $\epsilon ' _p -\eta ' = \epsilon _p -1 \not \in \Zint $ and $\epsilon ' _p -(\alpha ' +\beta '- \eta ')=\epsilon ' _p -\eta ' +( \alpha '+\beta '-2\eta ' ) \not \in \Zint $.
Hence $\epsilon '_p -\alpha ' , \epsilon '_p -\beta ' \not \in \Zint $ and there do not exist any  non-zero solutions of Eq.(\ref{Heun01}) written in the form $(w-p)^{1- \epsilon ' _p} p(w)$, where $p(w)$ is a polynomial.
Hence we have Theorem \ref{thm:nonlog-polynHeun} (vi) by combining with Proposition \ref{prop:nonlog-polynHeun} (vi).

If there exists a non-zero solution of Eq.(\ref{Heun02}), written as $(1/z)^{\eta }p(1/z)$ where $p(1/z) $ is a polynomial in $1/z$, then the exponent of the function $(1/z)^{\eta }p(1/z)$ is $-\eta -\deg_{1/z} p(1/z)$ and $-\eta -\deg_{1/z} p(1/z)= 0$ or $\epsilon _0$.
Hence if $\eta ,\epsilon '_0 \not \in \Zint$, then there do not exist any non-zero solutions of Eq.(\ref{Heun02}) written as $(1/z)^{\eta }p(1/z)$ where $p(1/z) $ is a polynomial in $1/z$.
By combining with Proposition \ref{prop:nonlog-polynHeun} (vii), we have Theorem \ref{thm:nonlog-polynHeun} (vii).
Thus Theorem \ref{thm:nonlog-polynHeun} is proved.

The following proposition concerning solutions of $D_{y_1}(\theta _0, \theta _1, \theta _t,\theta _{\infty}; \lambda ,\mu )$ and $D_{y_1}(\tilde{\theta }_0, \tilde{\theta }_1, \tilde{\theta }_t, \tilde{\theta }_{\infty} ; \tilde{\lambda } ,\tilde{\mu } )$ is proved similarly to Proposition \ref{prop:nonlog-polynHeun}.
\begin{prop} \label{prop:nonlog-polyn}
Let $a,b,c$ be elements of $\{0,1,t \}$ such that $a \neq b \neq c \neq a$.
Assume that $\lambda , \tilde{\lambda} \not \in \{ 0,1,t,\infty \}$.\\
(i) If $\theta _a \in \Zint _{\leq -1}$, $\kappa _2 \not \in \Zint$ and the singularity $w=a$ of the differential equation $D_{y_1}(\theta _0, \theta _1, \theta _t,\theta _{\infty}; \lambda ,\mu )$ in the variable $w$ is apparent, then there exists a non-zero solution $\tilde{y} (z)$ of the differential equation $D_{y_1}(\tilde{\theta }_0, \tilde{\theta }_1, \tilde{\theta }_t, \tilde{\theta }_{\infty} ; \tilde{\lambda } ,\tilde{\mu } )$ in the variable $z$ which can be written as $(z-a) ^{\tilde{\theta }_a}  h(z) $ where $h(z)$ is a polynomial of degree no more than $-\theta _a-1$.
Moreover if $\kappa _1 \not \in \Zint$, then $\deg _E h(z)= -\theta _a-1$.
\\
(ii) If $\theta _a \in \Zint _{\geq 0}$, $\kappa _2 \not \in \Zint$, the singularity $w=a$ of the differential equation $D_{y_1}(\theta _0, \theta _1, \theta _t,\theta _{\infty}; \lambda ,\mu )$ is apparent and there do not exist any non-zero solutions of $D_{y_1}(\theta _0, \theta _1, \theta _t,\theta _{\infty}; \lambda ,\mu )$ written in the form $(w-b)^{\alpha _b}(w-c)^{\alpha _c} p(w)$, where $p(w)$ is a polynomial and $(\alpha _b, \alpha _c)=(0,0)$, $(\theta _b, 0)$ or $(0,\theta _c )$, then there exists a non-zero solution of the differential equation $D_{y_1}(\tilde{\theta }_0, \tilde{\theta }_1, \tilde{\theta }_t, \tilde{\theta }_{\infty} ; \tilde{\lambda } ,\tilde{\mu } )$ which can be written as $(z-b) ^{\tilde{\theta }_b}(z-c) ^{\tilde{\theta }_c}  h(z) $, where $h(z)$ is a polynomial, and for the case $\kappa _1 \not \in \Zint$ we have $\deg h(z)= \theta _a$.\\
(iii) If $\tilde{\theta }_a \in \Zint _{\leq 0}$, $\kappa _{2 } \not \in \Zint$, there exists a non-zero solution of $D_{y_1}(\theta _0, \theta _1, \theta _t,\theta _{\infty}; \lambda ,\mu )$ which can be written as $(w-a) ^{\theta _a}h(w)$ where $h(w)$ is a polynomial and there do not exist any non-zero solutions of $D_{y_1}(\tilde{\theta }_0, \tilde{\theta }_1, \tilde{\theta }_t, \tilde{\theta }_{\infty} ; \tilde{\lambda } ,\tilde{\mu } )$ written as a polynomial in $z$, then the singularity $z=a$ of $D_{y_1}(\tilde{\theta }_0, \tilde{\theta }_1, \tilde{\theta }_t, \tilde{\theta }_{\infty} ; \tilde{\lambda } ,\tilde{\mu } )$ is apparent.\\
(iv) If $\tilde{\theta } _a \in \Zint _{\geq 1}$, $\kappa _{2 } , \theta _b, \theta _c \not \in \Zint$, there exists a non-zero solution of $D_{y_1}(\theta _0, \theta _1, \theta _t,\theta _{\infty}; \lambda ,\mu )$ which can be written as a product of $(w-b) ^{\theta _b}(w-c) ^{\theta _c}$ and a polynomial and there do not exist any solutions of $D_{y_1}(\tilde{\theta }_0, \tilde{\theta }_1, \tilde{\theta }_t, \tilde{\theta }_{\infty} ; \tilde{\lambda } ,\tilde{\mu } )$ written as a product of $z^{\tilde{\theta }_0}(z-1)^{\tilde{\theta }_1}(z-t)^{\tilde{\theta }_t}$ and a polynomial, then the singularity $z=a$ of $D_{y_1}(\tilde{\theta }_0, \tilde{\theta }_1, \tilde{\theta }_t, \tilde{\theta }_{\infty} ; \tilde{\lambda } ,\tilde{\mu } )$ is apparent.\\
(v) If $\theta _{\infty } \in \Zint _{\leq 0}$, $\kappa _2 \not \in \Zint$ and the singularity $w=\infty $ of $D_{y_1}(\theta _0, \theta _1, \theta _t,\theta _{\infty}; \lambda ,\mu )$ is apparent, then there exists a non-zero solution of $D_{y_1}(\tilde{\theta }_0, \tilde{\theta }_1, \tilde{\theta }_t, \tilde{\theta }_{\infty} ; \tilde{\lambda } ,\tilde{\mu } )$ which can be written as a polynomial of degree $-\theta _{\infty }$.\\
(vi) If $\theta _{\infty } \in \Zint _{\geq 1}$, $\kappa _2 \not \in \Zint$, the singularity $w=\infty $ of $D_{y_1}(\theta _0, \theta _1, \theta _t,\theta _{\infty}; \lambda ,\mu )$ is apparent and there do not exist any non-zero solutions of $D_{y_1}(\theta _0, \theta _1, \theta _t,\theta _{\infty}; \lambda ,\mu )$ written as $w^{\alpha _0} (w-1)^{\alpha _1} (w-t)^{\alpha _t} p(w)$ such that $p(w)$ is a polynomial and $(\alpha _0, \alpha _1, \alpha _t)=(\theta _0,0,0),$ $(0,\theta _1,0)$ or $(0,0,\theta _t)$, then there exists a non-zero solution of $D_{y_1}(\tilde{\theta }_0, \tilde{\theta }_1, \tilde{\theta }_t, \tilde{\theta }_{\infty} ; \tilde{\lambda } ,\tilde{\mu } )$ which can be written as $z ^{\tilde{\theta }_0 } (z-1) ^{\tilde{\theta }_1} (z-t) ^{\tilde{\theta }_t}  h(z)$, where $h(z)$ is a polynomial of degree $\theta _{\infty }-1$.\\
(vii) If $\kappa _1 \in \Zint _{\leq 0}$, $\kappa _2 \not \in \Zint$, there exists a non-zero solution of $D_{y_1}(\theta _0, \theta _1, \theta _t,\theta _{\infty}; \lambda ,\mu )$ written as a polynomial and there do not exist any non-zero solutions of $D_{y_1}(\tilde{\theta }_0, \tilde{\theta }_1, \tilde{\theta }_t, \tilde{\theta }_{\infty} ; \tilde{\lambda } ,\tilde{\mu } )$ written in the form $(1/z)^{-\kappa _2 +1}p(1/z)$, where $p(1/z) $ is a polynomial in $1/z$, then the singularity $z=\infty $ of $D_{y_1}(\tilde{\theta }_0, \tilde{\theta }_1, \tilde{\theta }_t, \tilde{\theta }_{\infty} ; \tilde{\lambda } ,\tilde{\mu } )$ is apparent.\\
(viii) If $\kappa _1 \in \Zint _{\geq 1}$, $\kappa _2 , \theta _0 ,\theta _1 ,\theta _t \not \in \Zint$, there exists a non-zero solution of $D_{y_1}(\theta _0, \theta _1, \theta _t,\theta _{\infty}; \lambda ,\mu )$ written as a product of $w^{\theta _0} (w-1) ^{\theta _1 }(w-t) ^{\theta _t } $ and a polynomial, then the singularity $z=\infty $ of $D_{y_1}(\tilde{\theta }_0, \tilde{\theta }_1, \tilde{\theta }_t, \tilde{\theta }_{\infty} ; \tilde{\lambda } ,\tilde{\mu } )$ is apparent.
\end{prop}
Theorem \ref{thm:nonlog-polyn} follows from Proposition \ref{prop:nonlog-polyn}.

\end{document}